\documentclass[letterpaper,11pt]{amsart}
\usepackage{amsmath,amssymb,amsthm,pinlabel,tikz,hyperref,mathrsfs,color}
\usepackage{verbatim}

\usepackage[hang,small,bf]{caption}
\usepackage[subrefformat=parens]{subcaption}
\usepackage{here}
\usepackage{bm}
\usepackage{graphicx, color}

\newcommand{\nc}{\newcommand}
\nc{\dmo}{\DeclareMathOperator}
\dmo{\ra}{\rightarrow}
\dmo{\N}{\mathbb{N}}
\dmo{\Z}{\mathbb{Z}}
\dmo{\Q}{\mathbb{Q}}
\dmo{\R}{\mathbb{R}}
\dmo{\C}{\mathcal{C}}
\dmo{\AC}{\mathcal{AC}}
\dmo{\Mod}{Mod}
\dmo{\PMod}{PMod}
\dmo{\B}{B}
\dmo{\PB}{PB}
\dmo{\I}{\mathcal{I}}
\dmo{\el}{\ell_{\C}}
\dmo{\NN}{\mathcal{N}}
\dmo{\rk}{rk}

\usetikzlibrary{decorations.markings}
\tikzset{->-/.style={decoration={
  markings,
  mark=at position #1 with {\arrow[scale=3]{>}}},postaction={decorate}}}
  
\nc{\nt}{\newtheorem}

\nt{theorem}{Theorem}
\nt{lemma}{Lemma}

\newtheorem{thm}{{\bf Theorem}}[section]

\newtheorem{cor}[thm]{{\bf Corollary}}

\newtheorem{remark}[thm]{Remark}
\newtheorem{ex}[thm]{Example}

\newtheorem{ques}[thm]{Question}

\numberwithin{equation}{section}
\nt*{theorem_nonumber}{Theorem}

\begin{document}

\title[Braids and periodic solutions of the $2n$-body problem]
{Braids, metallic ratios and periodic solutions of the $2n$-body problem}

\author[Y. Kajihara]{%
Yuika Kajihara}
\address{%
Department of Applied Mathematics and Physics, 
Graduate School of Informatics, Kyoto University, 
Yoshida-Honmachi, Sakyo-ku, Kyoto 606-8501, Japan
}
\email{%
kajihara.yuika.75e@st.kyoto-u.ac.jp
}

\author[E. Kin]{%
    Eiko Kin
}
\address{%
      Center for Education in Liberal Arts and Sciences, Osaka University, Toyonaka, Osaka 560-0043, Japan
}
\email{%
        kin.eiko.celas@osaka-u.ac.jp
}

\author[M. Shibayama]{%
    Mitsuru Shibayama
}
\address{%
Department of Applied Mathematics and Physics, 
Graduate School of Informatics, Kyoto University, 
Yoshida-Honmachi, Sakyo-ku, Kyoto 606-8501, Japan
}
\email{%
        shibayama@amp.i.kyoto-u.ac.jp}

\subjclass[2020]{%
}


\date{
March 30, 2022
}

\thanks{%
Y. K. was supported by JSPS KAKENHI Grant Number 
JP20J21214. \\
E. K. was supported by JSPS KAKENHI Grant Number 
JP18K03299 and JP21K03247. \\
M. S. was supported by JSPS KAKENHI Grant Number 
JP18K03366. \\
} 
	
\begin{abstract} 
Periodic solutions of the planar $N$-body problem determine braids  through the trajectory of $N$ bodies. 
Braid types can be used to classify periodic solutions.  
According to the Nielsen-Thurston classification of surface automorphisms,   
braids fall into three types: periodic, reducible and pseudo-Anosov. 
To a braid of pseudo-Anosov type, there is an associated stretch factor greater than $1$, 
and this is a conjugacy invariant of braids. 
In 2006, the third author  discovered a family of multiple choreographic solutions of the planar $2n$-body problem. 
We prove that braids obtained from the solutions in the family are of pseudo-Anosov type, 
and their stretch factors  are expressed in metallic ratios. 
New numerical periodic solutions of the planar $2n$-body problem are also provided. 
\end{abstract}
\maketitle

\section{Introduction}

Consider the motion of $m$ points in the plane ${\Bbb R}^2$ 
$${\bm x}(t)= (x_1(t), \dots, x_m(t)),$$ 
where $x_i(t) \in {\Bbb R}^2$ is the position of the $i$th point at $t \in {\Bbb R}$. 
Let $Q_m(t)= \{x_1(t), \dots, x_m(t)\}$. 
We assume the following. 
\begin{itemize}
\item 
 ${\bm x}(t)$ is collision-free, i.e., 
for any $t \in {\Bbb R}$, 
$x_i (t) \ne x_j(t)$ if $i \ne j$.

\item 
There exists  $t_0 >0$ such that 
$$Q_m(t+t_0)=Q_m(t). $$ 
\end{itemize}
Then we have a (geometric) braid 
$$b({\bm x}(t), [0,t_0])= \bigcup_{t \in [0, t_0]} \bigl\{(x_1(t), t), \dots, (x_m(t), t) \bigr\} \subset {\Bbb R}^2 \times [0,t_0]$$
with base points $Q_m(0)(=Q_m(t_0))$. 
The actual location of base points   is irrelevant for the study of braids. 
To remove the data of the location, 
we consider its braid type $\big\langle  b({\bm x}(t), [0,t_0])  \big\rangle$ instead of the braid
(See Section \ref{subsection_geometric-braids} for the definition of braid types).

We investigate  periodic solutions of the planar $N$-body problem 
given by the following ODEs. 
\begin{equation}
\label{equation_ODE}
m_i\ddot{x}_i= - \sum_{j \ne i} m_im_j \frac{x_i - x_j}{ |x_i - x_j|^3}, \hspace{0.5cm} 
x_i \in {\Bbb R}^2, \ m_i > 0\ (i= 1, \dots, N). 
\end{equation}
Suppose that 
${\bm x}(t)= (x_1(t), \dots, x_N(t))$ is a periodic solution with period $T$ of  (\ref{equation_ODE}). 
The solution ${\bm x}(t)$  determines a (pure) braid $b({\bm x}(t), [0,T])$ 
and its braid type $\big\langle b({\bm x}(t), [0,T] \big\rangle$.  
Braid types can be used to classify periodic solutions of the planar $N$-body problem.

\begin{ques}[Montgomery \cite{Montgomery21}, (cf. Moore \cite{Moore93})]
\label{question_Montgomery}
For any pure braid $b $ with $N$ strands, 
is there a periodic solution of the planar $N$-body problem 
whose  braid type is equal to $\langle b \rangle$? 
\end{ques}
Question \ref{question_Montgomery} is wide open for every $N > 3$. 
In the case of $N=3$ Question \ref{question_Montgomery} is true 
by work of Moeckel-Montgomery \cite{MoeckelMontgomery15}. 
For other studies on braids obtained from periodic solutions, 
see a pioneer work by Moore \cite{Moore93}. 
See also \cite{FG21,JamesKatrina13,Montgomery98}. 


\begin{remark} 
\label{remark_Montgomery}
We consider the following Newton equations
\begin{equation}
\label{equation_alpha}
m_i\ddot{x}_i= - \sum_{j \ne i} m_im_j \frac{x_i - x_j}{ |x_i - x_j|^{\alpha+1}}, \hspace{0.5cm} 
x_i \in {\Bbb R}^2, \ m_i > 0\ (i= 1, \dots, N), 
\end{equation}
where $\alpha \ge1$. 
The case $\alpha=2$ corresponds to  (\ref{equation_ODE})  
describing the motion of $n$ bodies under the influence of the gravitation. 
One can ask the same question as Question \ref{question_Montgomery} 
for the planar $N$-body problem given by  (\ref{equation_alpha}). 
It is known by Montgomery \cite{Montgomery98} that 
Question \ref{question_Montgomery} is true  for any ``tied" braid type when $\alpha \ge 3$ (i.e., under the assumption that  the  force is strong). 
\end{remark}

According to the Nielsen-Thurston classification of surface automorphisms \cite{FLP}, 
braids fall into three types: periodic, reducible and pseudo-Anosov. 
(See Section \ref{subsection_Nielsen-Thurston}.) 
To a braid $b$  of  pseudo-Anosov type, there is an associated stretch factor $\lambda(b)>1$,  
and this is a conjugacy invariant of pseudo-Anosov braids. 
Since the Nielsen-Thurston type is also a conjugacy invariant, 
one can define the stretch factor $\lambda(\langle b \rangle):= \lambda(b)$ 
for the pseudo-Anosov braid type $\langle b \rangle$ of $b$. 
See (\ref{equation_braid-type}) in Section \ref{subsection_Nielsen-Thurston}.

The stretch factor  tells us a dynamical complexity of pseudo-Anosov braids. 
In this paper we ask the following question related to Question \ref{question_Montgomery}.

\begin{ques}
\label{question_dilatation}
Let $b$ be a pure braid with $N$ strands. 
Suppose that $b$ is of pseudo-Anosov type. 
Is there a periodic solution of the planar $N$-body problem 
whose  braid type is equal to $\langle b \rangle$? 
\end{ques}


The stretch factor of each pseudo-Anosov braid with $3$ strands is a quadratic irrational (Section \ref{subsection_3braids}). 
This is not necessarily true for pseudo-Anosov braids with more than $3$ strands. 
Moore \cite{Moore93} and Chenciner-Montgomery \cite{ChencinerMontgomery00} found 
a simple choreographic solution to the $3$-body problem 
such that the three bodies chase one another along a figure-$8$ curve. 
The braid type of the solution is pseudo-Anosov  and 
its stretch factor is the $6$th power of the  $1$st metallic ratio $\mathfrak{s}_1$ 
(Example \ref{ex_figure-eight}), 
 i.e., golden ratio, 
where 
the $k$th metallic ratio  $\mathfrak{s}_k$ $(k \in {\Bbb N})$  
is given by 
$$\mathfrak{s}_k=\frac{1}{2} (k+ \sqrt{k^2+4}) = 
k+\cfrac{1}{k +\cfrac{1}{k +\cfrac{1}{k + \ddots}}} \quad$$
\smallskip

The study of braid types of the periodic solutions has been relatively less investigated. 
We hope that 
the following result sheds some light on 
Question \ref{question_dilatation}. 
Let $ \lfloor  \cdot \rfloor$ be the floor function.

\begin{thm}
\label{thm_stretch-factor}
For $n \ge 2$ and $p \in \{1, \dots, \lfloor  \frac{n}{2}\rfloor \}$, 
there exists a periodic solution ${\bm x}_{n,p}(t)$ of the planar $2n$-body problem 
with equal masses
whose braid type $X_{n,p}$ is pseudo-Anosov with the stretch factor 
$(\mathfrak{s}_{2p})^{\frac{2n}{d}}$, 
where $d= \gcd (n,p)$. 
\end{thm}
A representative of the braid type $X_ {n,p}$ in Theorem \ref{thm_stretch-factor} 
 is the $(\frac{n}{d})$th power $(\beta_{n,p})^{\frac{n}{d}}$ 
 of the $2n$-braid $\beta_{n,p}$ 
introduced in Section \ref{section_proof-of-theorem}. 
In 2006, the third author proved the existence of a family of multiple choreographic solutions ${\bm x}_{n,p}(t)$ 
of the planer $2n$-body problem  with equal masses \cite{Shibayama06}. 
Some of the solutions in the family had already found by Chen \cite{Chen01,Chen03-2} and Ferraio-Terracini \cite{FerrarioTerracini04}. 
The orbit of the periodic solution ${\bm x}_{n,p}(t)$ 
consists of $2d$ closed curves, 
each of which is the trajectory of $\frac{n}{d}$ bodies. 
The braid types $X_{n,p}(t)$ in Theorem \ref{thm_stretch-factor} are realized by  ${\bm x}_{n,p}(t)$ given in  \cite{Shibayama06}. 
More precisely,  
for $n \ge 2$ and  $p \in \{1, \dots, \lfloor  \frac{n}{2}\rfloor \}$, 
there exists 
a periodic solution 
$${\bm x}_{n,p}(t)= (x_1(t), \dots, x_{2n}(t))$$ 
with period $T>0$ of the planar $2n$-body problem such that 
$$x_i(t+  (\tfrac{d}{n})T ) = x_{\sigma_{n,p}(i)}(t) \hspace{2mm} \mbox{for}\hspace{2mm} i= 1, \dots, 2n, $$
where $\sigma_{n,p}= (1,3, \dots, 2n-1)^p (2,4, \dots, 2n)^{-p} \in \mathfrak{S}_{2n}$ is a permutation 
of  $2n$ elements.  
Thus,  the braid 
$y_{n,p}:= b({\bm x}_{n,p}(t), [0, (\frac{d}{n})T ])$ and the braid type $Y_{n,p}:= \langle y_{n,p} \rangle$ 
are obtained from the solution ${\bm x}_{n,p}(t)$, and 
the $(\frac{n}{d})$th power $(y_{n,p})^{\frac{n}{d}}$ represents the braid type $X_{n,p}$. 
See Figure \ref{fig_orbit} for periodic solutions ${\bm x}_{n,p}(t)$ for $0 \le t \le (\frac{d}{n})T$. 
Theorem \ref{thm_stretch-factor} follows from the following (see Remark \ref{remark_power}).

\begin{figure}[b]
\begin{center}

\includegraphics[width=1.8cm]{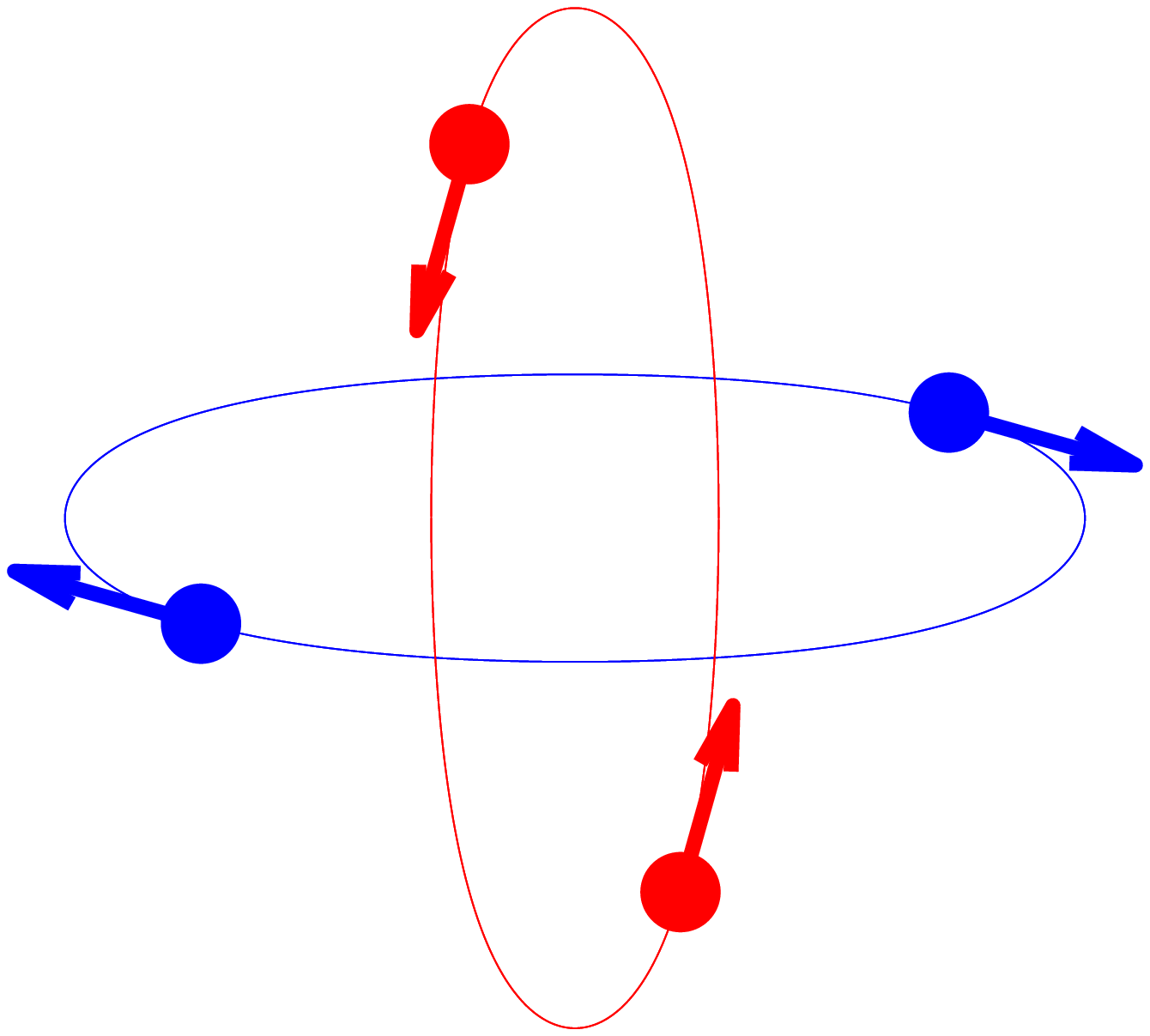}\hspace{0.5cm}
\includegraphics[width=1.8cm]{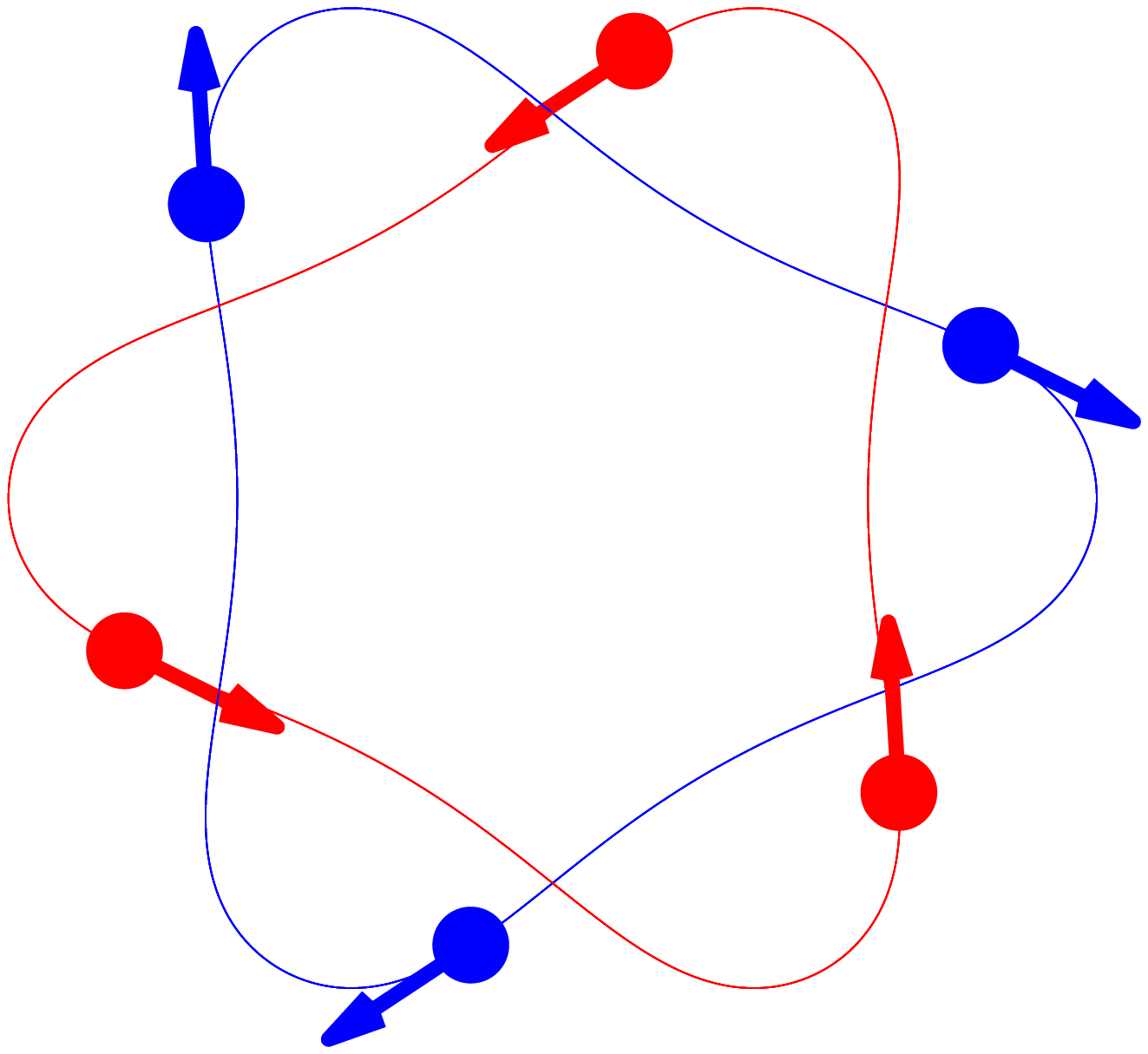}\hspace{0.5cm}
\includegraphics[width=1.8cm]{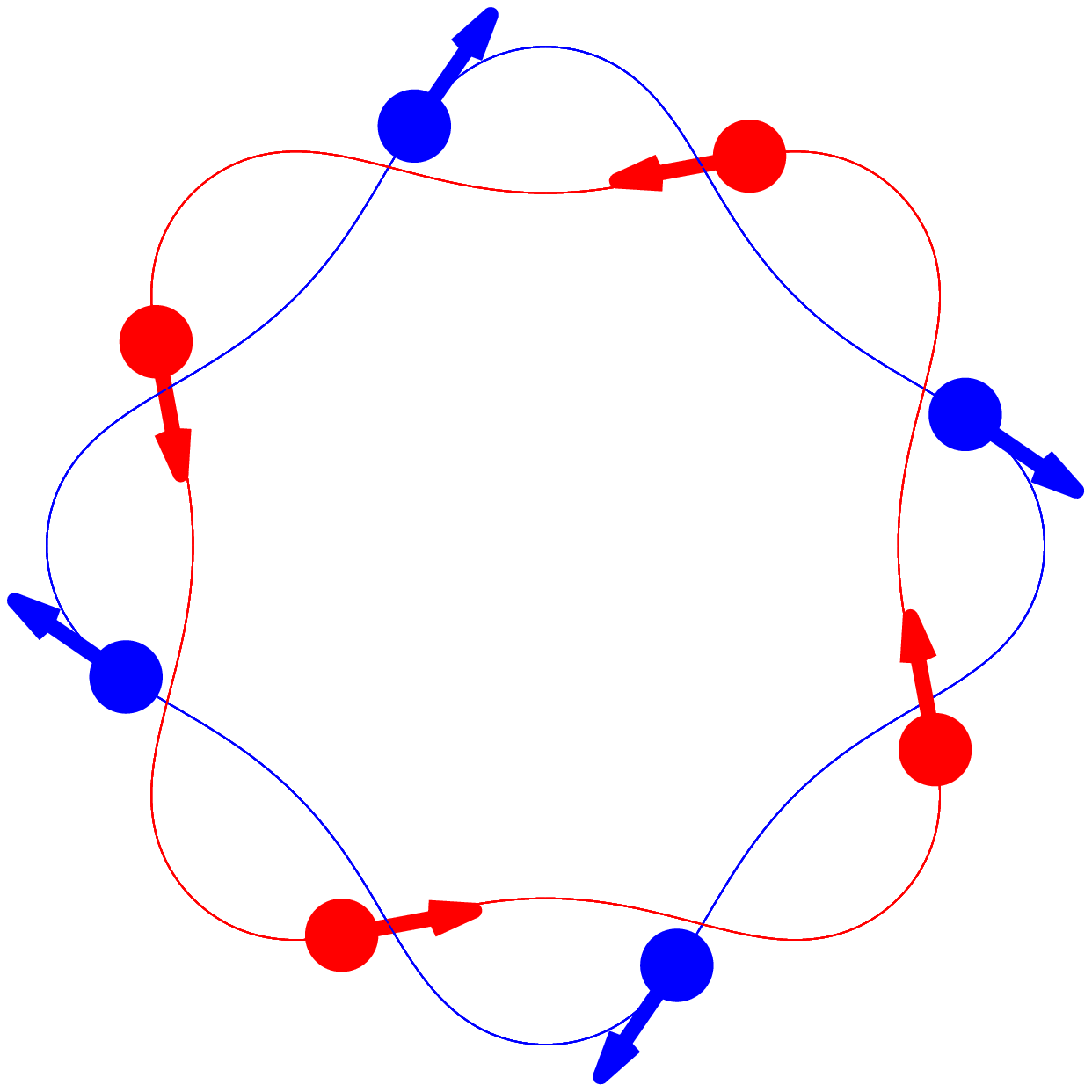}\hspace{0.5cm}
\includegraphics[width=1.8cm]{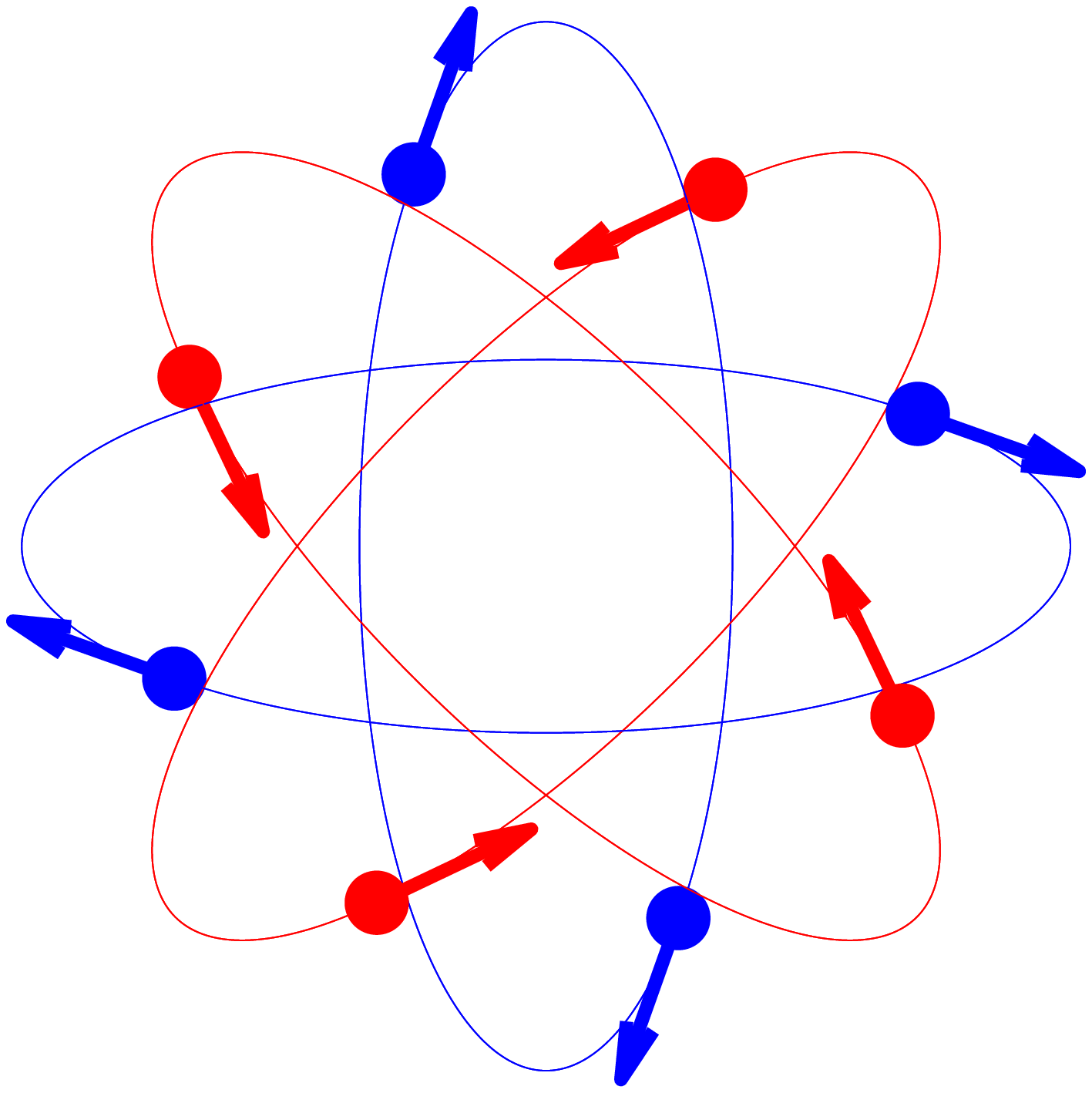}\hspace{0.5cm}
$t=\frac{dT}{n}$

\includegraphics[width=1.8cm]{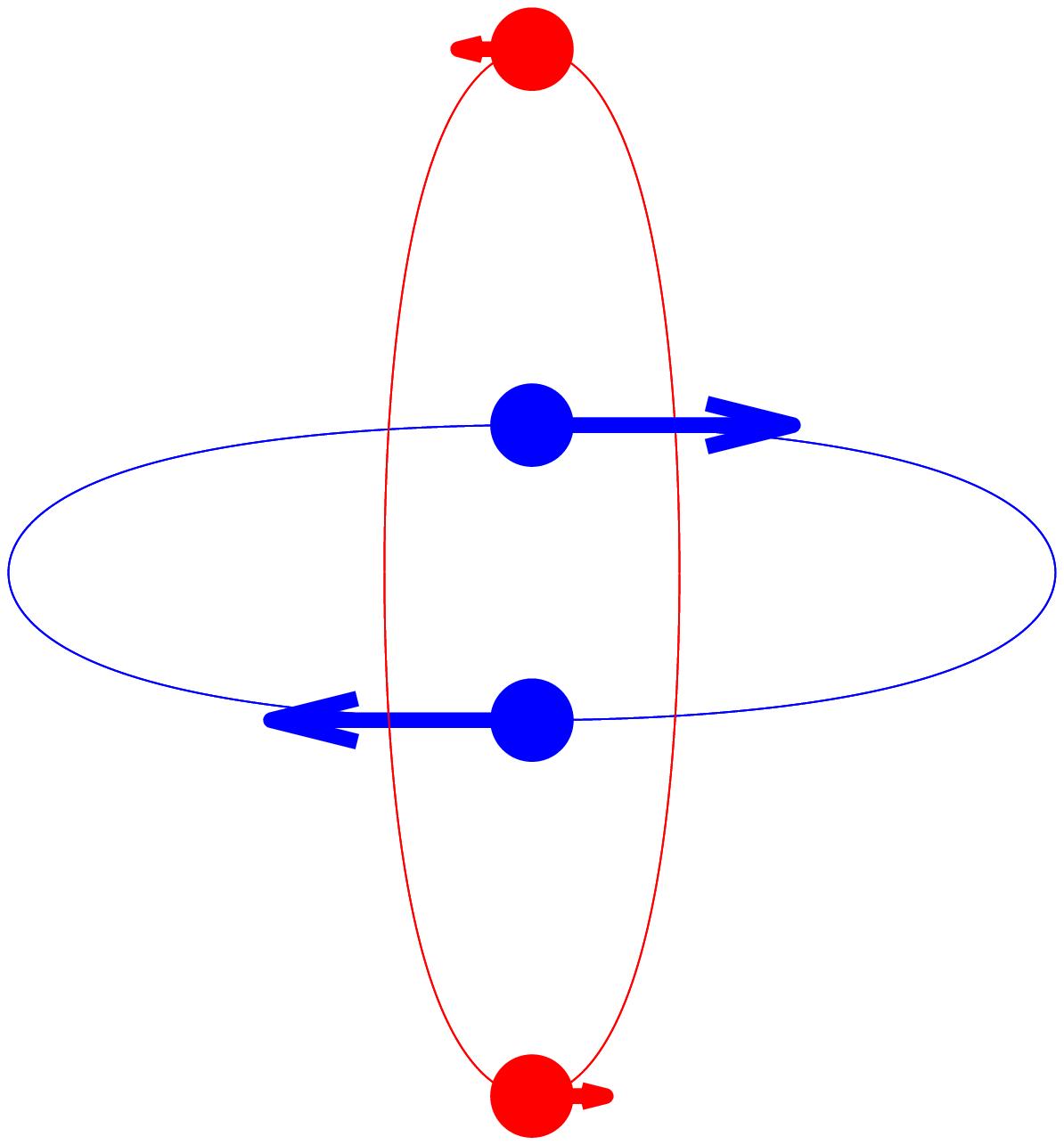}\hspace{0.5cm}
\includegraphics[width=1.8cm]{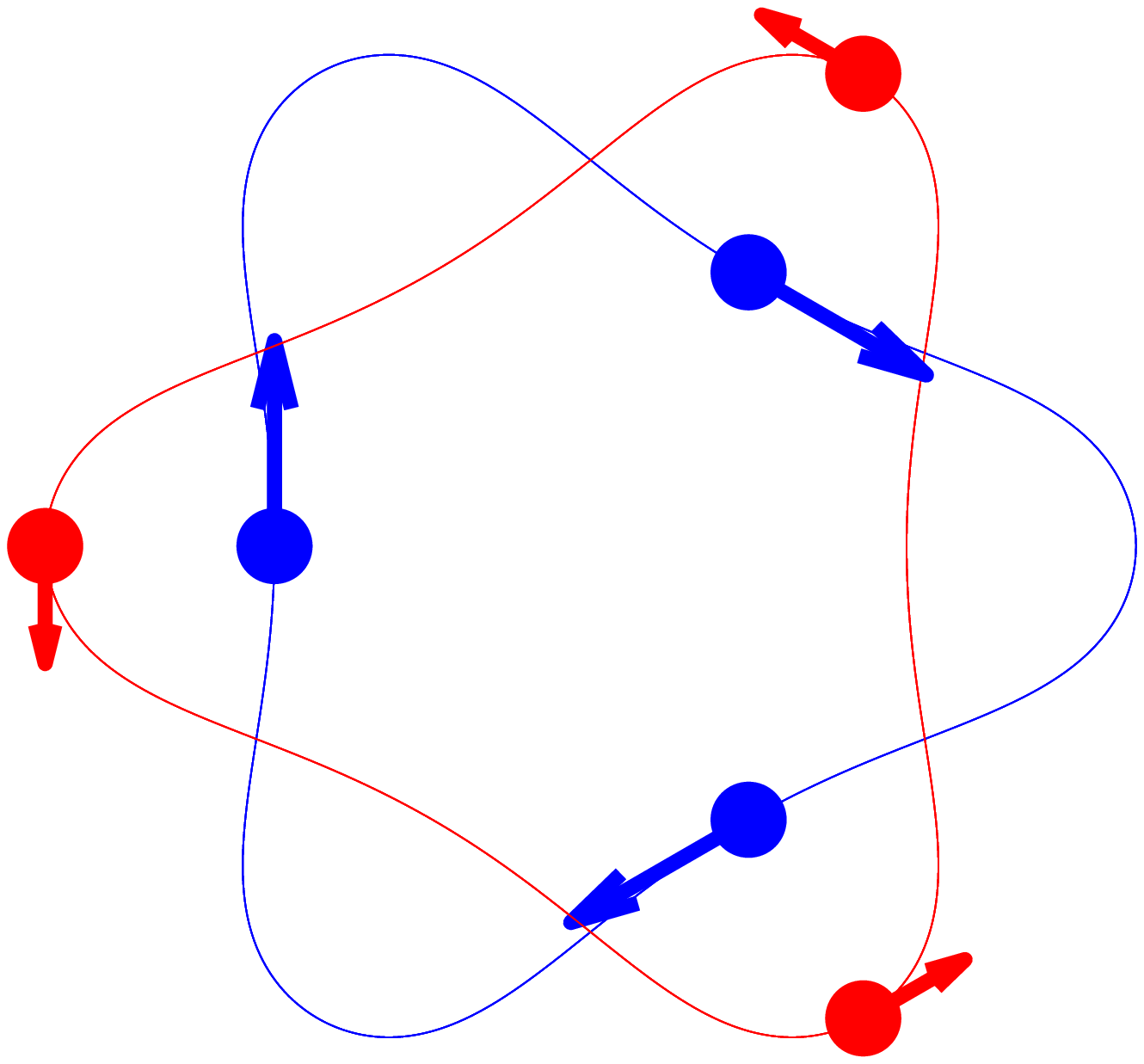}\hspace{0.5cm}
\includegraphics[width=1.8cm]{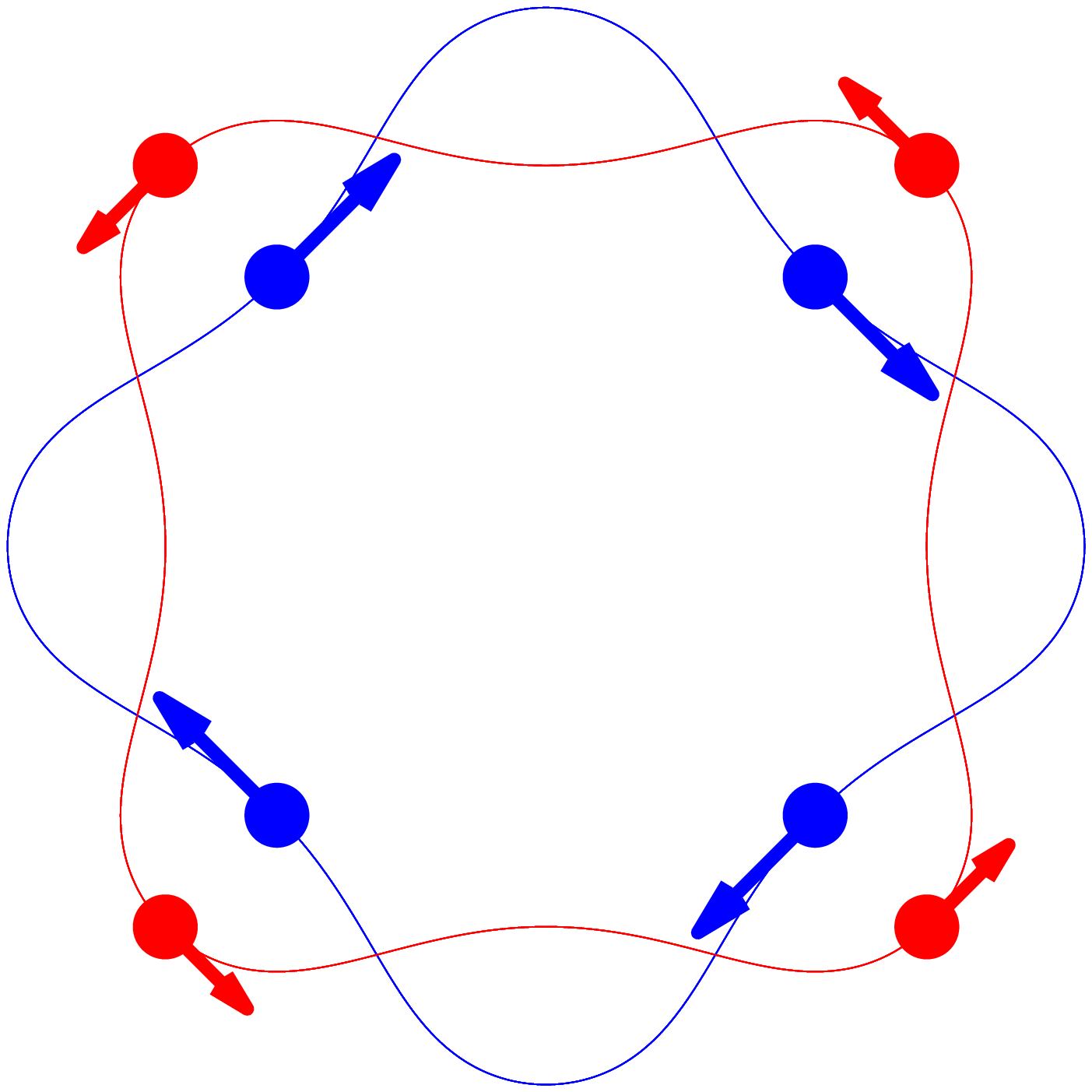}\hspace{0.5cm}
\includegraphics[width=1.8cm]{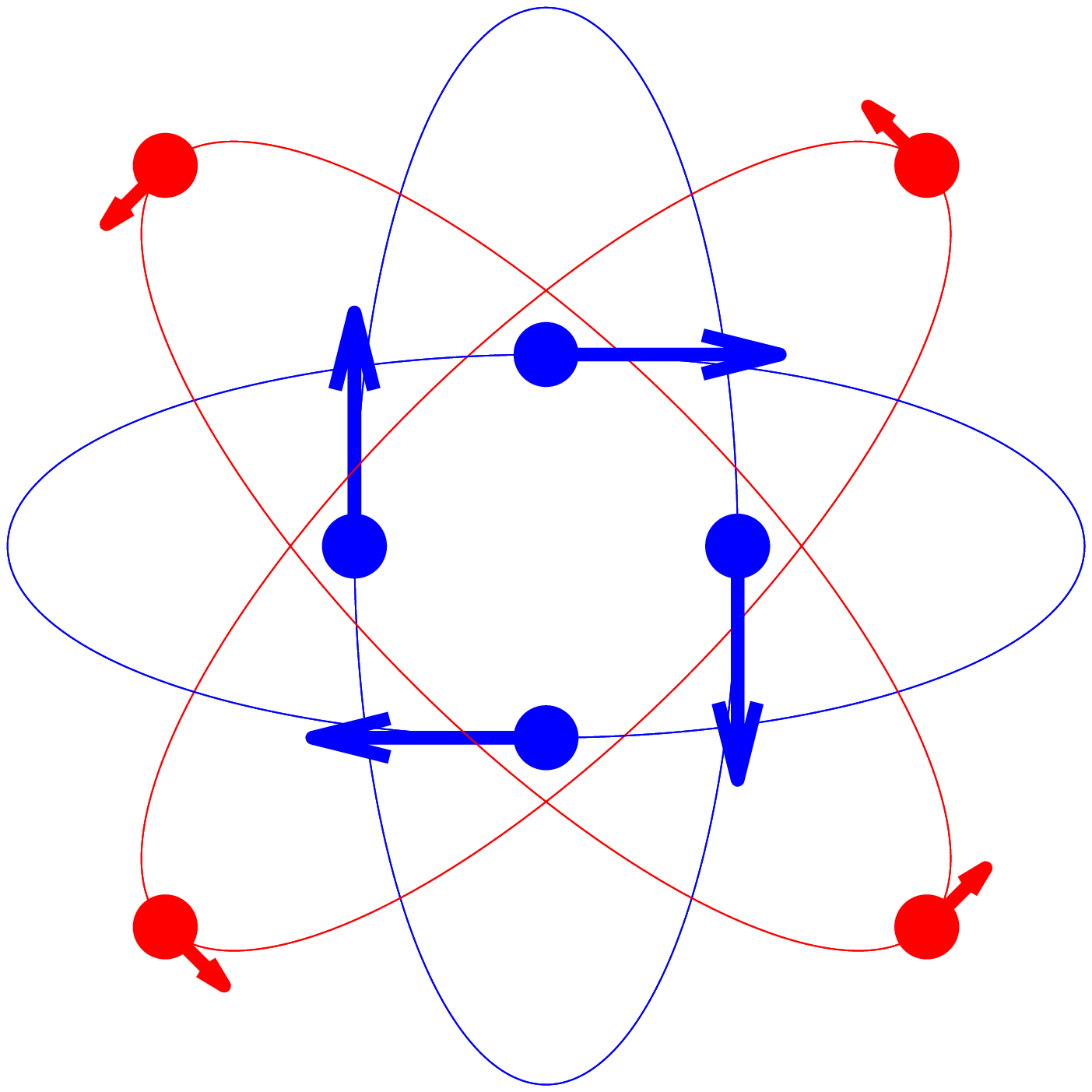}\hspace{0.5cm}
$t=\frac{3dT}{4n}$

\includegraphics[width=1.8cm]{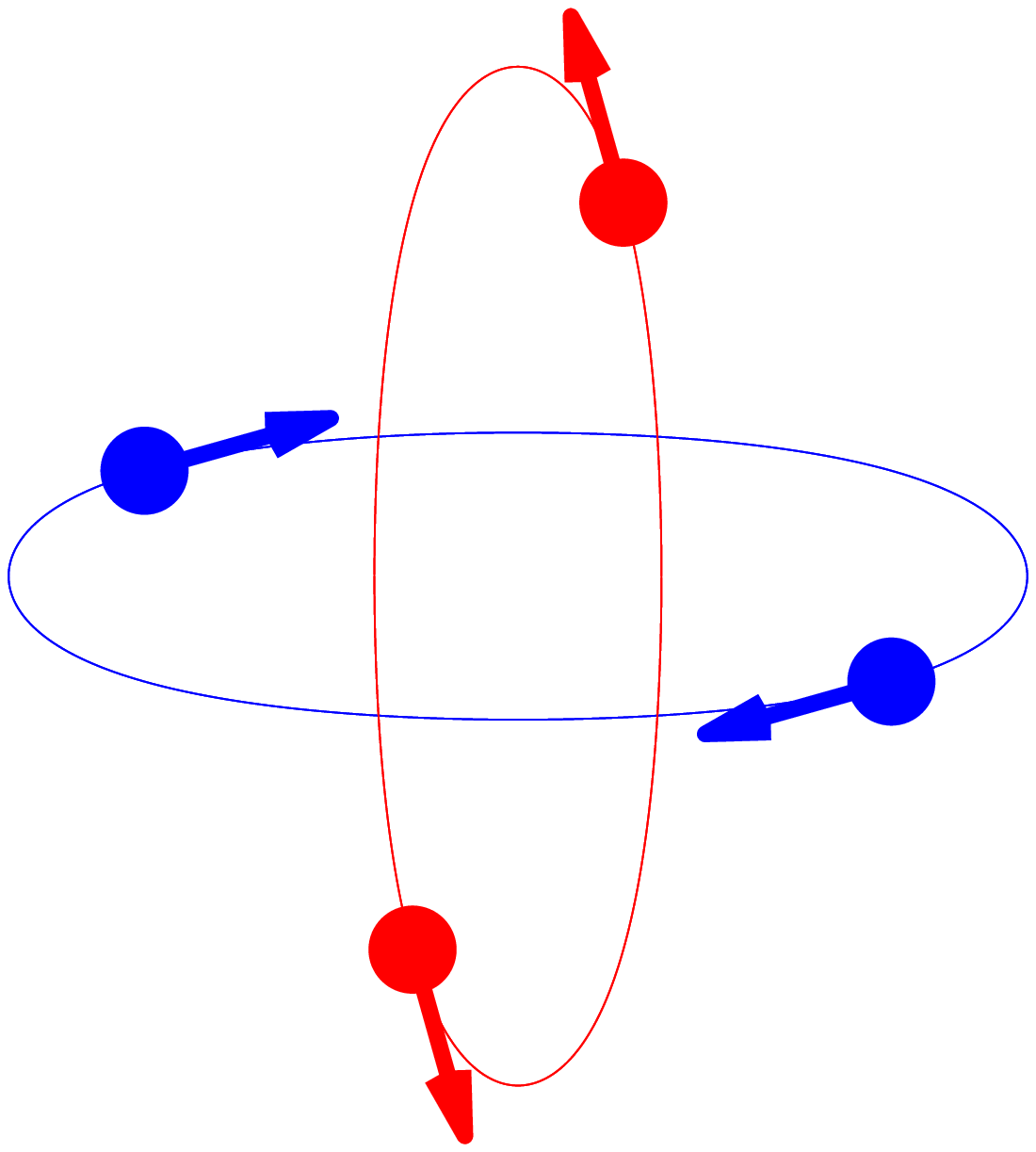}\hspace{0.5cm}
\includegraphics[width=1.8cm]{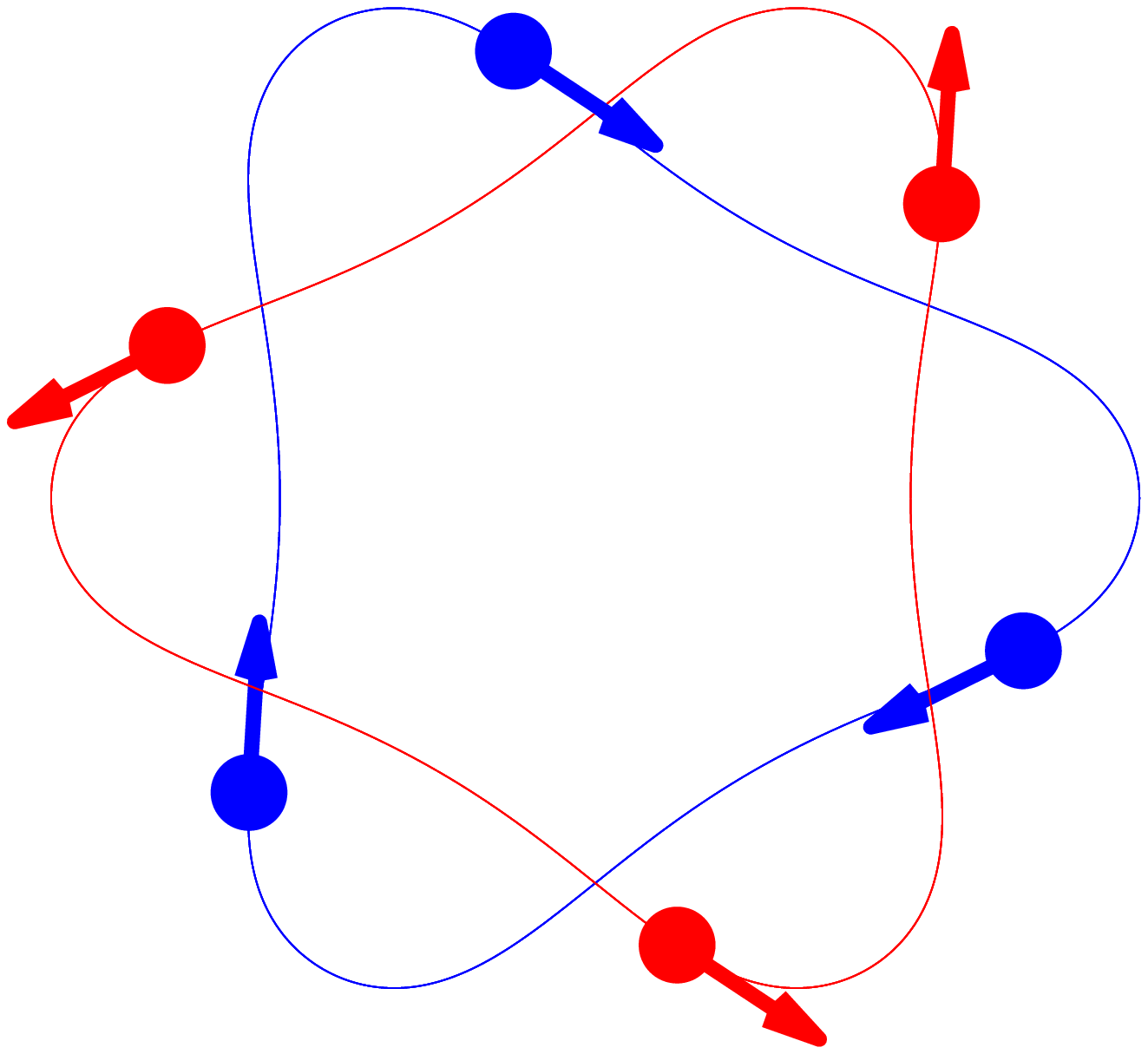}\hspace{0.5cm}
\includegraphics[width=1.8cm]{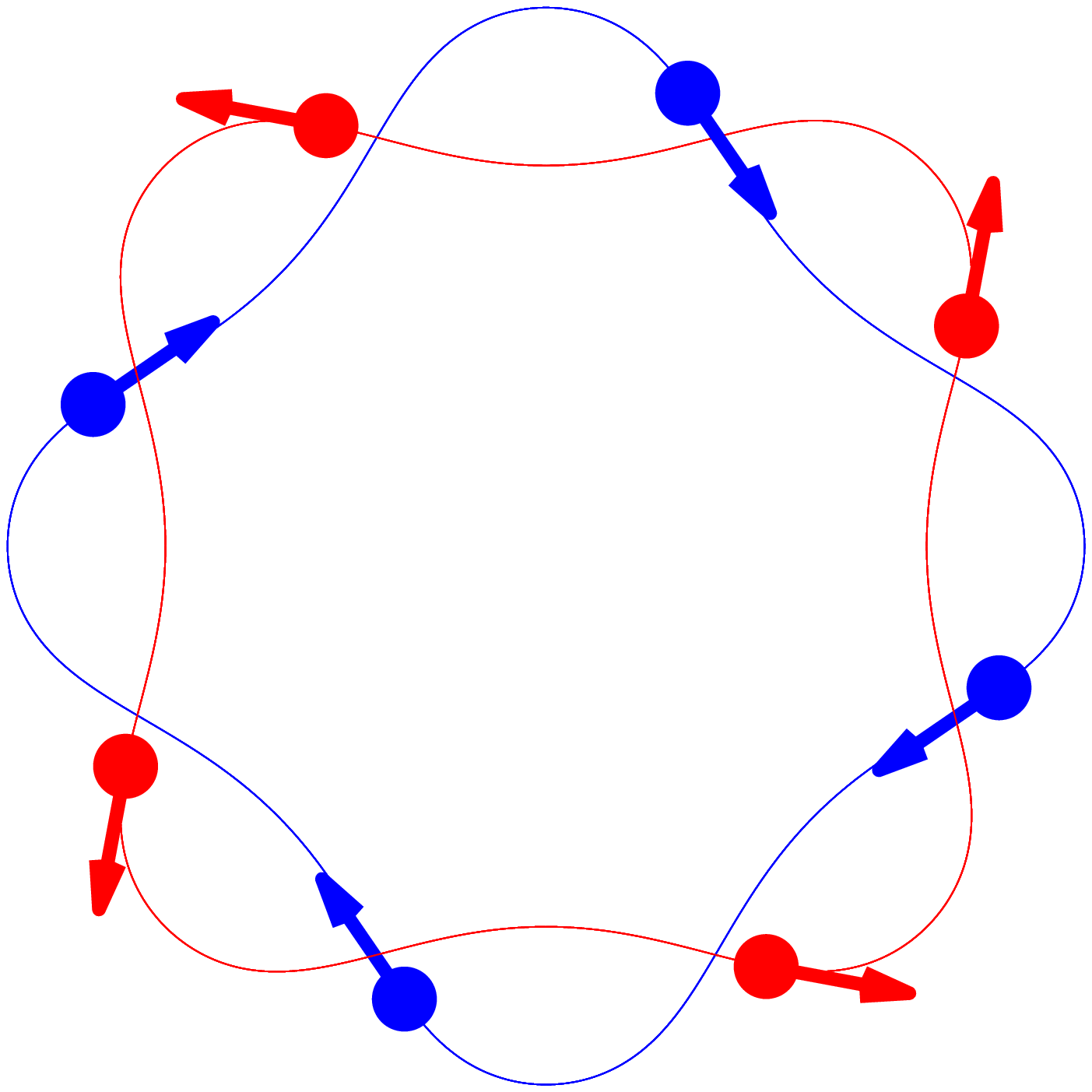}\hspace{0.5cm}
\includegraphics[width=1.8cm]{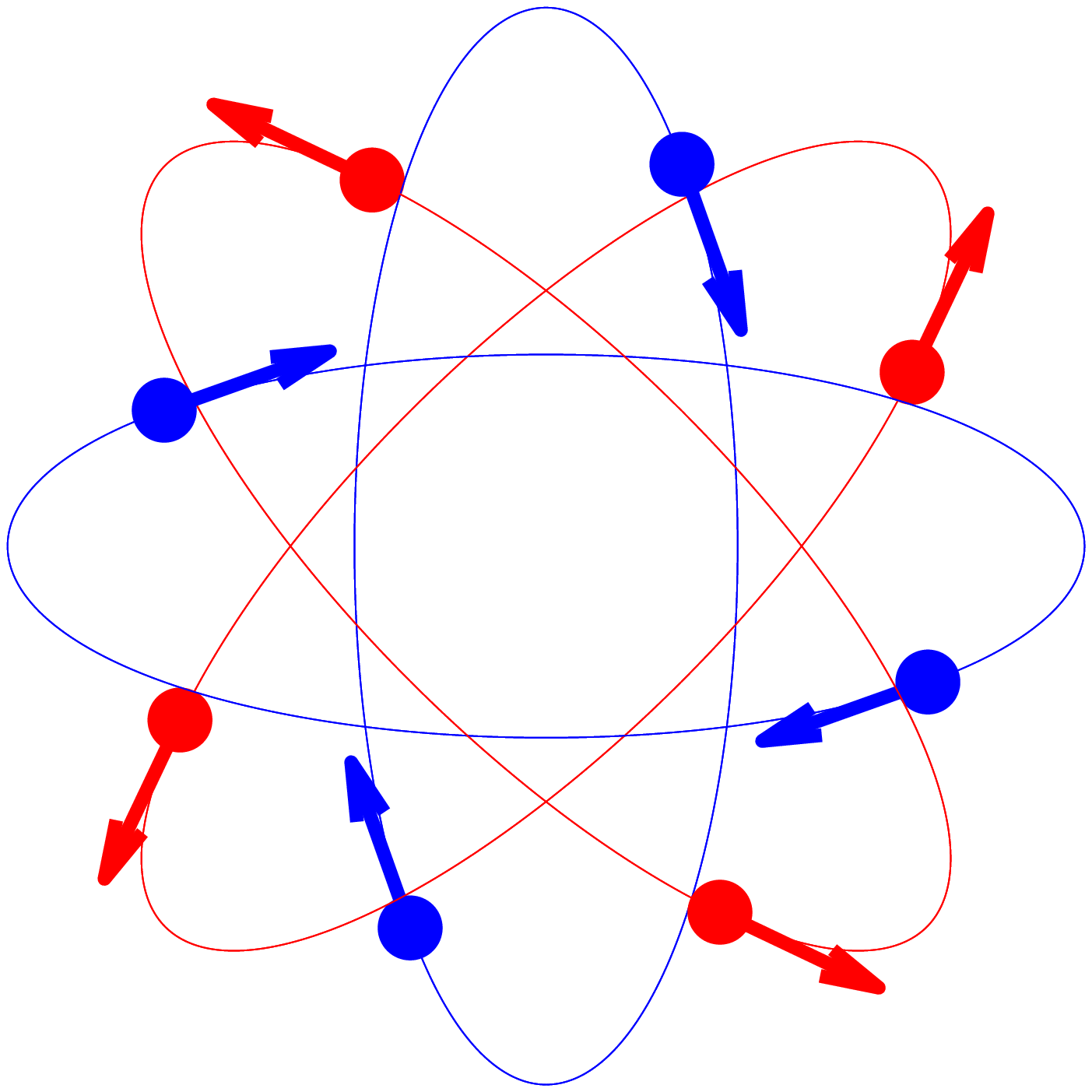}\hspace{0.5cm}
$t=\frac{dT}{2n}$

\includegraphics[width=1.8cm]{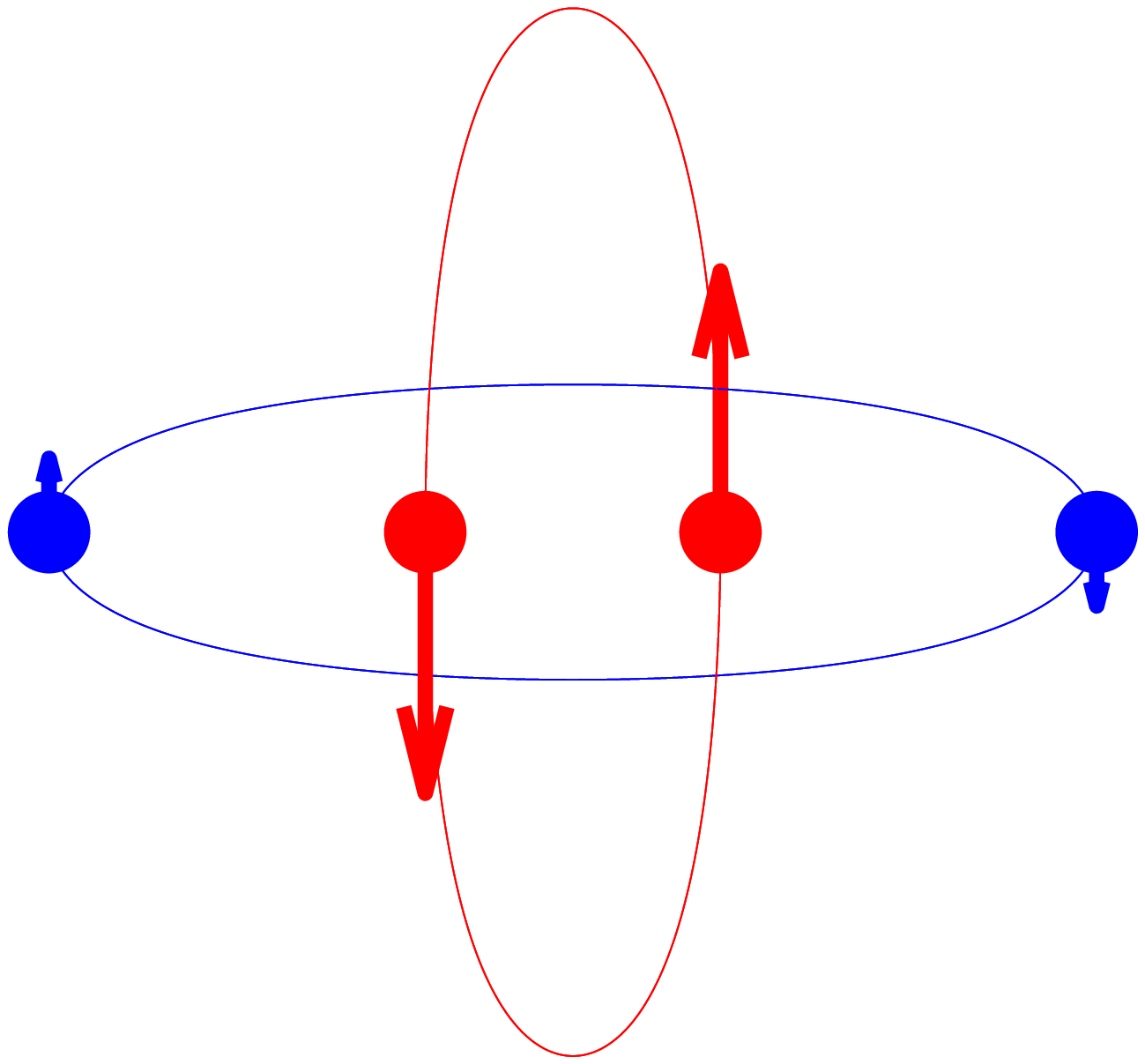}\hspace{0.5cm}
\includegraphics[width=1.8cm]{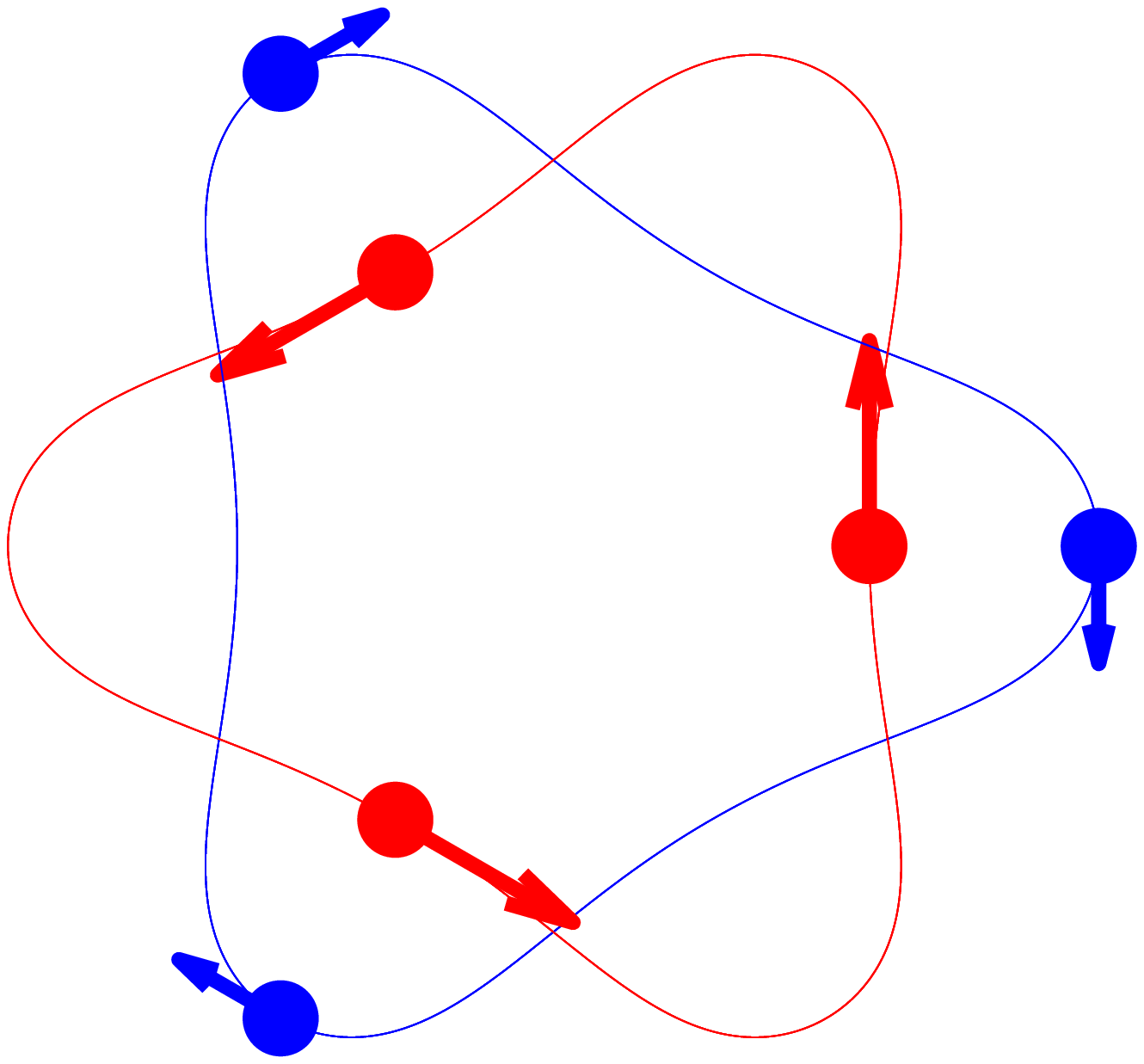}\hspace{0.5cm}
\includegraphics[width=1.8cm]{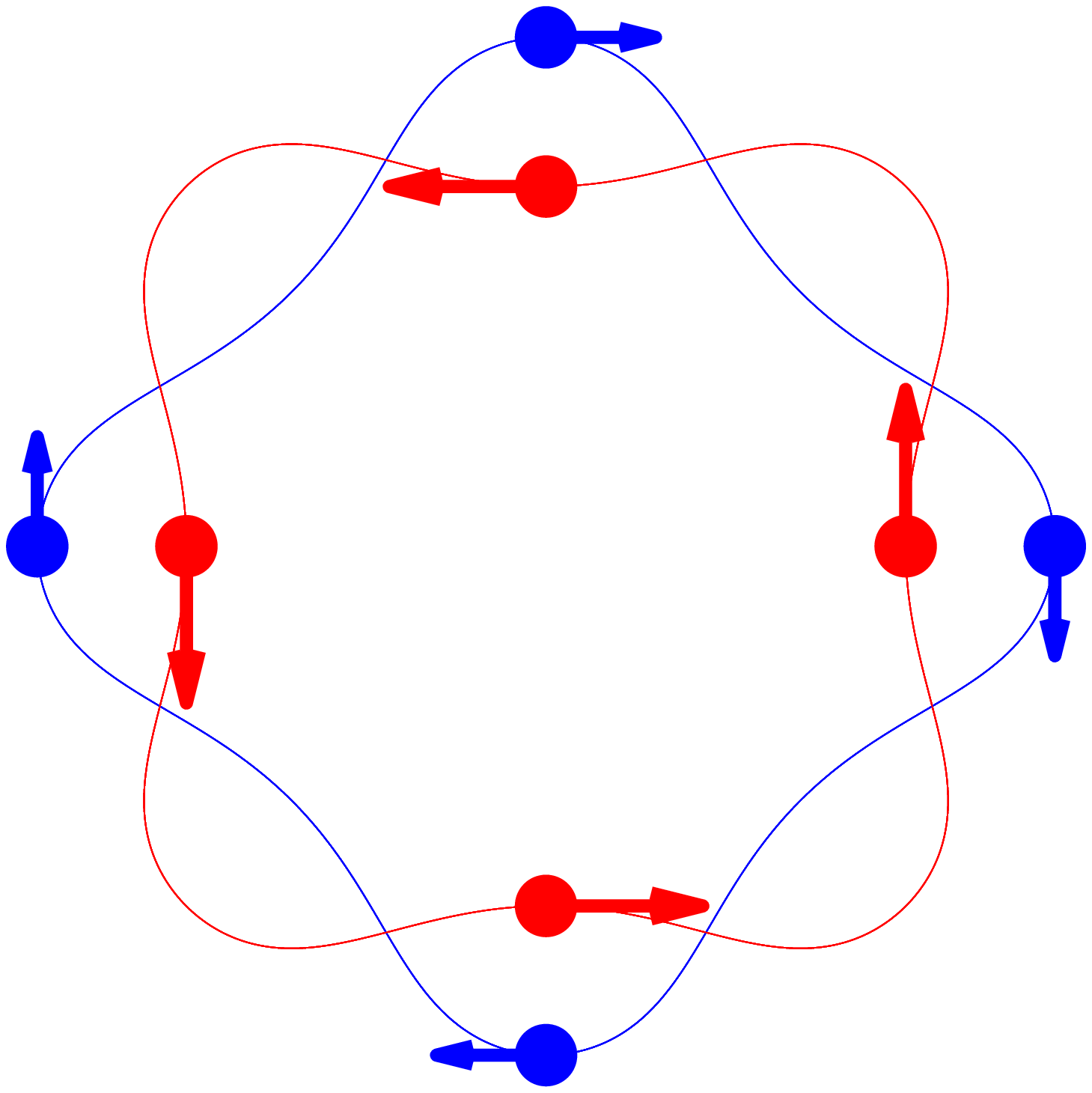}\hspace{0.5cm}
\includegraphics[width=1.8cm]{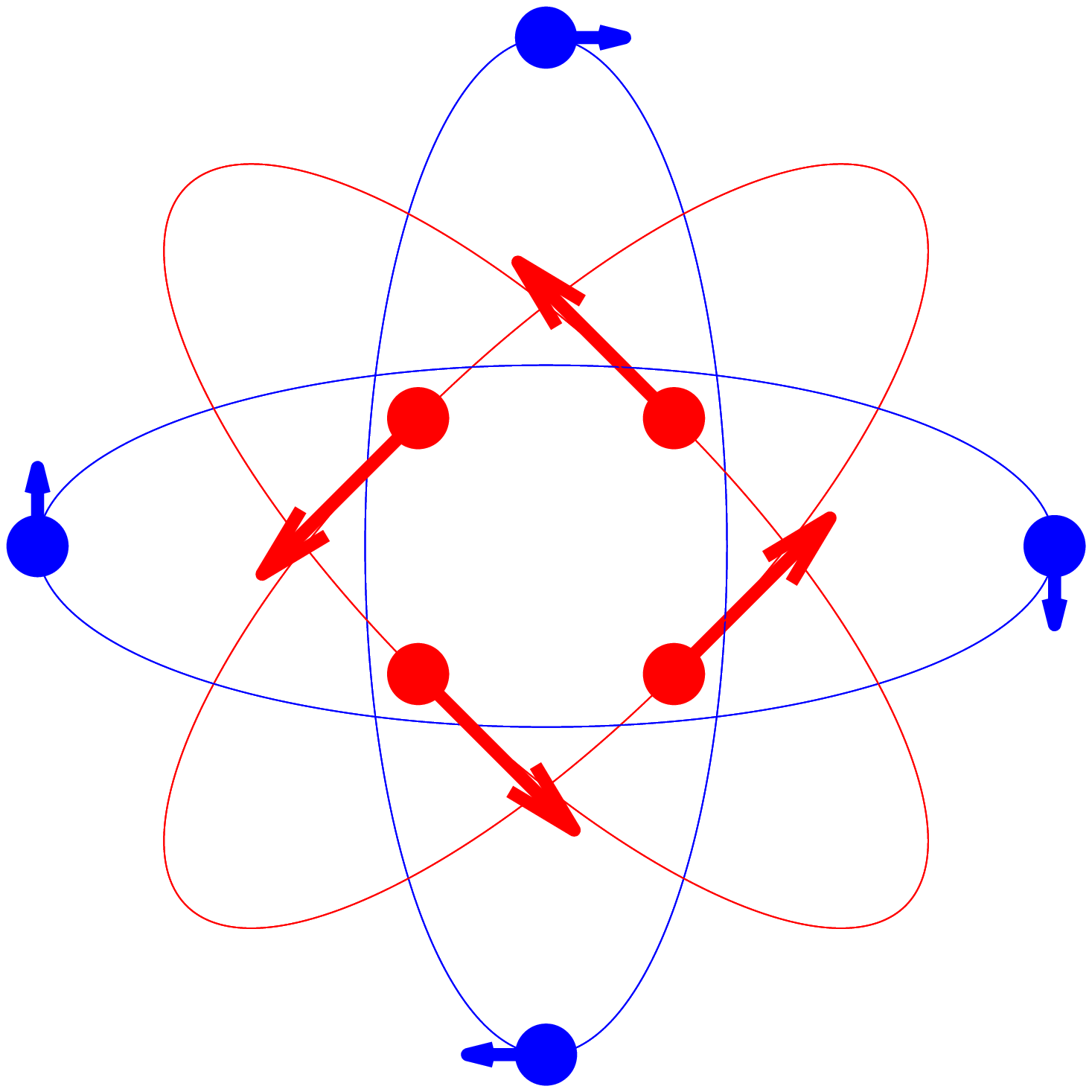}\hspace{0.5cm}
$t=\frac{dT}{4n}$
\vspace{0.2cm}

\includegraphics[width=1.8cm]{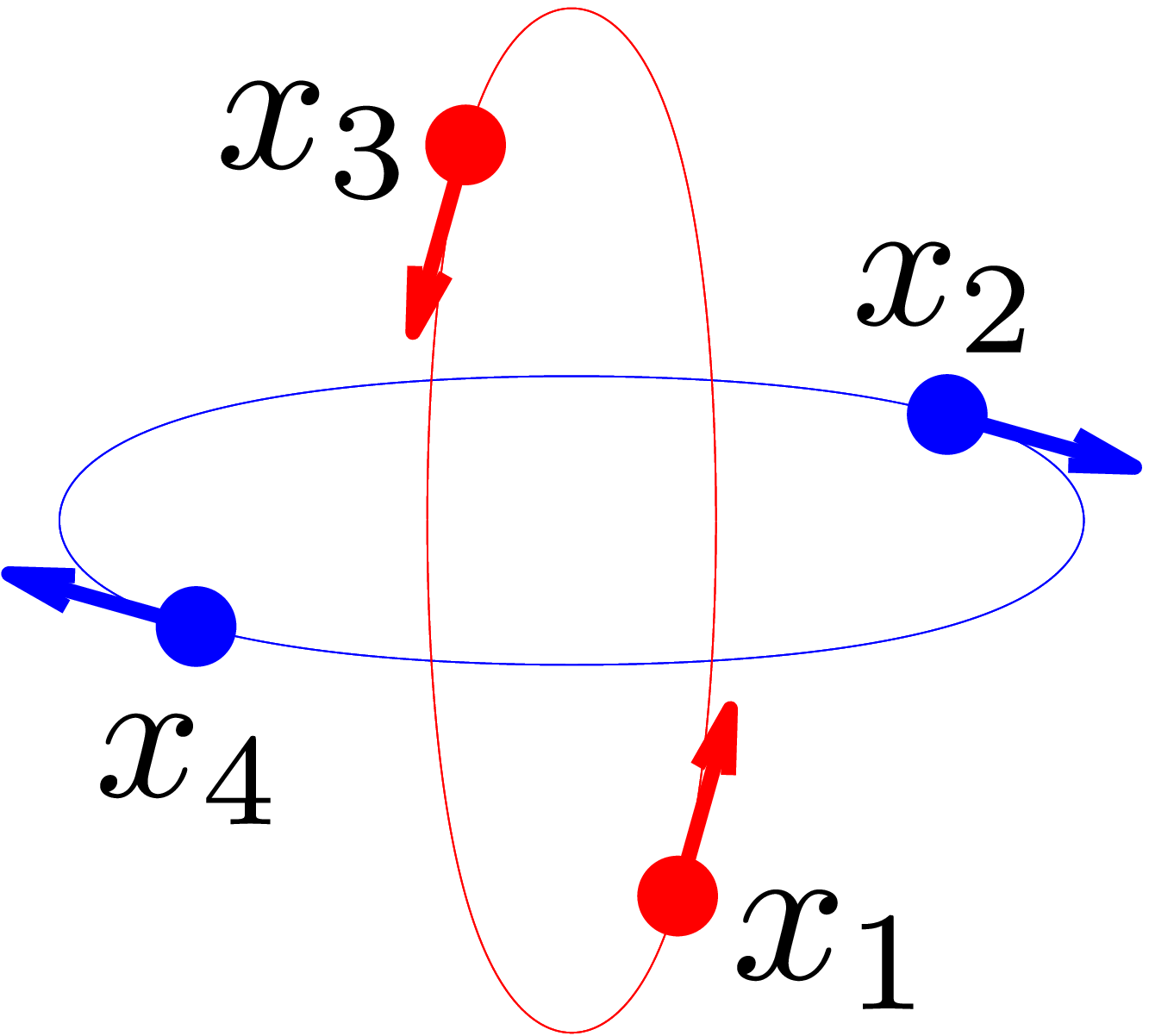}\hspace{0.5cm}
\includegraphics[width=1.8cm]{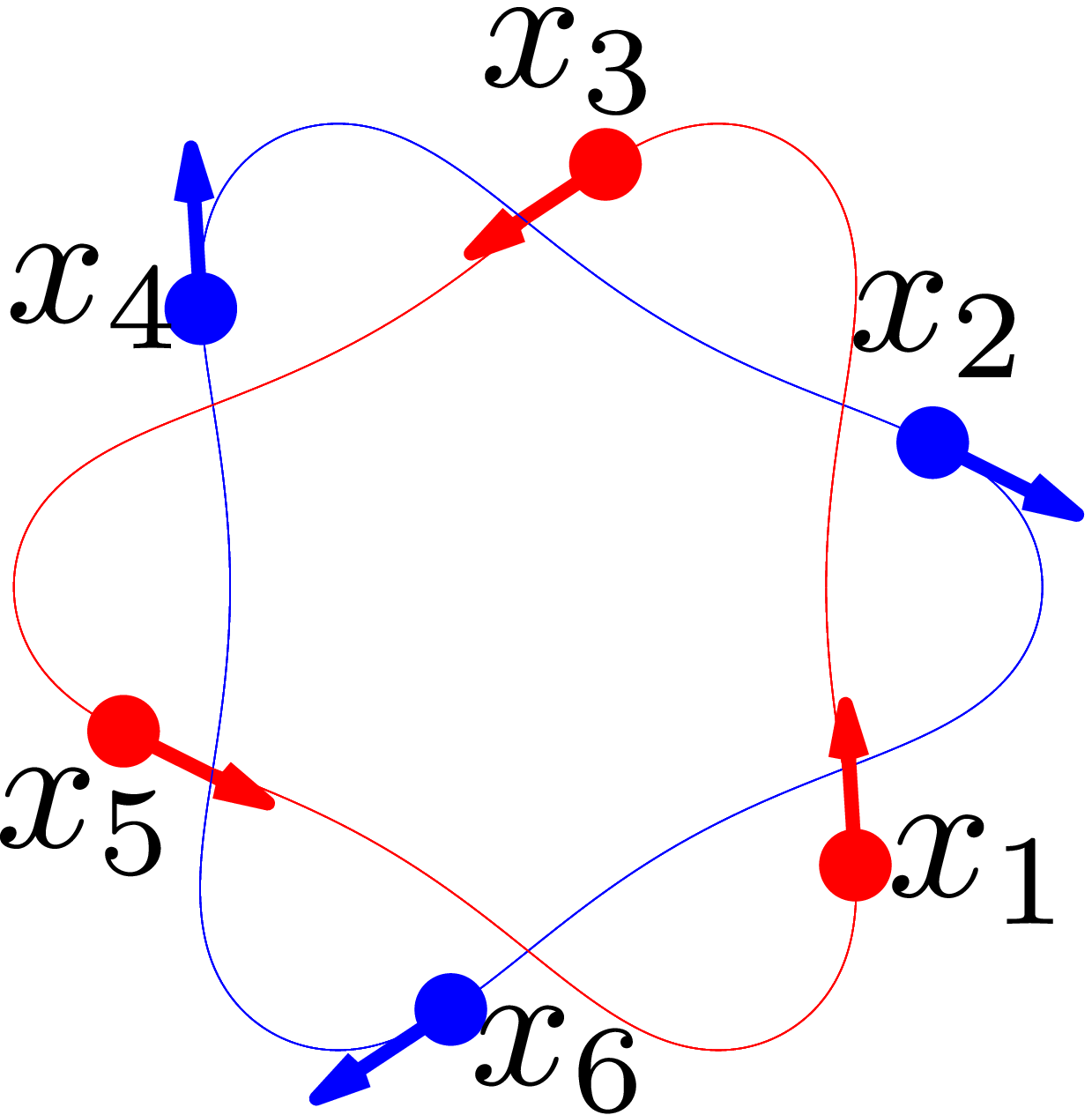}\hspace{0.5cm}
\includegraphics[width=1.8cm]{BRn=4_p=1_1.eps}\hspace{0.5cm}
\includegraphics[width=1.8cm]{BRn=4_p=2_1.eps}\hspace{0.5cm}
$t=0$


(1) \hspace{1.7cm}  (2) \hspace{1.7cm}  (3) \hspace{1.7cm}  (4) \hspace{2.0cm}

\end{center}
\end{figure}

\begin{figure}
\begin{center}
\includegraphics[width=1.8cm]{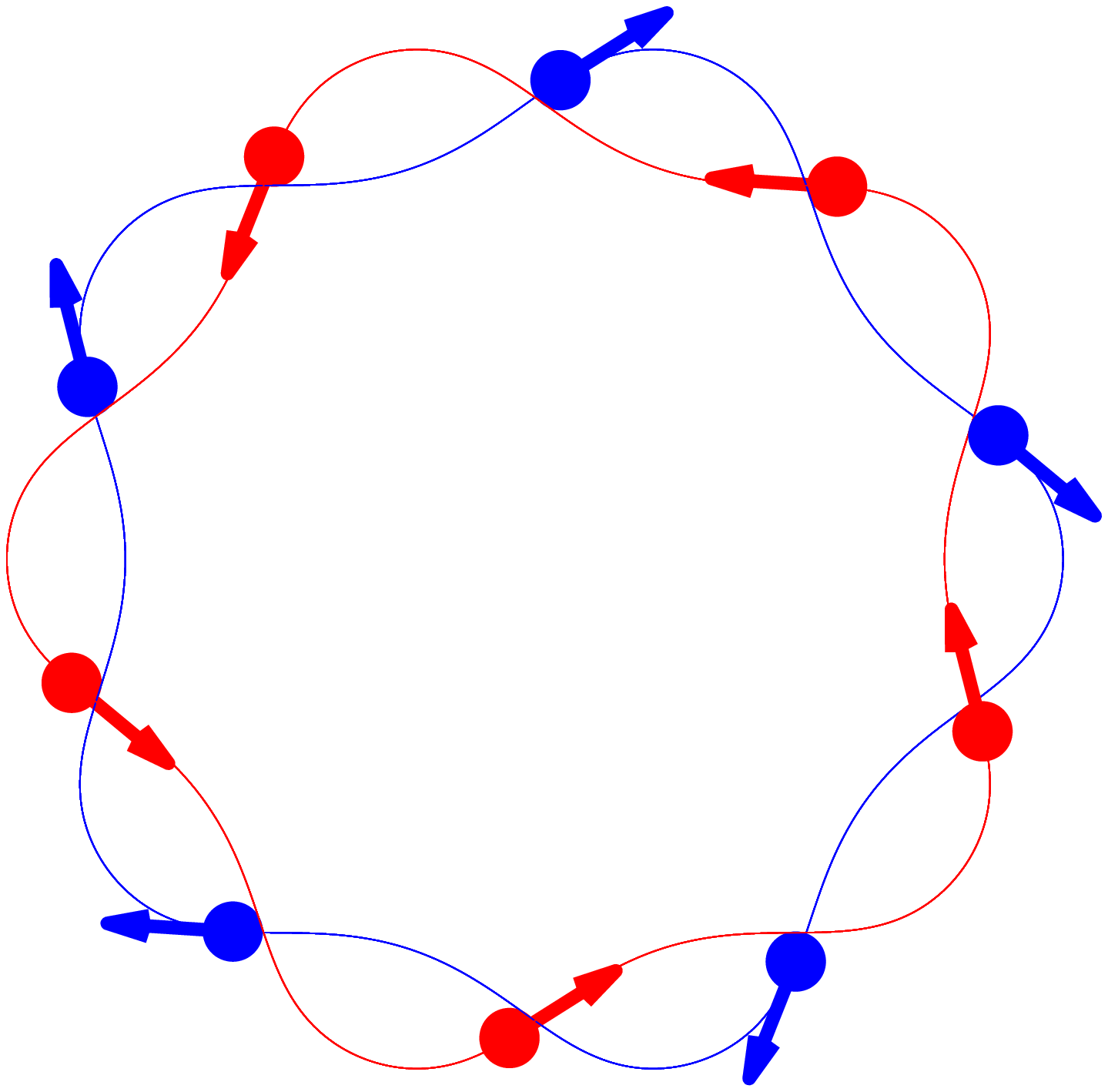}\hspace{0.5cm}
\includegraphics[width=1.8cm]{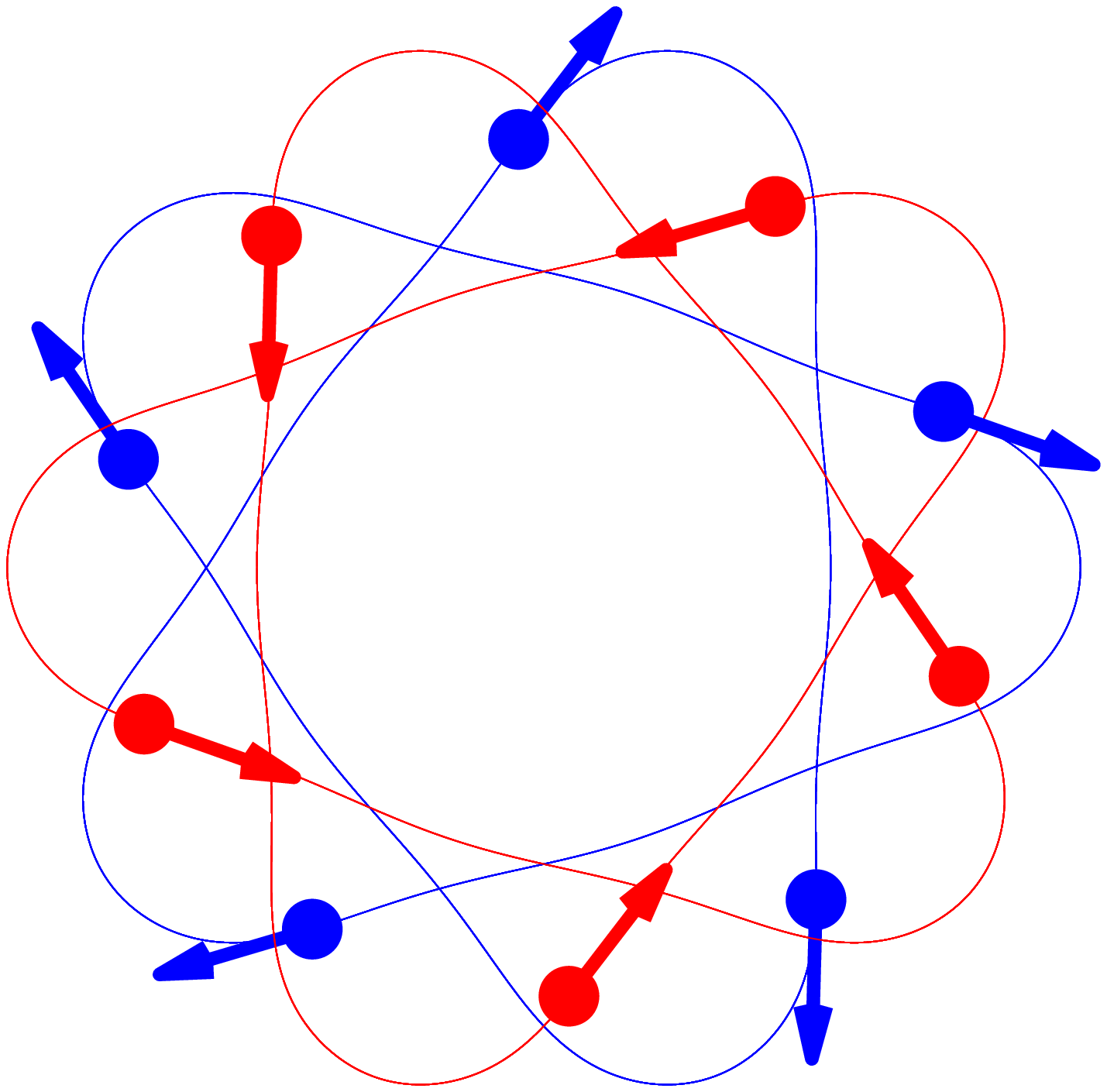}\hspace{0.5cm}
\includegraphics[width=1.8cm]{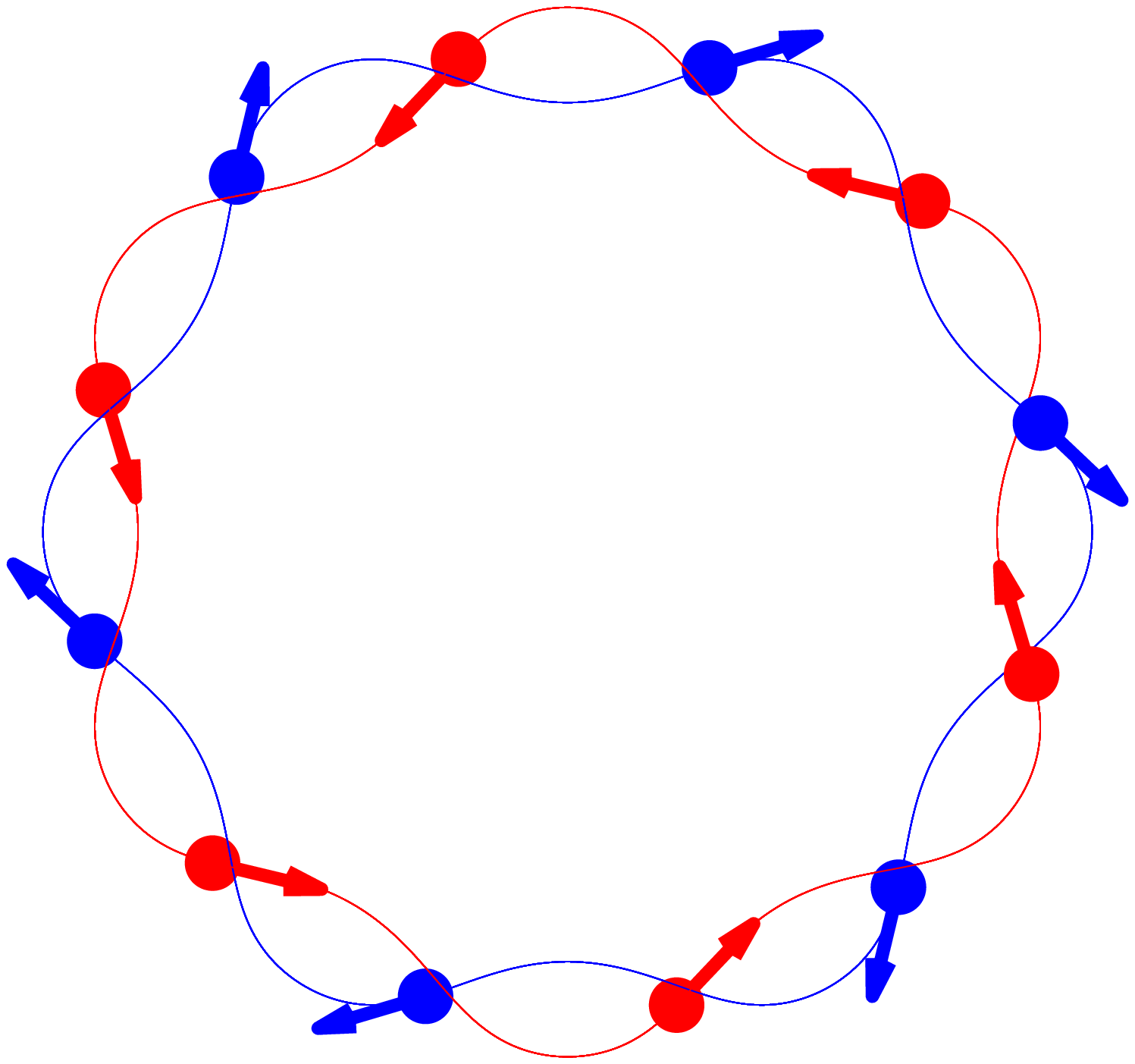}\hspace{0.5cm}
\includegraphics[width=1.8cm]{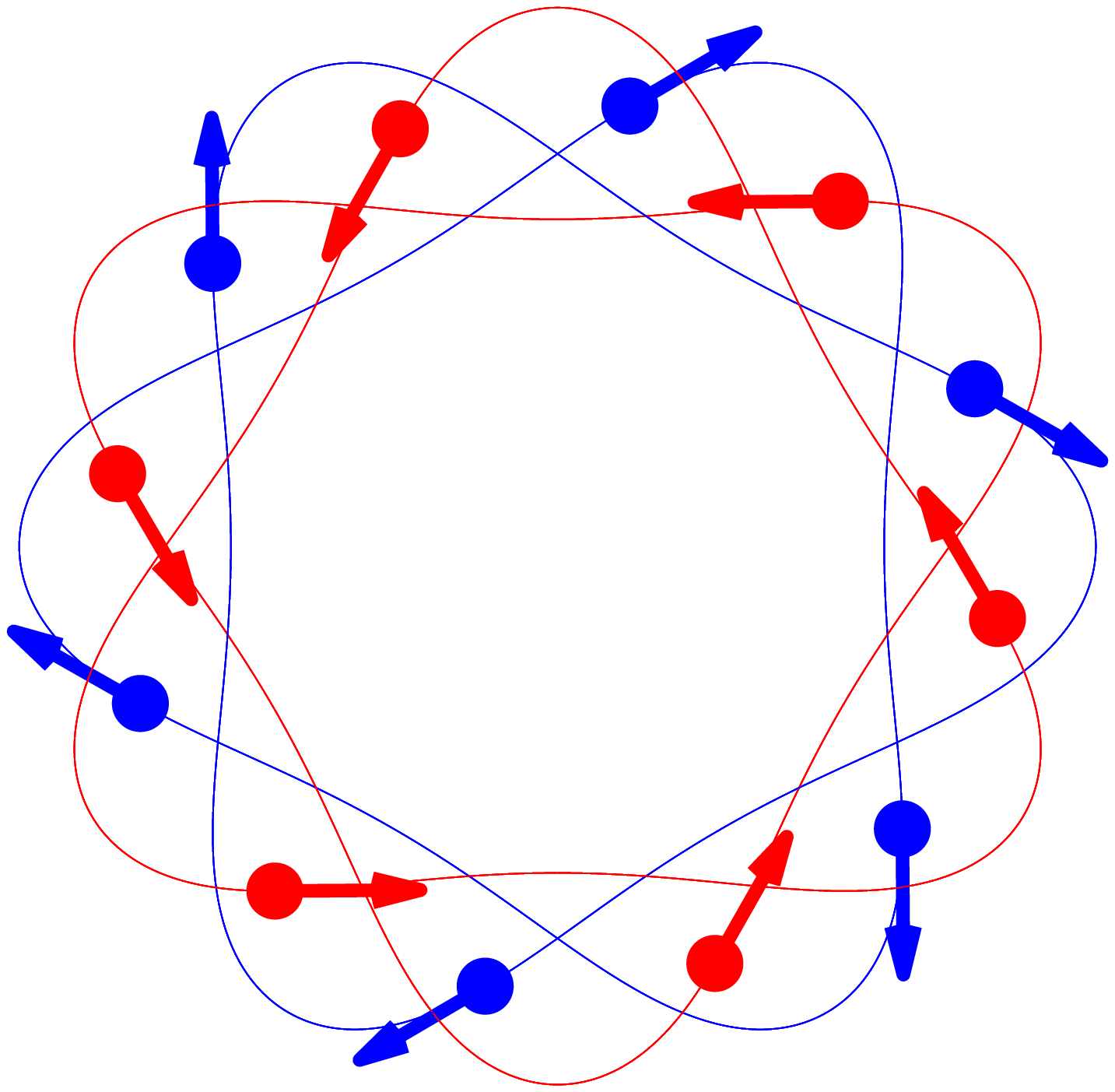}\hspace{0.5cm}
$t=\frac{dT}{n}$

\includegraphics[width=1.8cm]{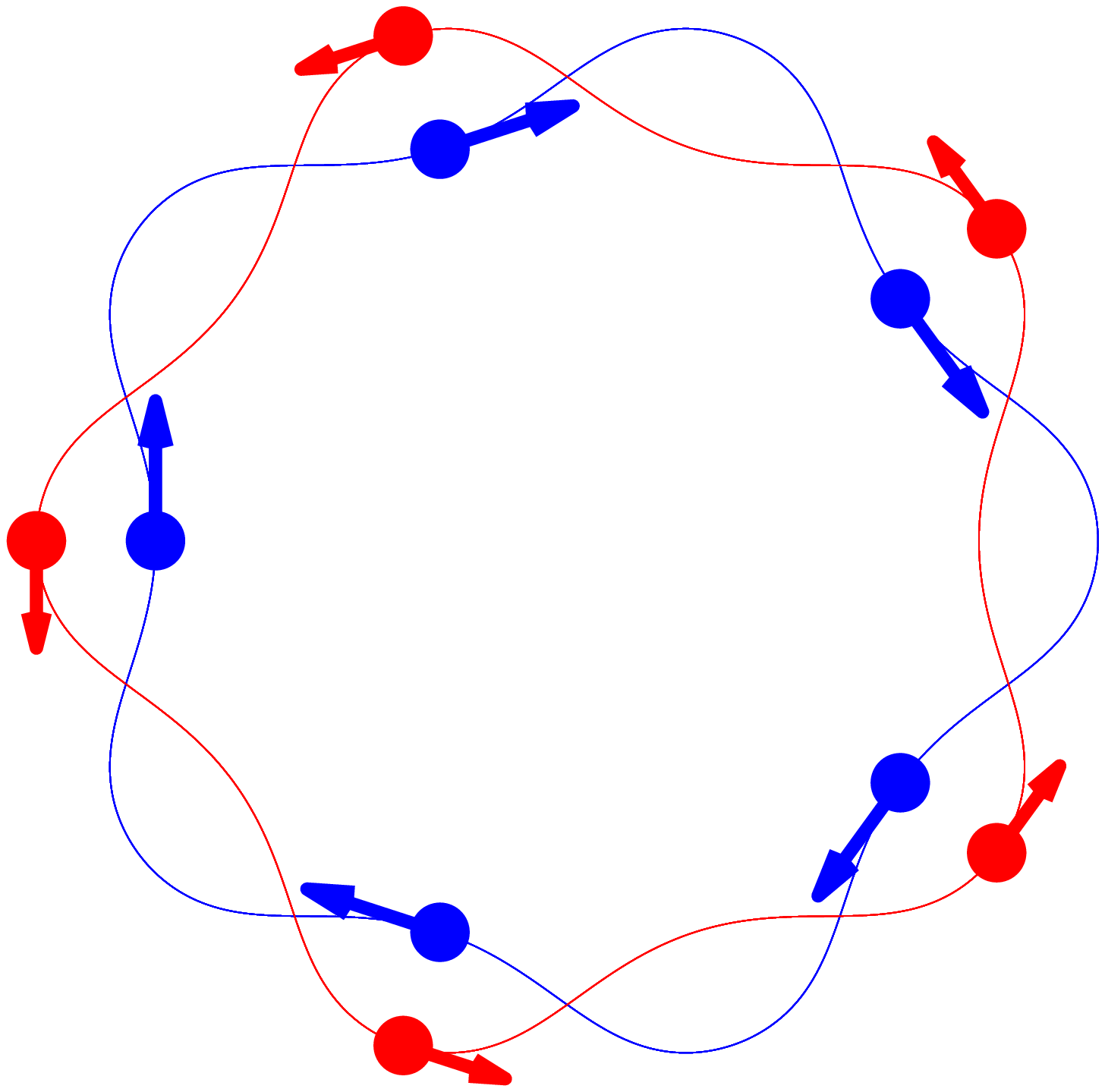}\hspace{0.5cm}
\includegraphics[width=1.8cm]{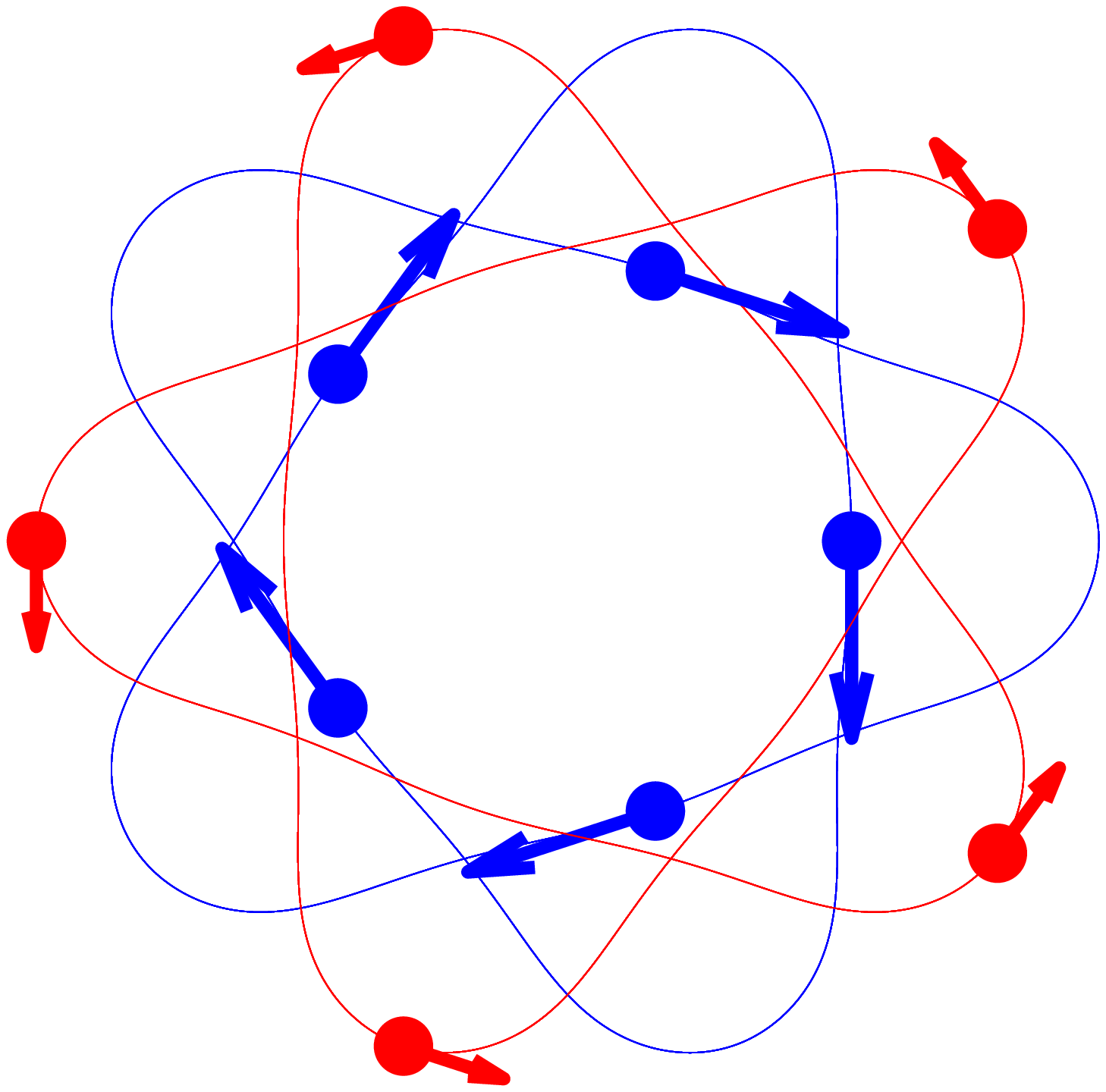}\hspace{0.5cm}
\includegraphics[width=1.8cm]{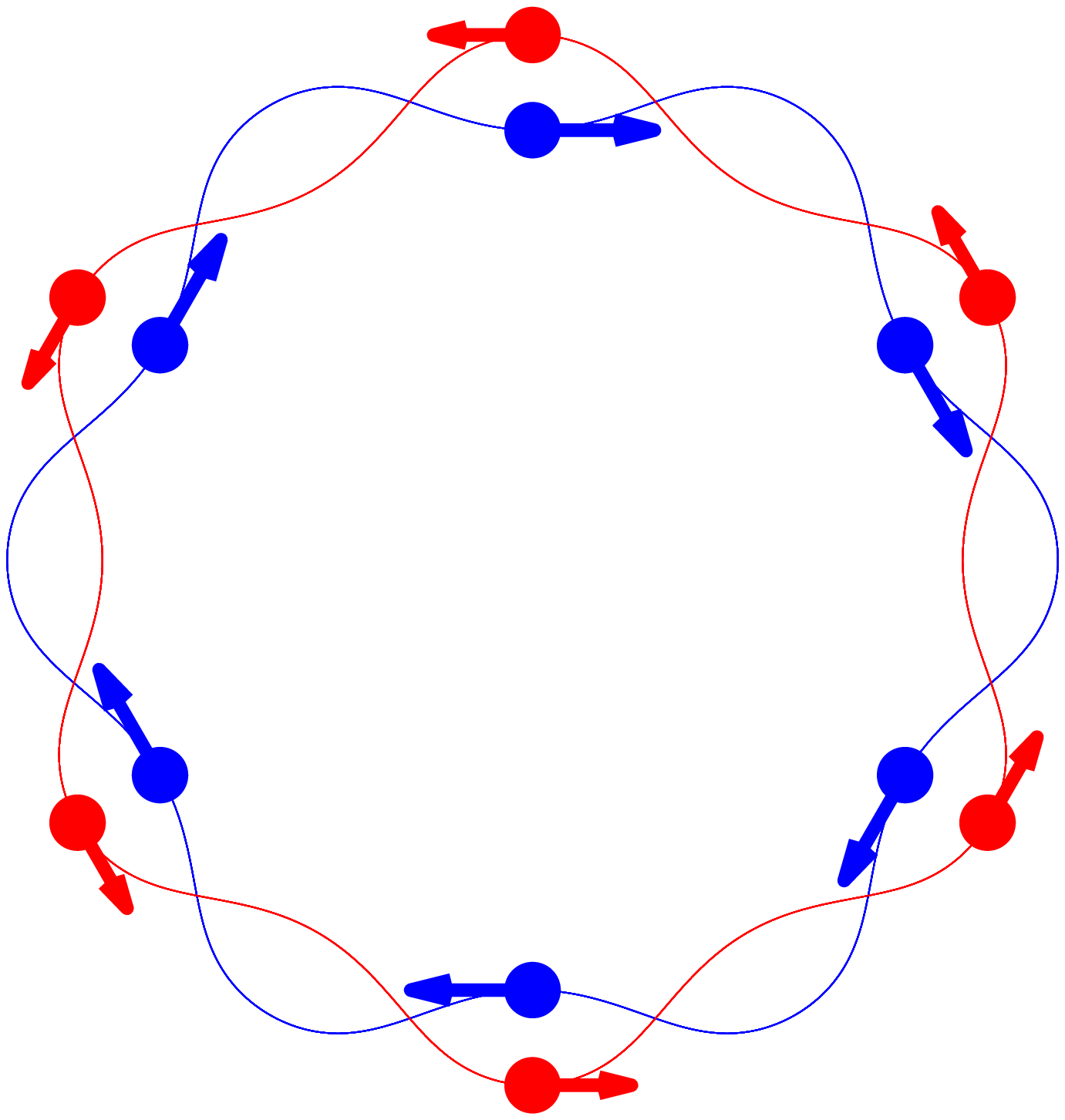}\hspace{0.5cm}
\includegraphics[width=1.8cm]{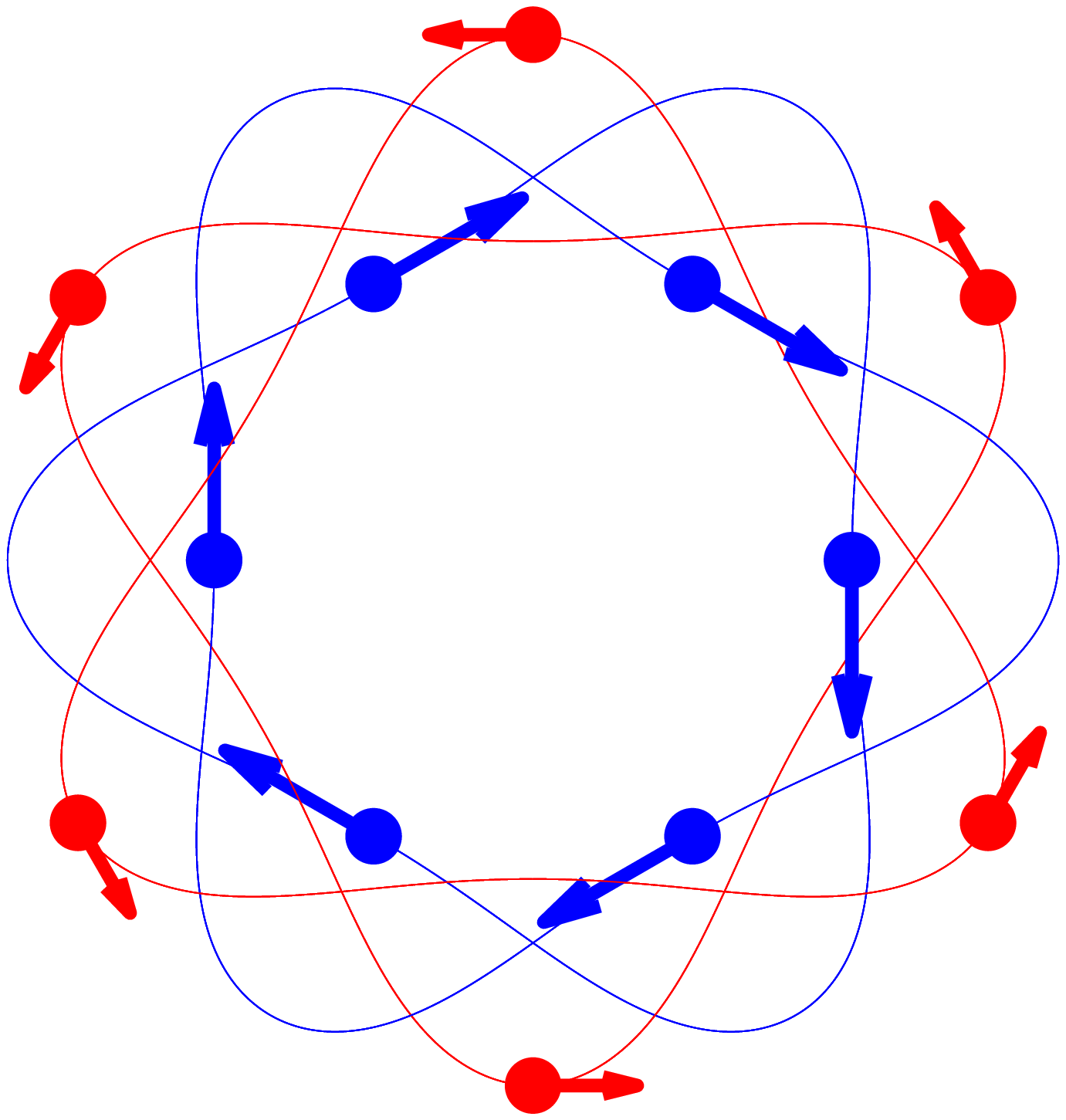}\hspace{0.5cm}
$t=\frac{3dT}{4n}$

\includegraphics[width=1.8cm]{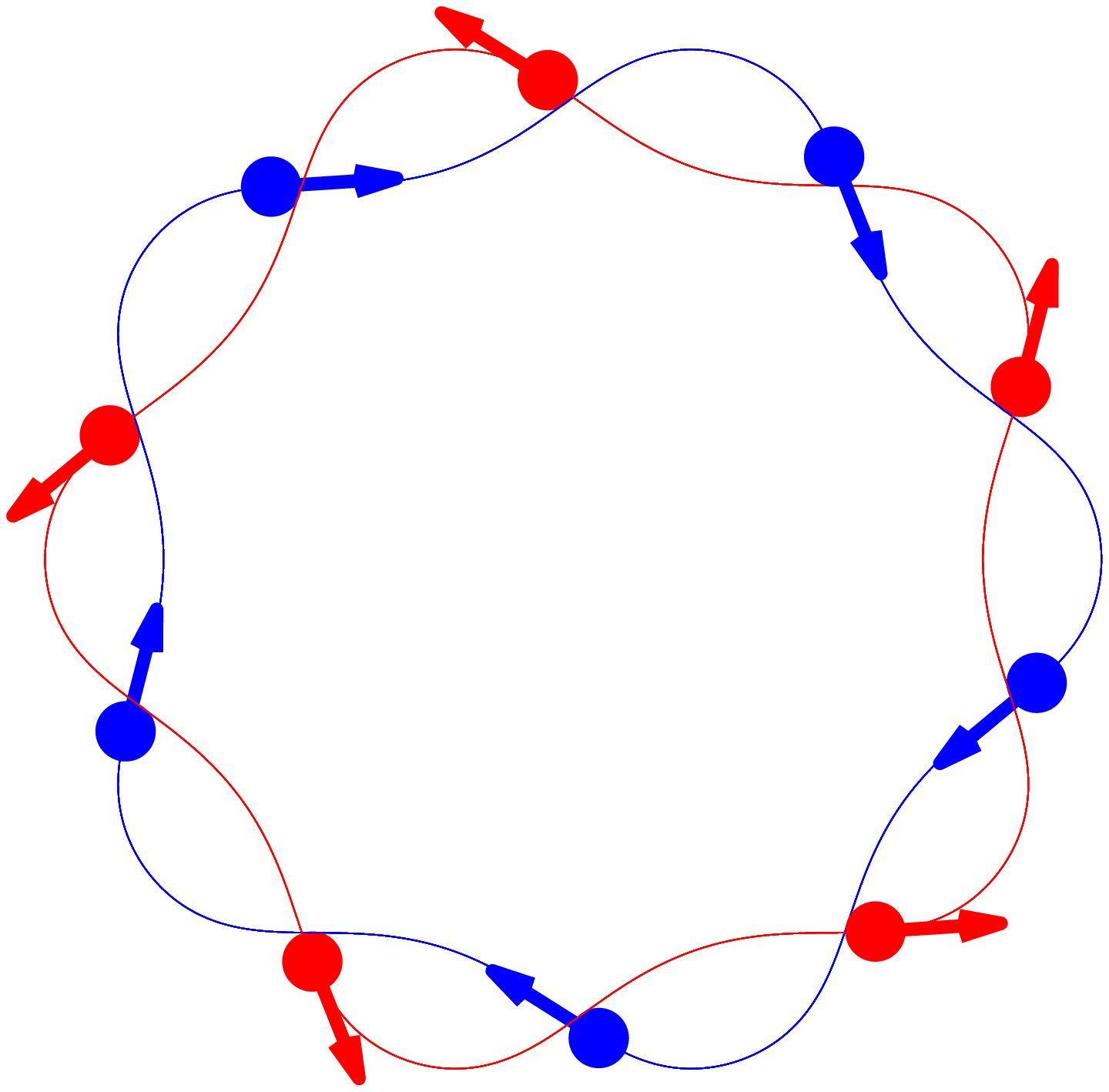}\hspace{0.5cm}
\includegraphics[width=1.8cm]{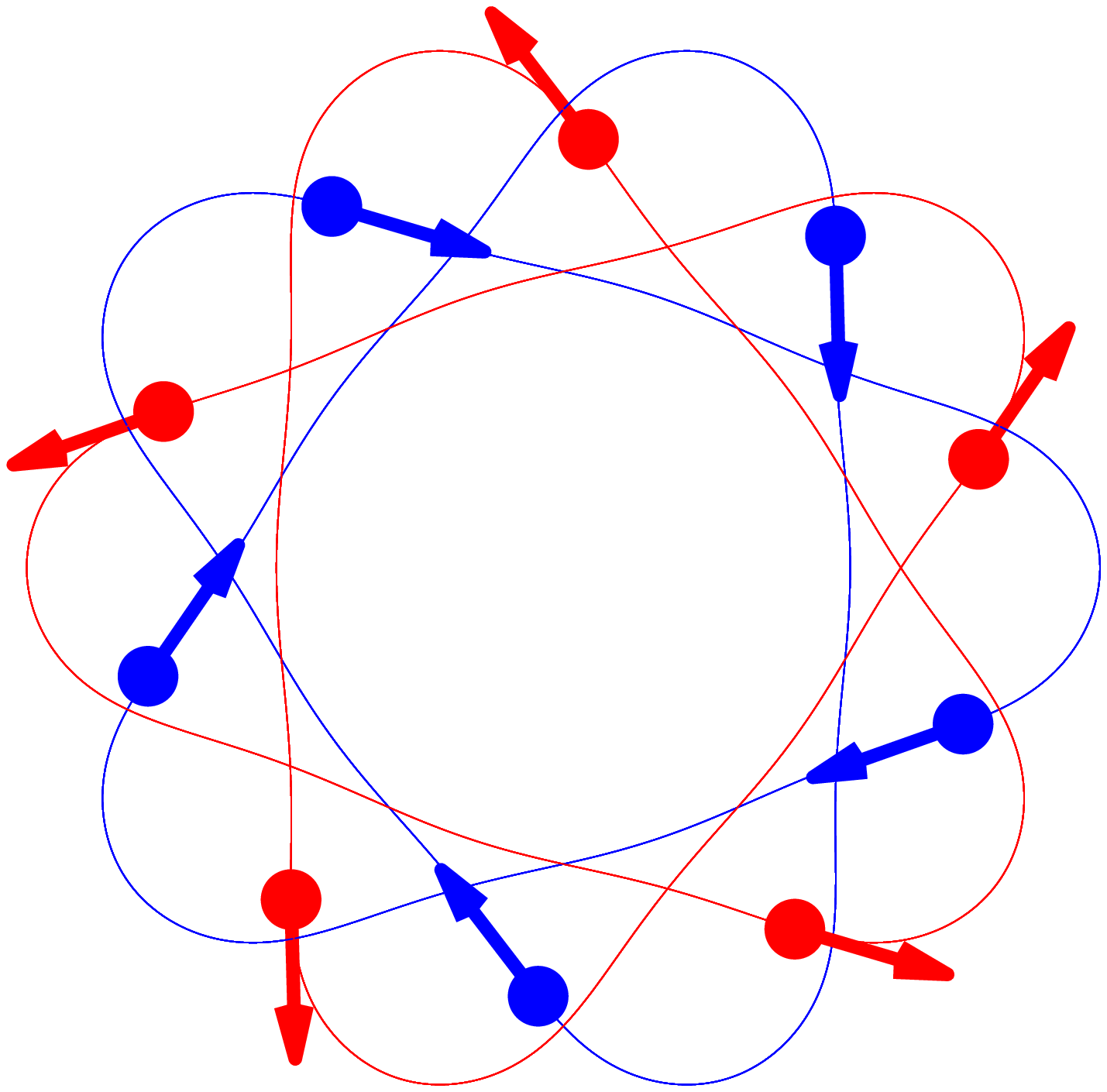}\hspace{0.5cm}
\includegraphics[width=1.8cm]{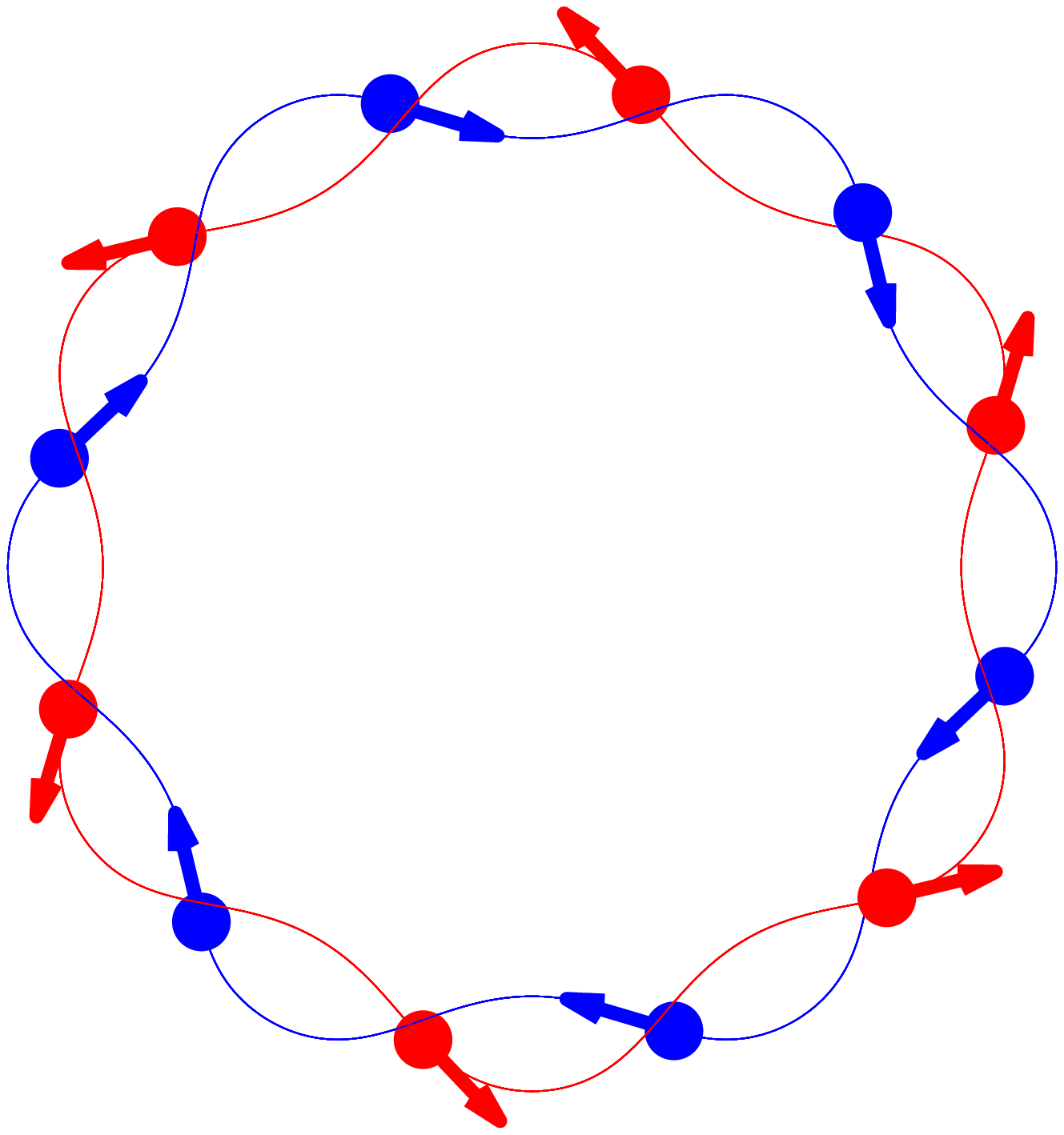}\hspace{0.5cm}
\includegraphics[width=1.8cm]{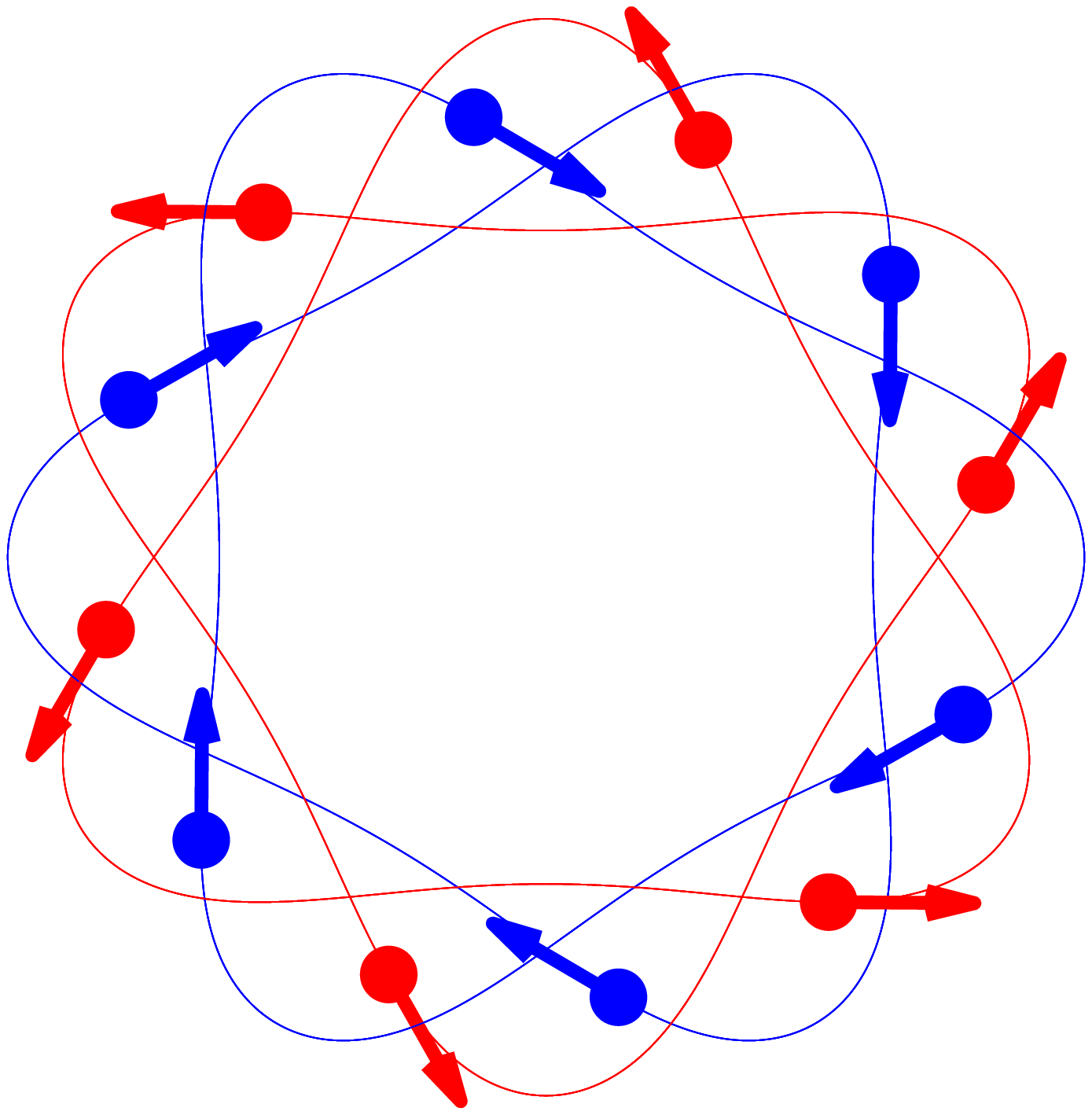}\hspace{0.5cm}
$t=\frac{dT}{2n}$

\includegraphics[width=1.8cm]{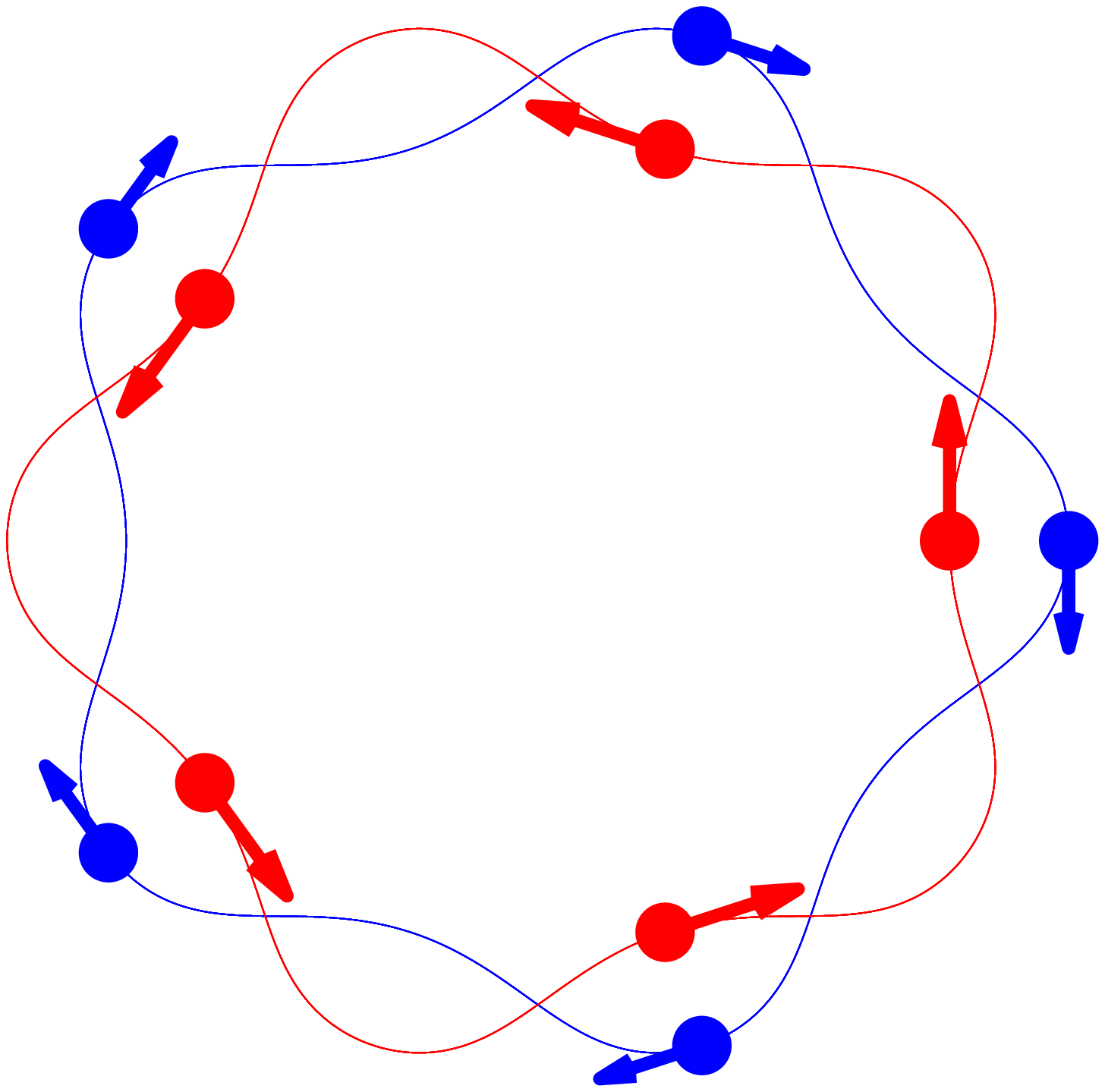}\hspace{0.5cm}
\includegraphics[width=1.8cm]{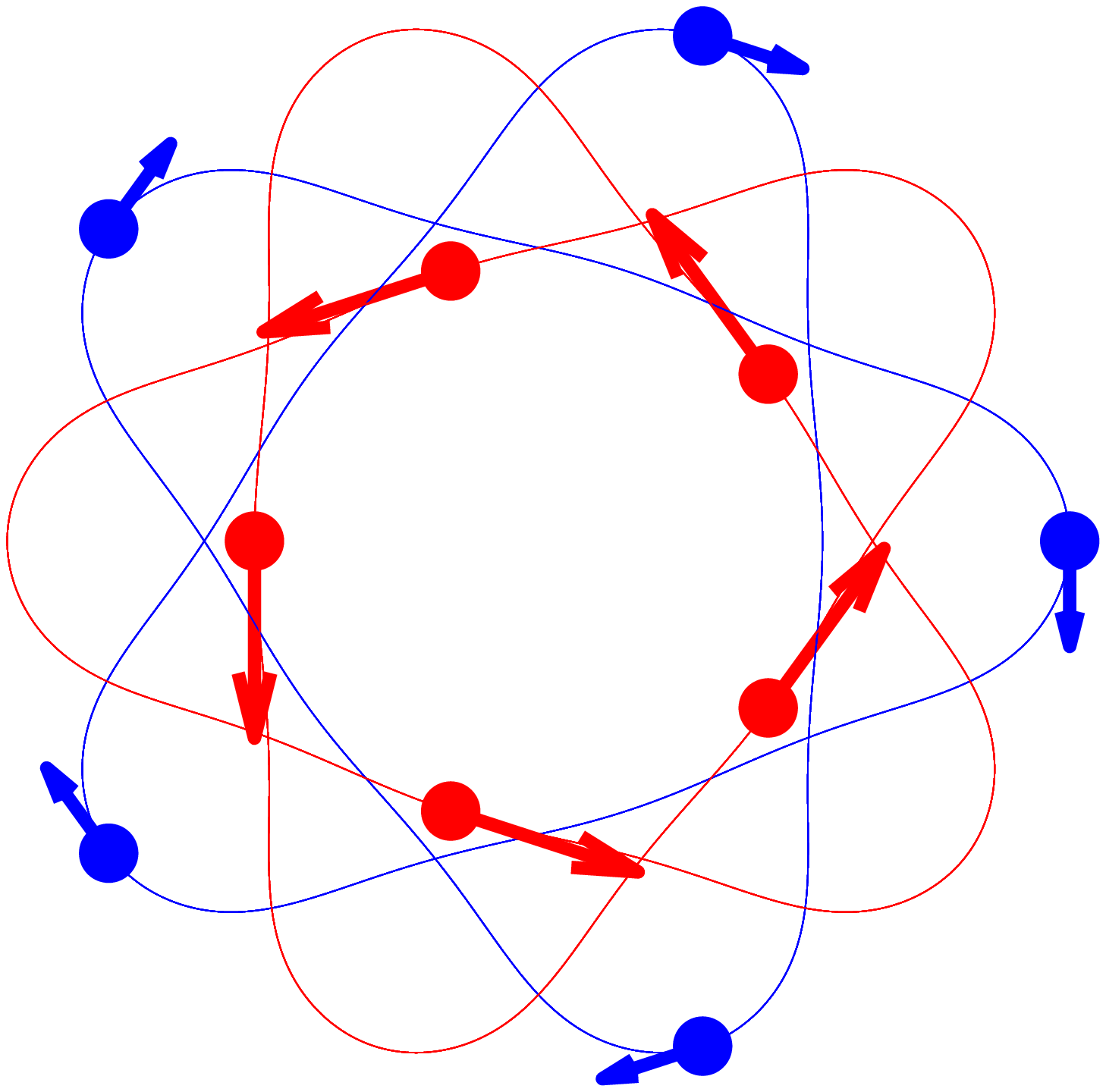}\hspace{0.5cm}
\includegraphics[width=1.8cm]{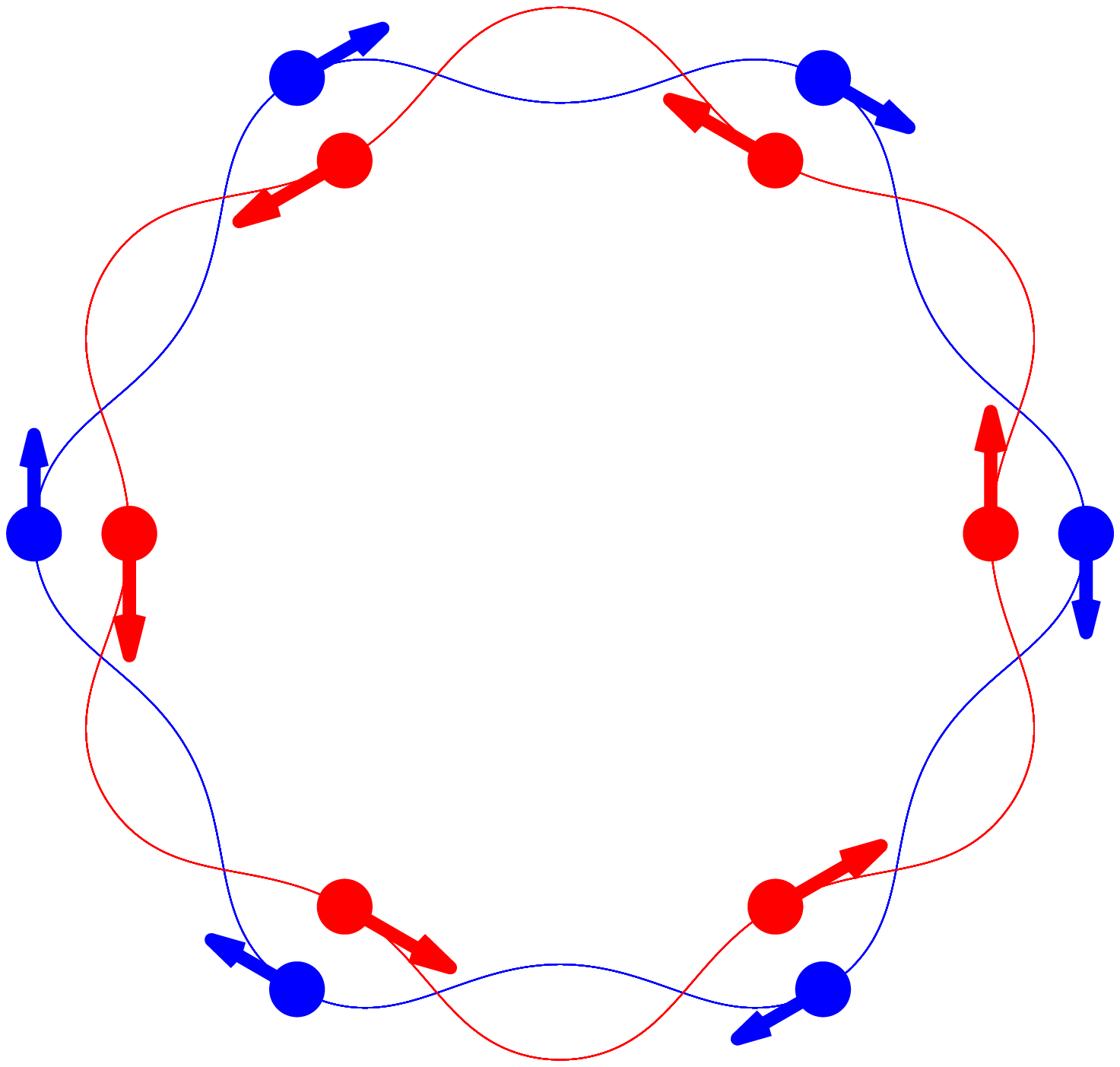}\hspace{0.5cm}
\includegraphics[width=1.8cm]{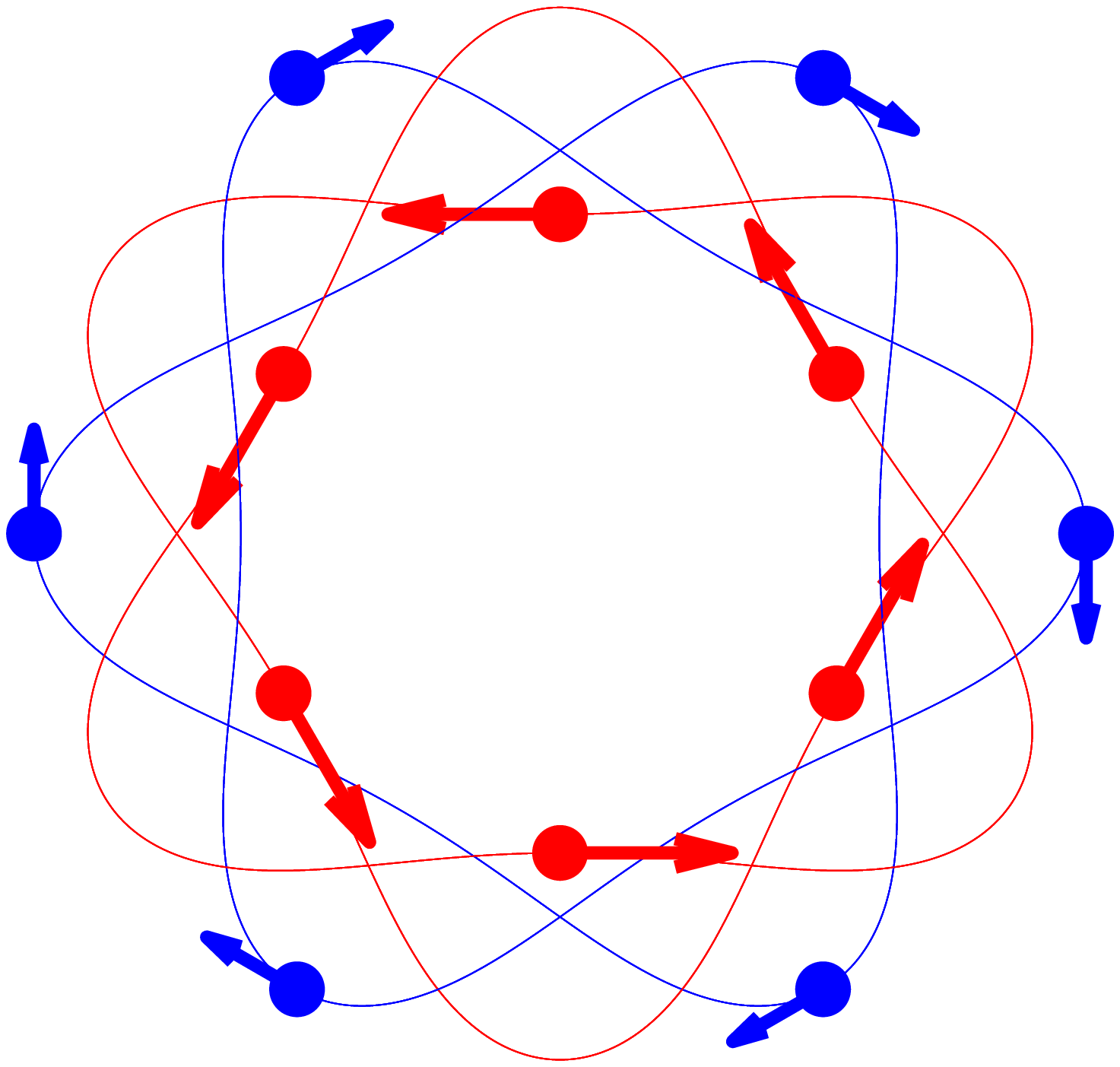}\hspace{0.5cm}
$t=\frac{dT}{4n}$

\includegraphics[width=1.8cm]{BRn=5_p=1_1.eps}\hspace{0.5cm}
\includegraphics[width=1.8cm]{BRn=5_p=2_1.eps}\hspace{0.5cm}
\includegraphics[width=1.8cm]{BRn=6_p=1_1.eps}\hspace{0.5cm}
\includegraphics[width=1.8cm]{BRn=6_p=2_1.eps}\hspace{0.5cm}
$t=0$

(5) \hspace{1.7cm}  (6) \hspace{1.7cm}  (7) \hspace{1.7cm}  (8) \hspace{2.0cm}

\caption{
${\bm x}_{n,p}(t)$ for  $0 \le t \le (\frac{d}{n})T$: 
(1) ${\bm x}_{2,1}(t)$. 
(2) ${\bm x}_{3,1}(t)$. 
(3) ${\bm x}_{4,1}(t)$. 
(4) ${\bm x}_{4,2}(t)$. 
(5) ${\bm x}_{5,1}(t)$. 
(6) ${\bm x}_{5,2}(t)$. 
(7) ${\bm x}_{6,1}(t)$. 
(8) ${\bm x}_{6,2}(t)$. 
}
\label{fig_orbit}
\end{center}

\end{figure}

\begin{thm}
\label{thm_braidtype}
For $n \ge 2$ and $p \in \{1, \dots, \lfloor  \frac{n}{2}\rfloor \}$, 
the braid type $Y_{n,p}$  is pseudo-Anosov with the stretch factor $(\mathfrak{s}_{2p})^2$. 
In particular,  the braid type $X_{n,p}$ of the solution ${\bm x}_{n,p}(t)$  is pseudo-Anosov 
with the stretch factor  $(\mathfrak{s}_{2p})^{\frac{2n}{d}}$. 
\end{thm}

Since  $\mathfrak{s}_k < \mathfrak{s}_{k'}$ if $k < k'$, 
we immediately have the following result.

\begin{cor}
\label{cor_prime}
Let $X_{n,p}$ be the braid type as in Theorem \ref{thm_braidtype}. 
For $n \ge 2$ and  $p, p' \in \{1, \dots, \lfloor  \frac{n}{2}\rfloor\}$ with $p< p'$,  
we have the following. 

\begin{enumerate}
\item[(1)] 
$\lambda(X_{n,p}) < \lambda(X_{n,p'})$ if $\gcd(n,p)= \gcd(n, p')$. 
In particular 
$X_{n,p} \ne X_ {n,p'}$. 

\item[(2)]  
$\lambda(X_{n,p}) < \lambda(X_{n,p'})$ if $n$ is prime. 
In particular 
$X_{n,p} \ne X_ {n,p'}$. 
\end{enumerate}
\end{cor}

The $2n$ bodies for the solution ${\bm x}_{n,p}(t)$ form a regular $2n$-gon at the initial time $t=0$, 
and  the next first time is  $t= (\tfrac{d}{2n})T$ when the $2n$ bodies form a regular $2n$-gon again. 
See Figure \ref{fig_orbit} for  ${\bm x}_{n,p}(t)$ at $t= (\tfrac{d}{2n})T$. 
From the viewpoint of the configuration of the ``next" regular $2n$-gon, 
it is proved in \cite{Shibayama06} that 
${\bm x}_{n,p}(t)$ and  ${\bm x}_{n,p'}(t)$ are distinct solutions for distinct $p, p' \in  \{1, \dots, \lfloor  \frac{n}{2}\rfloor\}$ 
(Remark \ref{remark_distinct}).  
On the other hand, 
from the viewpoint of braid types, 
Corollary \ref{cor_prime} tells us that 
$X_{n,p}$ is different from  $X_ {n,p'}$ if $\gcd(n,p)= \gcd(n, p')$.

Table \ref{table_entropy} shows the stretch factor 
$\lambda_ {n,p}= \lambda(X_{n,p})$  and the entropy $\log(\lambda_ {n,p})$ 
for  several pairs $(n,p)$. 
One can see from this table that 
$\lambda(X_{n,p}) \ne \lambda(X_{n,p'})$ if $p \ne p'$ up to $n=11$. 
Therefore the braid types $X_{n,p}$ and  $X_{n,p'}$ of the solutions  for $p \neq p' \in \{1, \dots, \lfloor  \frac{n}{2}\rfloor \}$  are distinct
up to $n=11$.

Because of an intriguing formula of metallic ratios, $\mathfrak{s}_k^3 = \mathfrak{s}_{k^3+ 3k}$ for example, 
stretch factors $\lambda(X_{n,p})$ happen to coincide with the ones for different pairs $(n,p')$ occasionally
(see Example \ref{example_occasionally}). 
Nevertheless, we conjecture that 
$X_{n,p} \ne X_ {n,p'}$ for all $p, p' \in \{1, \dots, \lfloor  \frac{n}{2}\rfloor\}$ with $p \ne p'$.

\begin{table}[hbtp]
\caption{
Some examples obtained from Theorem \ref{thm_stretch-factor}.} 
\label{table_entropy}
\begin{center}
\begin{tabular}{|c|c|c|c|c|}
\hline
$(n,p)$ & $d$ & $X_{n,p}= \langle (\beta_{n,p})^{\frac{n}{d}} \rangle $ & 
$\lambda_{n,p}= (\mathfrak{s}_{2p})^{\frac{2n}{d}}$ & $\log( \lambda_{n,p}) =(\tfrac{2n}{d}) \log(\frak{s}_{2,p})$ \\ \hline
\hline
$(2,1)$ & $1$ & $(\beta_{2,1})^2$ & $ (\mathfrak{s}_2)^4$ &  3.525494348078172
 \\ \hline
\hline
$(3,1)$ & $1$ & $(\beta_{3,1})^3$ & $ (\mathfrak{s}_2)^6$ &  5.288241522117257  \\ \hline
\hline
$(4,1)$ & $1$ & $(\beta_{4,1})^4$ & $ (\mathfrak{s}_2)^8$ & 7.050988696156343  \\ \hline
$(4,2)$ & $2$ & $(\beta_{4,2})^2$ & $ (\mathfrak{s}_4)^4$ &  5.774541900715241 \\ \hline
\hline
$(5,1)$ & $1$ & $(\beta_{5,1})^5$ & $ (\mathfrak{s}_2)^{10}$ &  8.813735870195430 \\ \hline
$(5,2)$ & $1$ & $(\beta_{5,2})^5$ & $ (\mathfrak{s}_4)^{10}$ &  14.436354751788103 \\ \hline
\hline
$(6,1)$ & $1$ & $(\beta_{6,1})^6$ & $ (\mathfrak{s}_2)^{12}$ & 10.576483044234514  \\ \hline
$(6,2)$ & $2$ & $(\beta_{6,2})^3$ & $ (\mathfrak{s}_4)^{6}$ & 8.661812851072861 \\ \hline
$(6,3)$ & $3$ & $(\beta_{6,3})^2$ & $ (\mathfrak{s}_6)^{4}$ &  7.273785836928267 \\ \hline
\hline
$(7,1)$ & $1$ & $(\beta_{7,1})^7$ & $ (\mathfrak{s}_2)^{14}$ & 12.339230218273601  \\ \hline
$(7,2)$ & $1$ & $(\beta_{7,2})^7$ & $ (\mathfrak{s}_4)^{14}$ &  20.210896652503344  \\ \hline
$(7,3)$ & $1$ & $(\beta_{7,3})^7$ & $ (\mathfrak{s}_6)^{14}$ &  25.458250429248935 \\ \hline
\hline
$(8,1)$ & $1$ & $(\beta_{8,1})^8$ & $ (\mathfrak{s}_2)^{16}$ &  14.101977392312687 \\ \hline
$(8,2)$ & $2$ & $(\beta_{8,2})^4$ & $ (\mathfrak{s}_4)^{8}$ & 11.549083801430482  \\ \hline
$(8,3)$ & $1$ & $(\beta_{8,3})^8$ & $ (\mathfrak{s}_6)^{16}$ &  29.095143347713069 \\ \hline
$(8,4)$ & $4$ & $(\beta_{8,4})^2$ & $ (\mathfrak{s}_8)^{4}$ & 8.378850189044405  \\ \hline
\hline
$(9,1)$ & $1$ & $(\beta_{9,1})^9$ & $ (\mathfrak{s}_2)^{18}$ &  15.864724566351773 \\ \hline
$(9,2)$ & $1$ & $(\beta_{9,2})^9$ & $ (\mathfrak{s}_4)^{18}$ & 25.985438553218586  \\ \hline
$(9,3)$ & $3$ & $(\beta_{9,3})^3$ & $ (\mathfrak{s}_6)^{6}$ &  10.910678755392400 \\ \hline
$(9,4)$ & $1$ & $(\beta_{9,4})^9$ & $ (\mathfrak{s}_8)^{18}$ &  37.704825850699820 \\ \hline
\hline 
$(10,1)$ & $1$ & $(\beta_{10,1})^{10}$ & $ (\mathfrak{s}_2)^{20}$ &  17.627471740390860 \\ \hline
$(10,2)$ & $2$ & $(\beta_{10,2})^{5}$ & $ (\mathfrak{s}_4)^{10}$ & 14.436354751788103  \\ \hline
$(10,3)$ & $1$ & $(\beta_{10,3})^{10}$ & $ (\mathfrak{s}_6)^{20}$ & 36.368929184641338  \\ \hline
$(10,4)$ & $2$ & $(\beta_{10,4})^{5}$ & $ (\mathfrak{s}_8)^{10}$ & 20.947125472611013  \\ \hline
$(10,5)$ & $5$ & $(\beta_{10,5})^{2}$ & $ (\mathfrak{s}_{10})^{4}$ & 9.249753365091010  \\ \hline
\hline
$(11,1)$ & $1$ & $(\beta_{11,1})^{11}$ & $ (\mathfrak{s}_2)^{22}$ &  19.390218914429944 \\ \hline
$(11,2)$ & $1$ & $(\beta_{11,2})^{11}$ & $ (\mathfrak{s}_4)^{22}$ &  31.759980453933828 \\ \hline
$(11,3)$ & $1$ & $(\beta_{11,3})^{11}$ & $ (\mathfrak{s}_6)^{22}$ &  40.005822103105473 \\ \hline
$(11,4)$ & $1$ & $(\beta_{11,4})^{11}$ & $ (\mathfrak{s}_8)^{22}$ & 46.083676039744226  \\ \hline
$(11,5)$ & $1$ & $(\beta_{11,5})^{11}$ & $ (\mathfrak{s}_{10})^{22}$ & 50.873643508000555  \\ \hline
\end{tabular}
\end{center}
\end{table}


The organization of the paper is as follows. 
In Section \ref{sec:xnp} 
we introduce a family of periodic solutions ${\bm x}_{n,p}(t)$ in \cite{Shibayama06} 
of the planar $2n$-body problem. 
In Section \ref{section_braid-groups} 
we briefly review the necessarily background on braid groups. 
We prove Theorem \ref{thm_braidtype} in Section \ref{section_proof-of-theorem}.  
In Section \ref{section_numerical}, we give new numerical periodic solutions 
${\bm x}_{n,p}(t)$ of the planar $2n$-body problem 
 when  $p> \lfloor \frac{n}{2}\rfloor$.


\section{Periodic solutions of the planar $2n$-body problem}
\label{sec:xnp}

This section is devoted to explain the periodic solutions ${\bm x}_{n,p}(t)$.  
The existence was proven with the variational method. 
They have high symmetries because they can be represented as elements of a functional space limited by several group actions.
The minimizers of the action functional under the symmetry correspond to those solutions.
They are also regarded as orbits on the {\it shape sphere}.
They are constructed through minimizing methods, and
we omit analytic techniques for the proof and describe geometric properties of ${\bm x}_{n,p}(t)$ 
including the group actions and shape sphere.

\subsection{Symmetry}
\label{subsection_symmetry}
Let $G$ be a finite group. 
We consider a $2$-dimensional orthogonal representation 
$\rho:  G \to O(2)$, 
a homomorphism 
$\sigma:  G \to \mathfrak{S}_{2n}$ 
to the symmetric group on $2n$ elements, and 
another $2$-dimensional orthogonal representation 
$\tau : G \to O(2) $. 
We will denote by $\Lambda$, the set of $T$-periodic orbits.
The action of $G$ to $\Lambda$ is defined by
\[
g \cdot ( (x_1,\dots,x_{2n})(t)) = (\rho(g)x_{\sigma(g^{-1})(1)},\dots,\rho(g)x_{\sigma(g^{-1})(2n)})(\tau (g^{-1})(t))
\]
for $g \in G$ and ${\bm x}(t)=(x_1,\dots,x_{2n})(t) \in \Lambda$, 
where the above $\rho$, $\sigma$, $\tau$ represent respectively actions of $G$ on ${\Bbb R}^2$ by 
orthogonal transformations, on indices $\{1, 2, \dots, 2n\}$ by permutations, and 
on the circle ${\Bbb R}/T {\Bbb Z}$. 
Specifically, we take $G$ as the group $G_{n,p}:=\langle g_{n}, h_{n,p} \rangle$ 
generated by the two elements $g_{n}$ and $h_{n,p}$, 
where
\begin{align*}
\rho(g_n)&=
\left(\begin{array}{cc}\cos(\frac{\pi}{n}) & -\sin(\frac{\pi}{n}) \\ \sin(\frac{\pi}{n}) & \cos(\frac{\pi}{n})\end{array}\right),\\
\sigma(g_n)&=(1,2,\dots,2n),\\
\tau(g_n)&=\left(\begin{array}{cc}1 & 0 \\0 & -1\end{array}\right)\ \mbox{and}\\
\rho(h_{n,p})&=1,\\
\sigma(h_{n,p})&=(1,3,\dots,2n-1)^{-p} (2,4,\dots,2n)^p,\\
\tau(h_{n,p})&=\left(\begin{array}{cc}\cos(\frac{2 \pi d}{n}) & -\sin(\frac{2 \pi d}{n}) \\ \sin(\frac{2 \pi d}{n}) & \cos(\frac{2 \pi d}{n})\end{array}\right) \hspace{2mm} 
(d:= \gcd(n,p)).
\end{align*}
Let us denote by  $\Lambda^G_{n,p}$, the invariant set under the action of $G_{n,p}$ in $\Lambda$, i.e.,
\begin{align*}
\Lambda^G_{n,p} = \{{\bm  x}(t) \in \Lambda \mid 
x_i (t) = \rho(g)&x_{\sigma(g^{-1})(i)}(\tau (g^{-1})(t)) \\
& (i=1,2,\dots,2n, \ g \in G_{n,p} ,\ t \in \mathbb{R})  
\}.
\end{align*}
We now check the properties of $\Lambda^{G}_{n,p}$. 
First, from the invariance under $g_n$, 
we have 
$$g_n\cdot ( (x_1,x_2, \dots,x_{2n})(t)) = (e^{\frac{\pi i}{n}} x_{2n}, e^{\frac{\pi i}{n}}x_1, \dots, e^{\frac{\pi i}{n}}x_{2n-1})(-t).$$
Here we identify $\mathbb{R}^2$ with $\mathbb{C}$.
In particular, 
$x_1(t), x_2(-t), \dots ,x_{2n-1}(t), x_{2n}(-t)$ forms a regular $2n$-gon, and 
$n$ bodies with odd indices  and $n$ bodies with even  indices rotate in mutually opposite directions. 

Second, since
\begin{align*}
\rho(g_n^2)&=
\left(\begin{array}{cc}\cos(\frac{2\pi}{n}) & -\sin(\frac{2\pi}{n}) \\ \sin(\frac{2\pi}{n}) & \cos(\frac{2\pi}{n})\end{array}\right),\\
\sigma(g_n^2)&=(1,3,\dots,2n-1)(2,4,\dots,2n) \text{\ and}\\
\tau(g_n^2)&=1,
\end{align*}
the configuration always consists of two regular $n$-gons,
which are formed by $n$ bodies $x_1(t), x_3(t), \dots, x_{2n-1}(t)$ 
of odd indices and 
$n$ bodies $x_2(t), x_4(t), \dots, x_{2n}(t)$ 
of even indices. 
Thus, to determine the positions of $2n$ bodies $x_1, \dots, x_{2n}$, 
it is sufficient to know the positions of two bodies $x_1$ and $x_2$. 
In fact 
for each $k \in \{1,\dots,n \}$ and $t \in \mathbb{R}$,
\begin{align*}
x_{2k-1}(t) = \omega^{(k-1)} x_1(t), \
x_{2k}(t) = \omega^{(k-1)} x_2(t),
\end{align*}
where $\omega = e^{2\pi i/n}$.
This enables us to use the shape sphere (introduced in Section \ref{subsection_shapesphare}) 
which represents configurations of $2n$ bodies in the periodic solutions.

Lastly, the invariance under $h_{n,p}$ tells us that 
$$h_{n,p} \cdot ((x_1, x_2, \dots, x_{2n})(t)) 
= (x_ {\sigma(h_{n,p}^{-1})(1)}, x_ {\sigma(h_{n,p}^{-1})(2)},\dots, x_ {\sigma(h_{n,p}^{-1})(2n)} ) (t - \tfrac{dT}{n}),$$ 
and hence 
\[
x_{i}\left( t+ \tfrac{dT}{n} \right) = x_{\sigma(h_{n,p}^{-1}) (i)}(t) \hspace{2mm} (i=1,2,\dots ,2n),
\]
where 
\begin{equation} 
\label{equation_permutation}
\sigma(h_{n,p}^{-1})= (1,3, \dots, 2n-1)^p (2,4, \dots, 2n)^{-p} \in \mathfrak{S}_{2n}. 
\end{equation} 
This implies that  
 ${\bm x}_{n,p}(t)$ consists of $2d$ closed curves and 
$\tfrac{n}{d} $ bodies chase one another along each closed curve. 
See Figure \ref{fig_orbit}. 

\subsection{The shape sphere}
\label{subsection_shapesphare}
We consider the group action on the circle $S^1$ to $\mathbb{C}^2$ by
\[
z \cdot (x_1,x_2) = (zx_1,zx_2), \ z \in S^1, (x_1,x_2) \in \mathbb{C}^2.
\]
The quotient space $(\mathbb{C}^2 - \{0\}) / S^1$ under the above action is realized by the following projection:
\begin{align*}
\pi \colon \mathbb{C}^2 - \{ {0} \} &\longrightarrow \mathbb{R}^3 - \{ { 0} \} \ (\cong (\mathbb{C}^2 - \{0\}) / S^1) \\
(x_1,x_2) &\longmapsto {\bm u}(t)=(u_1,u_2,u_3)
\end{align*}
where 
\begin{align*}
(u_1,u_2,u_3)
&=(|x_1|^2-|x_2|^2, 2 \mathrm{Re}(x_1 \bar{x_2}), 2 \mathrm{Im}(x_1 \bar{x_2}))\\
&=(r_1^2-r_2^2, 2 r_1r_2 \cos(\theta_1 - \theta_2), 2r_1r_2 \sin(\theta_1-\theta_2)).
\end{align*}
Here $x_1=r_1 e^{i\theta_1}$ and $x_2 = r_2 e^{i \theta_2}$.
Set the rays
\begin{align*}
\mathrm{A}_{\pm} &= \{ (\pm s, 0, 0) \mid s \in  \mathbb{R}_{>0}  \},\\
\mathrm{B}_{2k} &=\left\{
\left( 0, s \cos (\frac{2 \pi k}{n} ), s \sin (\frac{2 \pi k}{n} ) \mid s \in \mathbb{R}_{>0} \right)
\right\} \ (k \in \mathbb{Z})\ \mbox{and}\\
\mathrm{B}_{2k-1} &= \left\{
\left( 0, s \cos (\frac{(2k-1)\pi }{n} ), s \sin (\frac{(2k-1)\pi }{n} ) \mid s \in \mathbb{R}_{>0} \right)
\right\} \ (k \in \mathbb{Z}).
\end{align*}
In the quotient space $(\mathbb{C}^2 - \{0\}) / S^1$, 
the sets 
$\mathrm{A}_{\pm}$ and $\mathrm{B}_{2k}$ correspond to collisions of the original $2n$ bodies. 
If ${\bm u}(t) \in \mathrm{A}_{+}$ (resp. $\mathrm{A}_-$), then all bodies with odd (resp. even) indeces collide at $t \in {\Bbb R}$
and if ${\bm u}(t) \in \mathrm{B}_{2k}$, then two regular $n$-gons fit. 
See Figure \ref{diagram} 
for the configurations of $8$ bodies corresponding to 
$\mathrm{B}_{0}$,  $\mathrm{B}_{2}$,  $\mathrm{B}_{4}$ and $\mathrm{B}_{6}$.

Let $\bm{u}(t)(=\bm{u}_{n, p}(t))$ be a curve corresponding to the solution $\bm{x}_{n, p}(t)$. 
As a result, ${\bm u}(t)$ passes through neither $\mathrm{A}_{\pm}$ nor $\mathrm{B}_{2k}$, 
because  ${\bm x}_{n,p}(t)$ has no collision (\cite[Proposition $3$]{Shibayama06}). 
On the other hand, each $\mathrm{B}_{2k-1}$ represents  a configuration where $2n$ bodies form a regular $2n$-gon. 
See Figure \ref{diagram} 
for the configurations of $8$ bodies corresponding to 
$\mathrm{B}_{-1}$,  $\mathrm{B}_{1}$,  $\mathrm{B}_{3}$ and $\mathrm{B}_{5}$.

Set
\[
M(k) = 
\left(\begin{array}{ccc}-1 & 0 & 0 \\0 & \cos (\frac{2\pi k}{n}) & \sin (\frac{2\pi k}{n})   \\ 0 &  \sin (\frac{2\pi k}{n})  &  -\cos (\frac{2\pi k}{n}) \end{array}\right).
\]
It is easy to see that $M(k)$ is an orthogonal matrix with eigenvalues $\lambda =1,-1$.
The eigenvector for $\lambda=1$ is $\mathrm{B}_k$, 
and hence  $M(k)$ represents $\pi$-rotation with respect to $\mathrm{B}_k$.
The invariance under $g_n$ is associated with
\begin{align*}
 \left(\begin{array}{c}u_1(-t) \\u_2(-t) \\u_3(-t) \end{array}\right)
=
M(-1)
\left(\begin{array}{c}u_1(t) \\u_2(t) \\u_3(t)\end{array}\right)
\end{align*}
and it implies that ${\bm u}(t)$ and ${\bm u}(-t)$ are symmetric with respect to $\mathrm{B}_{-1}$.
In other words, rotating this curve $\pi$ with respect to $\mathrm{B}_{-1}$,
${\bm u}(t)$ coincides with ${\bm u}(-t)$.
In particular ${\bm u}(0) \in \mathrm{B}_{-1}$.
Similarly, the invariance under $h_{n,p}$ is associated with
\begin{align*}
\left(\begin{array}{c}u_1(t + 2\bar{T}) \\u_2(t +2 \bar{T}) \\u_3(t + 2\bar{T}) \end{array}\right)
=
\left(\begin{array}{ccc}1 & 0 & 0 \\0 & \cos (\frac{4\pi p}{n})  & -\sin (\frac{4\pi p}{n})  \\0 & \sin (\frac{4\pi p}{n}) & \cos (\frac{4\pi p}{n})  \end{array}\right)
\left(\begin{array}{c}u_1(t) \\u_2(t) \\u_3(t)\end{array}\right),
\end{align*}
where $\bar{T}=\frac{dT}{2n}$.
Substituting $ -(t+\bar{T})$ into $t$ and applying the invariance under $g_{n}$, we obtain
\begin{align*}
\left(\begin{array}{c}u_1(-t+\bar{T}) \\u_2(-t+\bar{T}) \\u_3(-t+\bar{T}) \end{array}\right)
=
M(2p-1)
\left(\begin{array}{c}u_1(t+\bar{T}) \\u_2(t+\bar{T}) \\u_3(t+\bar{T})\end{array}\right).
\end{align*}
Taking $t=0$ gives
\begin{align*}
\left(\begin{array}{c}u_1(\bar{T}) \\u_2(\bar{T}) \\u_3(\bar{T}) \end{array}\right)
=
M(2p-1)
\left(\begin{array}{c}u_1(\bar{T}) \\u_2(\bar{T}) \\u_3(\bar{T})\end{array}\right),
\end{align*}
and hence  ${\bm u}(t+\bar{T})$ and ${\bm u}(-t + \bar{T})$ are symmetric with respect to $\mathrm{B}_{2p-1}$ and
${\bm u}(\bar{T}) \in \mathrm{B}_{2p-1}$.
It means that the original configuration of $2n$ bodies forms a regular $2n$-gon again at  $t=\bar{T}$.
Other symmetries with respect to $\mathrm{B}_{2jp-1}$ for $j=2,3, \dots$ can be seen in the same manner.

\begin{remark}
\label{remark_distinct}
It is proved in \cite[Proposition $5$]{Shibayama06} that 
${\bm u}(t) \notin \mathrm{B}_{2k-1}$ for all $t \in (0, \bar{T})$ 
and $k \in \mathbb{Z}$ .
It implies that  
 ${\bm x}_{n,p}(t)$ and ${\bm x}_{n,p'}(t)$ are distinct smooth solution for $p \neq p'$ in the sense that 
 ${\bm u}(\bar{T})$ belongs to $\mathrm{B}_{2p-1}$, 
 that is in the sense that 
 the configuration 
 of the first regular $2n$-gon lives in  the distinct $\mathrm{B}_{2p-1}$ for each $p$. 
 \end{remark}

Consider the projection from $ \mathbb{R}^3 - \{ { 0} \} $ to the $2$-sphere $\mathbb{S}^2$. 
The projective space is called the {\it shape sphere}. 
The image  of ${\bm u}(t) \in  \mathbb{R}^3 - \{ { 0} \}$ under the projection is also denoted by the same notation ${\bm u}(t)$, 
and we call a family $\{{\bm u}(t)\}_t$ (on the shape sphare) the {\it shape curve}. 
Determining the shape curve ${\bm u}(t)$ for $t \in (0,\bar{T})$, 
we obtain the shape curve ${\bm u}(t)$ for all $t \in {\Bbb R}$ from the above symmetries. 
For example, 
we show the shape curves ${\bm u}(t)$ for $t \in {\Bbb R}$ 
when $(n,p)=(3,1)$ and $(n,p)=(4,2)$ in Figure \ref{sphere}.
Each point $\mathrm{B}_i$ in Figure  \ref{sphere} 
indicates the projection of the ray $\mathrm{B}_i$ onto the shape sphere. 
The solid arrows (resp. the dotted arrows) illustrate the shape curve ${\bm u}(t)$ in the front side on the shape sphere, 
i.e., $u_1(t)(= |x_1(t)|^2 - |x_2(t)|^2)>0$, 
(resp. the back side, i.e.,  $u_1(t) <0$). 
The dotted arrow of label $2$ follows from symmetry of the solid arrow of label $1$ with respect to $\mathrm{B}_{1}$.
The remaining cases are treated in the same fashion.

%

\begin{figure}[htbp]
  \begin{minipage}[b]{0.45\linewidth}
\centering
\includegraphics[height=4.4cm]{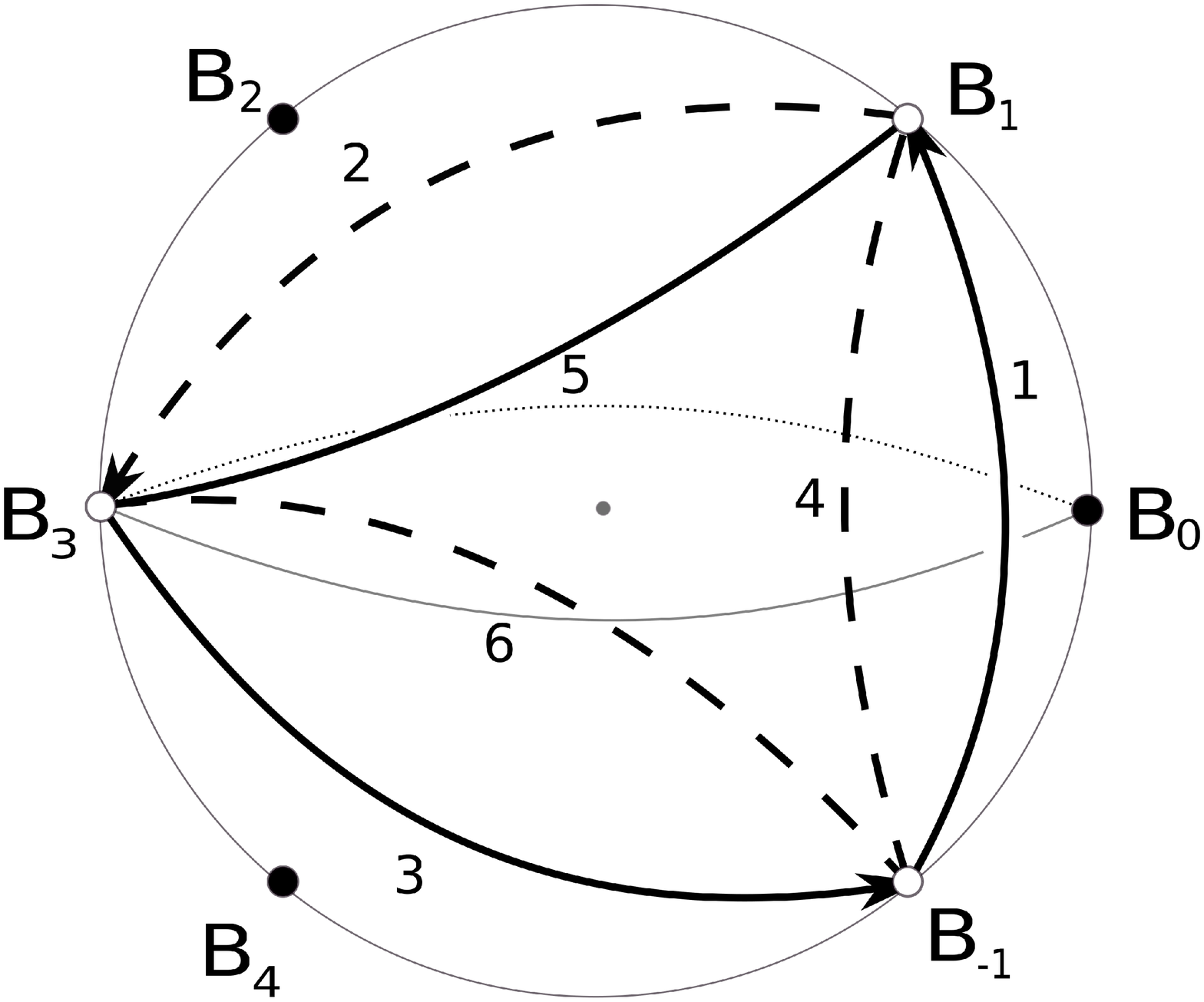}
\subcaption{$(n,p)=(3,1)$}
  \end{minipage}
  \begin{minipage}[b]{0.45\linewidth}
  \centering
\includegraphics[height=4.8cm]{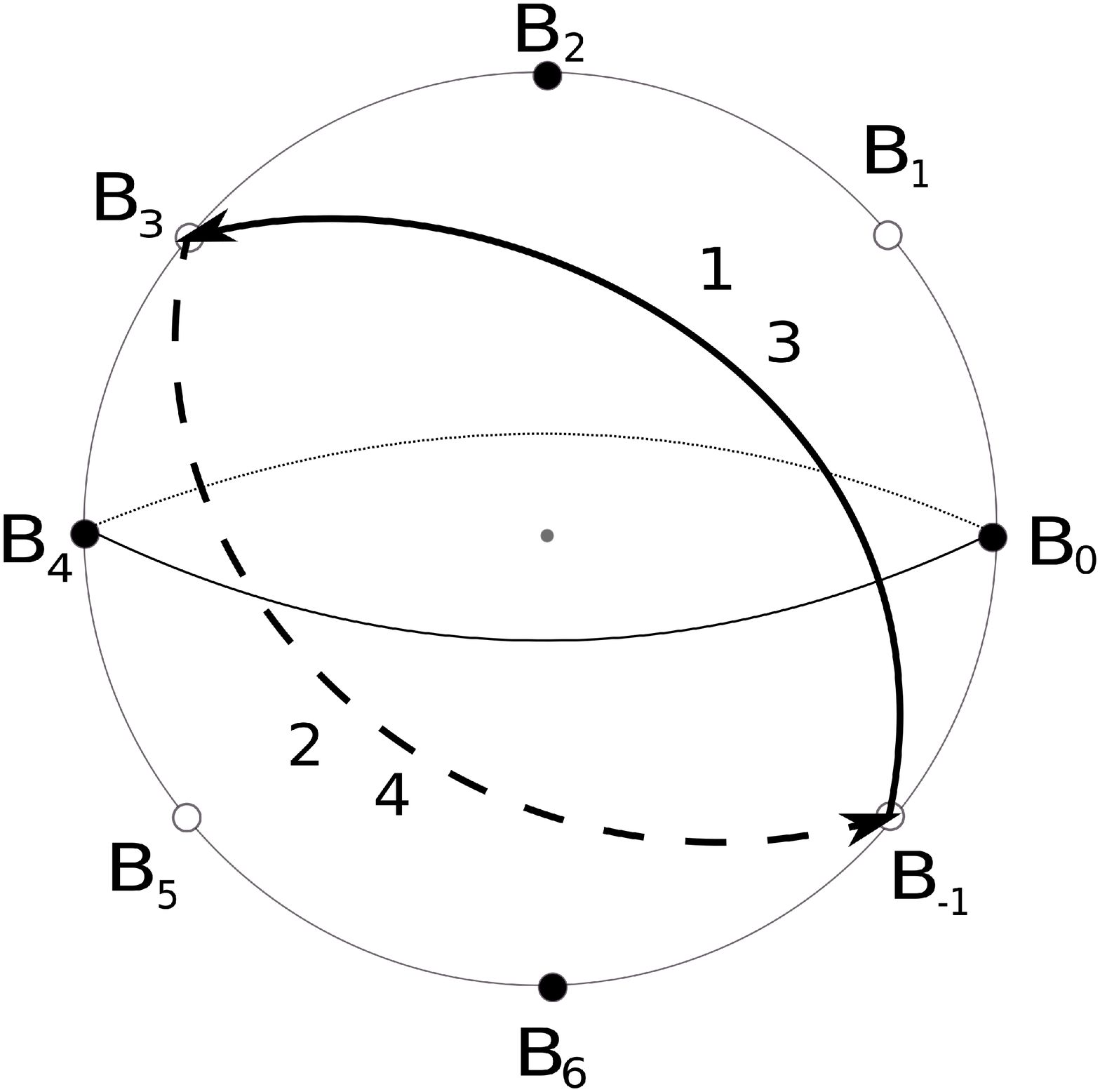}
\subcaption{$(n,p)=(4,2)$}
  \end{minipage}
  \caption{The shape curve ${\bm u}(t)$ for $t \in {\Bbb R}$ 
  when (1) $(n,p)= (3,1)$ and (2) $(n,p)= (4,2)$.  }
\quote{ The point $\mathrm{B}_i$ in the figure is the projection of the ray $\mathrm{B}_i$ onto the shape sphere.}
  \label{sphere}
  \end{figure}

  \begin{remark}
Though the set $\Lambda^G_{n,p}$ does not determine how the shape curve ${\bm u}(t)$ moves on the shape sphere for $t \in (0,\bar{T})$,
we can prove through variational arguments that it does not happen like Figure \ref{error}(1) or (2).
See \cite[Propositions $5$ and $6$]{Shibayama06} for the proof. 
 \end{remark}

 \begin{figure}[htbp]
  \begin{minipage}[b]{0.45\linewidth}
\centering
\includegraphics[height=4.8cm]{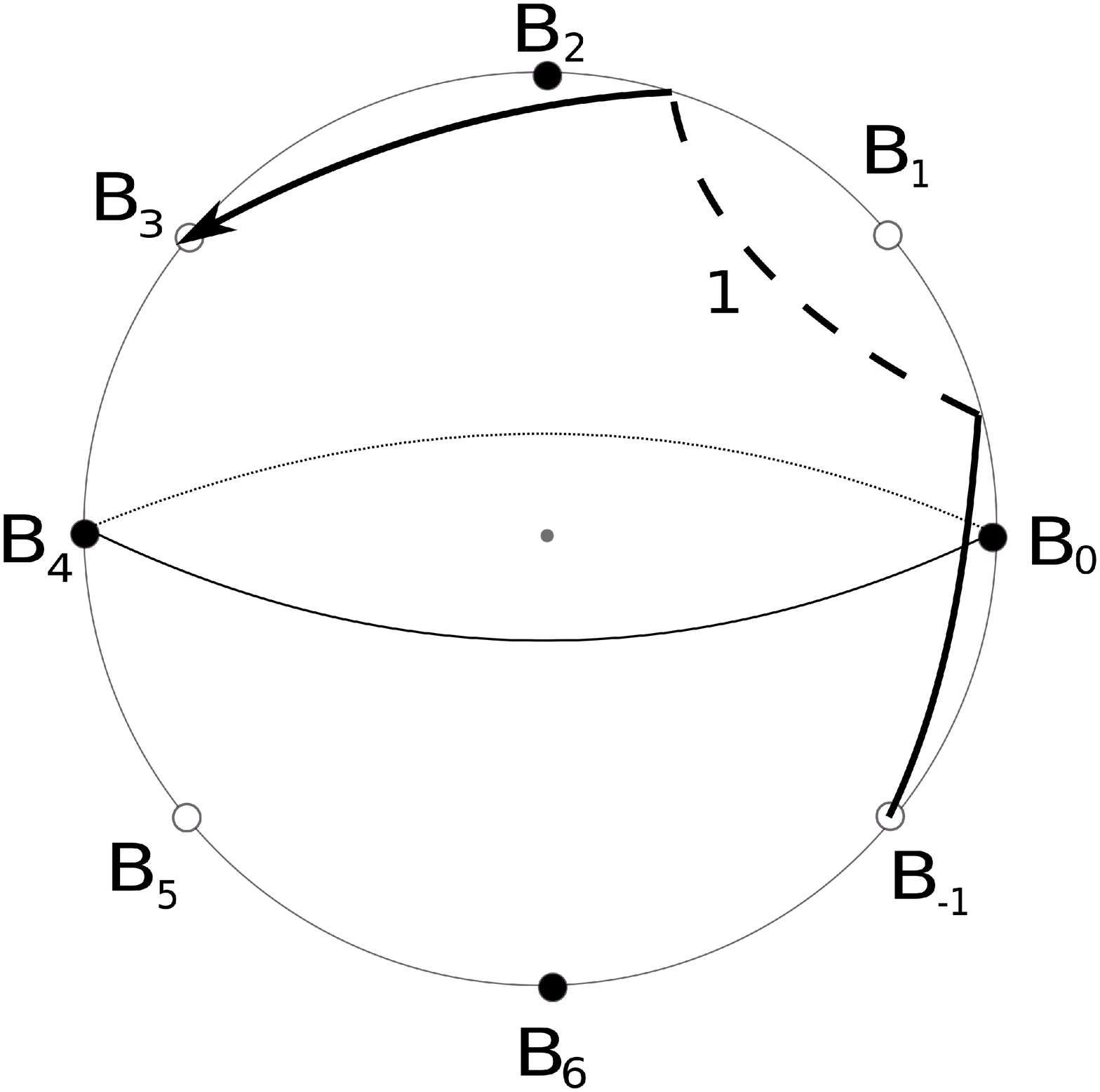}
\subcaption{Error type 1: Across the $u_2u_3$-plane}
  \end{minipage}
  \begin{minipage}[b]{0.45\linewidth}
  \centering
\includegraphics[height=4.8cm]{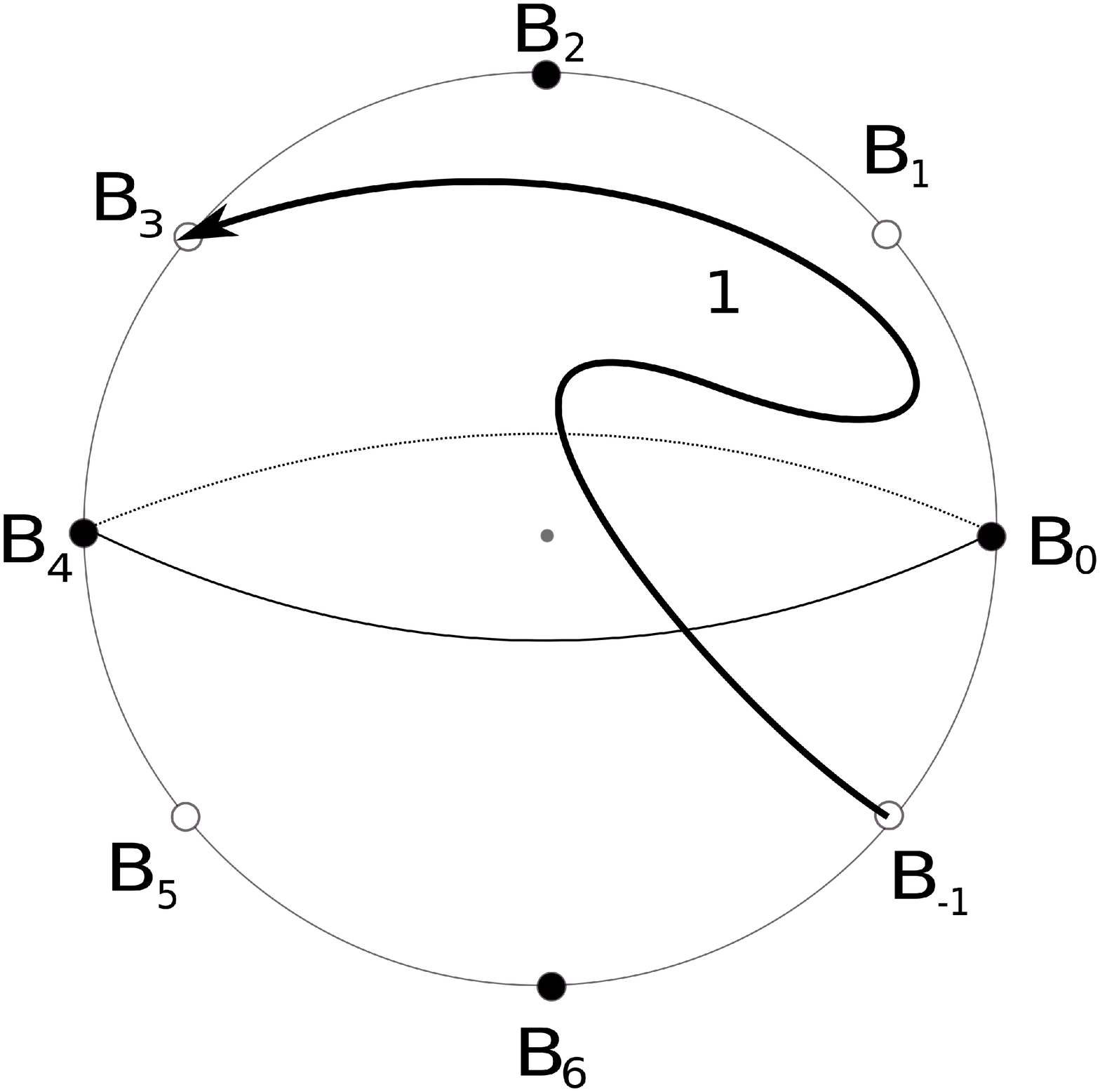}
\subcaption{Error type 2: Non-monotonicity}
  \end{minipage}
  \caption{Error types of the shape curve ${\bm u}(t)$ $(0 \le t \le \bar{T})$ 
  when $(n,p)=(4,2)$. 
  (1) Error type 1: ${\bm u}(t)$ $(0 \le t \le \bar{T})$ is across the $u_2u_3$-plane. 
  (2) Error type 2: ${\bm u}(t)$ is not monotone. 
  For the shape curve ${\bm u}(t)$ $(0 \le t \le \bar{T})$, 
  see the solid arrow with  label $1$ in Figure \ref{sphere}(2).}
  \label{error}
  \end{figure}

Figure \ref{diagram} illustrates 
the projection of the shape curve ${\bm u}(t)$   onto the $u_2u_3$-plane
together with the configuration of $8$ bodies corresponding to  each $\mathrm{B}_{i}$ 
when $(n,p)=(4,1)$ and $(4,2)$.
%

 \begin{figure}[htbp]
  \begin{minipage}[b]{0.99\linewidth}
\centering
\includegraphics[height=8.5cm]{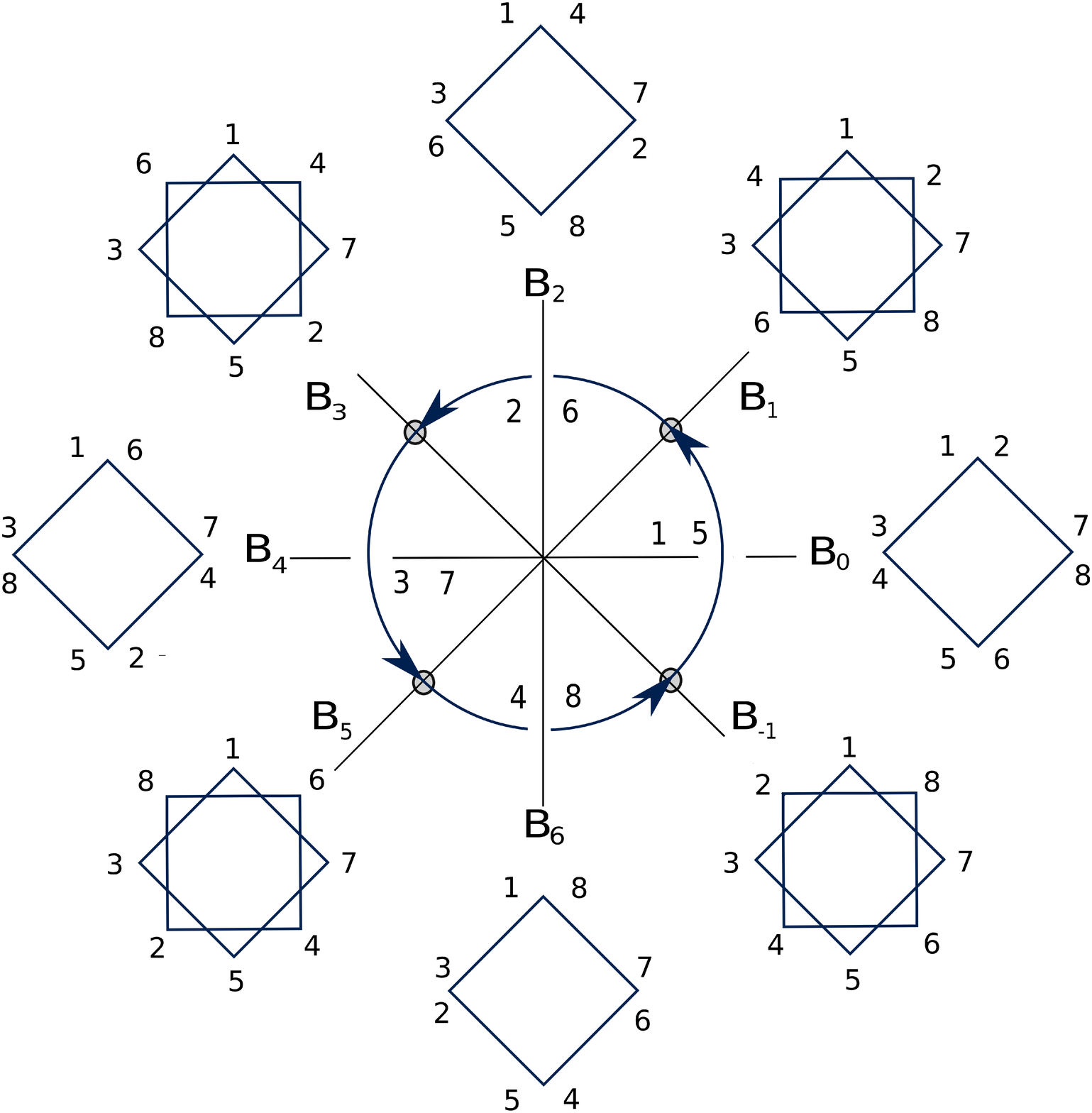}
\subcaption{$(n,p)=(4,1)$}
  \end{minipage}
  \begin{minipage}[b]{0.99\linewidth}
  \centering
\includegraphics[height=8.5cm]{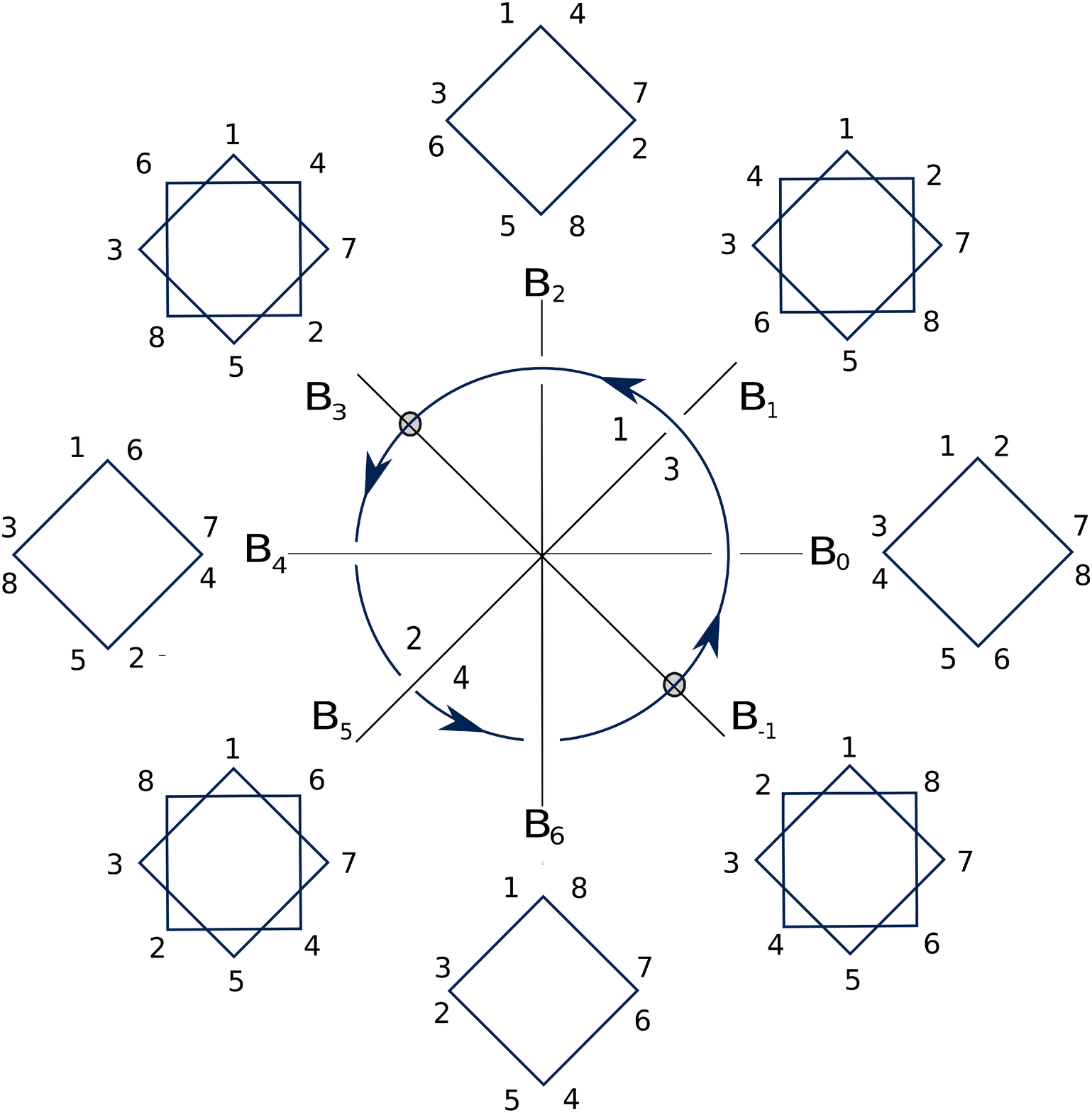}
\subcaption{$(n,p)=(4,2)$}
  \end{minipage}
  \caption{The projection of the shape curve ${\bm u}(t)$ for $t \in {\Bbb R}$ onto the $u_2u_3$-plane 
  when (1) $(n,p)= (4,1)$ and (2) $(n,p)= (4,2)$. 
   The configuration of $8$ bodies corresponding to  $\mathrm{B}_{i}$ is illustrated.}
  \label{diagram}
  \end{figure}

\section{Braid groups and mapping class groups} 
\label{section_braid-groups}

\subsection{Geometric braids} 
\label{subsection_geometric-braids}

\begin{center}
\begin{figure}[t]
\includegraphics[height=3.5cm]{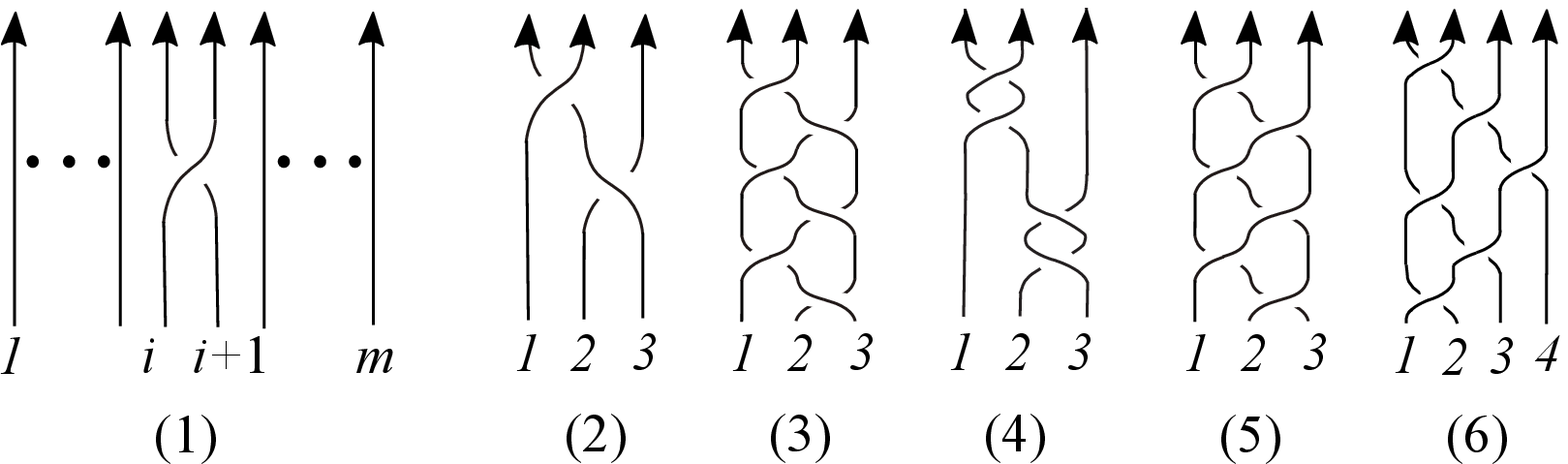}
\caption{(1) $\sigma_i \in B_m$. 
(2) $\sigma_1 \sigma_2^{-1} \in B_3$. 
(3) $(\sigma_1 \sigma_2^{-1})^3 \in P_ 3 < B_3$. 
(4) $ \sigma_1^2 \sigma_2^{-2} \in P_3< B_3$. 
(5) A full twist $\Delta^2 \in P_3< B_3$. 
(6) A half twist $\Delta \in B_4$.}
\label{fig_generator}
\end{figure}
\end{center}

In this section, we recall  definitions of (geometric) braids and the braid types. 
For the basics on braid groups, see Birman \cite{Birman74}. 
Let $D$ be a closed disk in the plane ${\Bbb R}^2$ and 
$Q_m= \{q_1, \dots, q_m\}$ be a set of $m$ points in the interior of $D$. 
Let $\gamma_1, \dots, \gamma_m$ be mutually disjoint $m$ arcs in $D \times [0,1]$ 
with the following properties. 
\begin{itemize}
\item 
$\partial (\gamma_1 \cup \dots \cup \gamma_m) = \{(q_1,t), \ldots, (q_m,t)\ |\ t \in \{0,1\}\} \subset D \times \{0,1\}$,

\item 
$\gamma_i$ ($i= 1, \dots, m$) starts at $(q_i,0)= \gamma_i \cap (D \times \{0\})$ and 
it goes up monotonically with respect to the $[0,1]$-factor. 
In particular 
$\gamma_i \cap (D \times \{t\})$ consists of a single point for $0 \le t \le 1$. 
\end{itemize}
We call $b= \gamma_1 \cup \dots \cup \gamma_m \subset D \times [0,1]$ 
a {\it (geometric) braid} {\it with base points} $Q_m$ and call each $ \gamma_i$ a {\it strand} of the braid $b$. 
We say that braids $b$ and $b'$ with base points $Q_m$ are {\it equivalent} 
if there is a $1$-parameter family of braids  with base points $Q_m$ 
deforming $b$ to $b'$. 
By abuse of notations, the equivalence class $[b]$ is also denote by $b$.

For braids $b$ and $b'$ with base points $Q_m$, 
the product $bb'$ is defined as follows. 
We first stuck $b$ on $b'$ and concatenate them to get disjoint arcs 
properly embedded in $D \times [0,2]$. 
By normalizing its height, 
we obtain a braid (in $D \times [0,1]$) with the same base points $Q_m$ and 
this is the braid $bb'$. 
The set of all braids with base points $Q_m$ 
with this product gives a group structure. 
The group is called the {\it (geometric) braid group} {\it with base points} $Q_m$ 
and it is denoted by $B(Q_m)$. 
Note that the identity element $1_{Q_m} \in B(Q_m)$ is given by a braid  consisting of straight arcs.

Let $A_m = \{a_1, \dots, a_m\}$ be a set of $m$ points in the interior of $D$ 
such that 
$a_1, \ldots, a_m$ lie on a segment in this order. 
We write $B_m= B(A_m)$ and call $B_m$ the {\it $m$-braid group}.
The isomorphism class of the above braid group $B(Q_m)$ with base points $Q_m$ does not depend on the location of base points, 
and $B(Q_m)$ is isomorphic to $B_m$. 
To define braid types of geometric braids with arbitrary base points $Q_m$,  
we now take an isomorphism between $B(Q_m)$ and $B_m$. 
We first choose an orientation preserving homeomorphism 
$f: D \rightarrow D$ such that 
$f(A_m)= Q_m$. 
Then take an isotopy 
$\{f_t\}_{0 \le t \le 1}$ on $D$ between the identity map $\mathrm{id}_{D}$ and $f$, i.e., $f_0= \mathrm{id}_{D}$ and $f_1= f$. 
We consider two kinds of mutually disjoint $m$ arcs 
$\gamma^+$ and $\gamma^{-}$ properly embedded in $D \times [0,1]$ as follows. 
\begin{eqnarray*}
\gamma^+ &=& \bigcup_{t \in [0,1]} \big\{(f_t(a_1), t), \dots, (f_t(a_m), t)\bigr\}, 
\\
\gamma^{-} &=& \bigcup_{t \in [0,1]} \bigl\{(f_{1-t}(a_1), t), \dots, (f_{1-t}(a_m), t)\bigr\}. 
\end{eqnarray*}
Note that 
$Q_m= f_1(A_m) =  \{f(a_1), \dots, f(a_m)\}$ and 
$A_m= f_0(A_m) = \{a_1, \dots, a_m\}$. 
Because of this, it makes sense to stack 
a  braid $b \in B(Q_m)$  on $\gamma^{+}$, and we obtain the resulting disjoint $m$ arcs 
$b \cdot \gamma^+ \subset D \times [0, 2]$.  
Then we stack $\gamma^{-}$ on  $b \cdot \gamma^+ $. 
As a result we have disjoint $m$ arcs 
$$\gamma^{-} \cdot b \cdot \gamma^+  \subset D \times [0,3].$$  
By normalizing the height of  the arcs, 
we obtain a braid (in $D \times [0,1]$) with base points $A_m$, and we still denote it by 
the same notation 
$\gamma^{-} \cdot b \cdot \gamma^+$. 
In particular if $b= 1_{Q_m} \in B(Q_m)$, 
then $\gamma^{-} 1_{Q_m} \gamma^+ = 1_{A_m} \in B_m$. 
The correspondence $b \mapsto \gamma^{-} \cdot b \cdot \gamma^+$ gives us an isomorphism between $B(Q_m)$ and $B_m$.

For an element $b \in B_m$, 
we put indices $1, \dots, m$ at the bottoms of strands  
so that the index $i$ indicates  $(a_i, 0) \in D \times \{0\}$.  
Let $\sigma_i$ be an element of $B_m$  as in Figure \ref{fig_generator}(1).  
The braid group  $B_m$ is  generated by 
$\sigma_1, \sigma_2, \dots, \sigma_{m-1}$, and it has  the following braid relations. 
\begin{enumerate}
\item[(B1)]
$\sigma_i \sigma_j= \sigma_j \sigma_i$  ($|i-j| \ge 2$). 

\item[(B2)] 
$\sigma_i \sigma_{i+1} \sigma_i = \sigma_{i+1} \sigma_i \sigma_{i+1}$  ($1 \le i \le m-2$). 
\end{enumerate}
%
See Figure \ref{fig_generator}(2)--(6) for some braids. 
There is a surjective homomorphism 
$$\hat{\sigma}: B_m \to \mathfrak{S}_m$$ 
from $B_m$ to the symmetry group $\mathfrak{S}_m$ of $m$ elements  
sending each $\sigma_j$ to the transposition $(j, j+1)$. 
The kernel of $\hat{\sigma}$ is called the {\it pure braid group} (or {\it colored braid group}) $P_m < B_m$. 
An element of $P_m$ is called a {\it pure braid}.  
See Figure \ref{fig_generator}(3)(4)(5) for some pure braids.

Let $Z(B_m)$ be the center of $B_m$ 
which is an  infinite cyclic group generated by a {\it full twist} $\Delta^2$, 
where a {\it half twist} 
$\Delta= \Delta_m \in B_m$ is given by 
$$\Delta_m= (\sigma_1 \sigma_2 \dots \sigma_{m-1}) (\sigma_1 \sigma_2 \dots \sigma_{m-2}) \dots (\sigma_1 \sigma_2) \sigma_1.$$ 
See Figure \ref{fig_generator}(6) for a half twist $\Delta \in B_4$. 

Given a braid $b \in B_m$, 
consider the projection $\overline{b}$ in the quotient group 
$$\mathcal{B}_{2n}=B_{2n}/Z(B_{2n}).$$
The {\it braid type} $\langle b \rangle$ of $b$ is  a conjugacy class of $\overline{b}$ in $\mathcal{B}_{2n}$. 

In the case of the braid group $B(Q_m)$ with base points $Q_m$, 
the {\it braid type} $\langle b \rangle$ of $b \in B(Q_m)$  is defined by 
the braid type $\big\langle \gamma^{-} \cdot b \cdot \gamma^+ \big\rangle$ 
of  the braid $\gamma^{-} \cdot b \cdot \gamma^+ \in B_m$ (with base points $A_m$), 
where $\gamma^+$ and $\gamma^-$ are arcs as above. 
The braid type $\langle b \rangle$ is well-defined, i.e., 
it does not depend on the  above orientation preserving homeomorphism 
$f: D \rightarrow D$ and the isotopy $\{f_t\}_{0 \le t \le 1}$.

\begin{ex}
\ 
\begin{enumerate}
\item[(1)]
For the $3$-braid $\sigma_1 \sigma_2^{-1}$, it follows that 
$$\hat{\sigma}(\sigma_1 \sigma_2^{-1}) = \hat{\sigma}(\sigma_1) \hat{\sigma}(\sigma_2^{-1}) = (12)(23) = (123) \in \mathfrak{S}_3,$$ 
see Figure \ref{fig_generator}(2). 
Hence $\hat{\sigma}((\sigma_1 \sigma_2^{-1})^3) = 1 \in \mathfrak{S}_3$ which means that 
$(\sigma_1 \sigma_2^{-1})^3 \in P_3$. 

\item[(2)] 
For the $3$-braid $\sigma_1^2 \sigma_2^{-2}$, it follows that 
$$\hat{\sigma}(\sigma_1^2 \sigma_2^{-2}) = \hat{\sigma}(\sigma_1^2) \hat{\sigma}(\sigma_2^{-2}) = 1 \cdot 1 = 1 \in \mathfrak{S}_3,$$  
see Figure \ref{fig_generator}(4). 
Hence $\sigma_1^2 \sigma_2^{-2}  \in P_3$. 
\end{enumerate}
\end{ex}

\begin{ex}
\label{ex_trivial-braid}
For the Euler's periodic solution of the planar $3$-body problem, 
 three bodies are collinear at every instant. 
A full twist $\Delta^2 = (\sigma_1 \sigma_2 \sigma_1)^2 = (\sigma_1 \sigma_2)^3 \in B_3$ (Figure \ref{fig_generator}(5)) 
represents the braid type of the solution. 
Since $Z(B_3)$ is generated by $\Delta^2$, 
the braid type of the Euler's periodic solution is trivial. 
Similarly, it is the trivial  braid type for the Lagrange's periodic solution of the planar $3$-body problem, 
since the triangle formed by the three bodies is equilateral for all time. 
\end{ex}

\subsection{Mapping class groups}

Let $X_1, \dots, X_n$ be possibly empty subspaces of an orientable manifold $M$. 
For instance 
$M$ is a connected orientable surface $\Sigma_{g,m}$ of genus $g \ge 0$ with $m$ punctures (possibly $m = 0$) 
and 
$X_i$ ($i= 1, \dots, n$) is a finite set in $\Sigma_{g,m}$. 
Let $\mathrm{Homeo}_+ (M, X_1, \dots, X_n)$ be the group of orientation-preserving 
self-homeomorphisms of $M$ that map $X_i$ onto $X_i$ for each $i= 1, \dots, n$. 
We do not require that 
homeomorphisms fix the boundary $\partial M$ pointwise. 
The mapping class group 
$\mathrm{MCG}(M, X_1, \dots, X_n)$ is defined by 
$$\mathrm{MCG}(M, X_1, \dots, X_n)= \pi_0(\mathrm{Homeo}_+ (M, X_1, \dots, X_n)),$$
that is the group of isotopy classes of elements of $\mathrm{Homeo}_+ (M, X_1, \dots, X_n)$. 
When $X$ is an empty subspace of $M$, then we write 
$\mathrm{MCG}(M)= \mathrm{MCG}(M,X)$. 
We apply elements of mapping class groups from right to left, i.e., 
we apply $g$ first for the product $fg$.

\begin{center}
\begin{figure}[t]
\includegraphics[height=3cm]{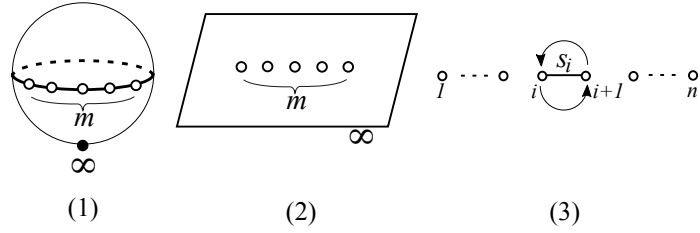}
\caption{
(1) A pair $(\Sigma_{0,m}, \{\infty\})$.  
(2) An $m$-punctured plane. 
(3) A half twist $h_i$.}
\label{fig_plane}
\end{figure}
\end{center}

Let $D_m= D \setminus A_m$ be an $m$-punctured disk, 
where $A_m = \{a_1, \dots, a_m\}$ is the set of $m$ points in the interior of $D$ 
as in Section \ref{subsection_geometric-braids}. 
By definition, 
$\mathrm{MCG}(D_m) $ is the group of isotopy classes of elements of $\mathrm{Homeo}_+ (D_m)$ 
which fix $\partial D$ setwise. 
In this paper, we mainly consider an $m$-punctured disk $D_m$ or 
an $m$-punctured sphere $\Sigma_{0,m}$ 
as an orientable manifold $M$ for the  mapping class groups. 
We take a point in $\Sigma_{0,m}$ and call it  $\infty$. 
An element $f \in \mathrm{Homeo}_+(\Sigma_{0,m}, \{\infty\}) $ means that $f$ fixes the point $\infty$. 
Puncturing the point $\infty$, 
we think of $\mathrm{MCG}(\Sigma_{0,m}, \{\infty\})$ 
as a subgroup of $\mathrm{MCG}(\Sigma_{0,m+1})$. 
Also we may regard $\mathrm{MCG}(\Sigma_{0,m}, \{\infty\})$ as the mapping class group of an $m$-punctured plane. 
See Figure \ref{fig_plane}(1)(2).

The mapping class group $\mathrm{MCG}(D_m)$ is generated by 
$h_1, \dots, h_{m-1}$, 
where $h_i $ is the right-handed {\it half twist} about a segment $s_i$ connecting the $i$th and $(i+1)$th punctures, 
see Figure \ref{fig_plane}(3). 
More precisely, 
let ${\Bbb D}_i \subset \mathrm{int}(D)$ be a closed disk such that 
${\Bbb D}_i$ contains  the two points $a_i$ and $a_ {i+1}$ together with a segment $s_i$ between the punctures 
$a_i$ and $a_{i+1}$. 
Moreover ${\Bbb D}_i$ contains no other points of $A_m$. 
Then the right-handed  half-twist $h_i \in  \mathrm{MCG}(D_m)$ 
is a mapping class that fixes the exterior of ${\Bbb D}_i$ and 
rotates  $s_i$ in ${\Bbb D}_i$ by $\pi$ in the counter-clockwise direction. 
Hence $h_i$ interchanges the $i$th puncture  with the $(i+1)$th puncture.

We now recall a relation between $B_ m$ and $\mathrm{MCG}(D_m)$. 
There is a surjective homomorphism 
$$\Gamma: B_m \rightarrow \mathrm{MCG}(D_m)$$  
which sends $\sigma_i$ to $h_i$ for $i = 1, \dots, m-1$. 
The kernel of $\Gamma$ is the center $Z(B_m)$, and hence 
$\mathcal{B}_m = B_m / Z(B_m)$ is isomorphic to  $\mathrm{MCG}(D_m)$.  
Collapsing $\partial D$ to the point $\infty$ in the sphere, 
we  have a homomorphism 
$$\mathfrak{c}: \mathrm{MCG}(D_n) \rightarrow \mathrm{MCG}(\Sigma_{0, m}, \{\infty\}).$$ 
By abuse of notations, we simply denote by $b $, 
the mapping class 
$\mathfrak{c}(\Gamma(b)) \in \mathrm{MCG}(\Sigma_{0,m}, \{\infty\})$. 
Also we denote by $\langle b \rangle$, 
the conjugacy class $\big\langle \mathfrak{c} (\Gamma(b)) \big\rangle$ of 
$\mathfrak{c}(\Gamma(b)) \in \mathrm{MCG}(\Sigma_{0,m}, \{\infty\})$. 
Note that this notation $\langle b \rangle$ is the same as the braid type of $b \in B_m$.

\subsection{Nielsen-Thurston classification}
\label{subsection_Nielsen-Thurston}

According to the Nielsen-Thurston classification \cite{Thurston88}, 
mapping classes fall into three types: periodic, reducible and pseudo-Anosov. 
Assume that $3g-3+m \ge 1$. 
A mapping class $\phi \in \mathrm{MCG}(\Sigma_{g,m})$ is {\it periodic}  
if $\phi$ is of finite order. 
A mapping class 
$\phi \in \mathrm{MCG}(\Sigma_{g,m})$ is {\it reducible} 
if there is a  collection of mutually disjoint and non-homotopic essential simple closed curves 
$C_1, \ldots, C_j$ in $\Sigma_{g,m}$ for $j \ge 1$ such that 
$C_1 \cup \dots \cup C_j$ is preserved by $\phi$. 
Here a simple closed curve $C$ in $\Sigma_{g,m}$ is {\it essential} 
if each component of $\Sigma_{g,m} \setminus C$ has negative Euler characteristic.  
(There is a mapping class that is periodic and reducible.)  
A mapping class $\phi \in \mathrm{MCG}(\Sigma_{g,m})$ is {\it pseudo-Anosov} 
if $\phi$ is neither periodic nor reducible.  
Note that the Nielsen-Thurston type is a conjugacy invariant, 
i.e., 
two mapping classes are conjugate to each other in $\mathrm{MCG}(\Sigma_{g,m})$, 
then their Nielsen-Thurston types are the same.

Pseudo-Anosov mapping classes have many important properties for the study of dynamical systems. 
For more details which we describe below, see \cite{FLP,FarbMargalit12}. 
A homeomorphism $\Phi: \Sigma_{g,m} \rightarrow \Sigma_{g,m}$ is {\it pseudo-Anosov} 
if there exist a constant $\lambda= \lambda(\Phi)>1$ and 
a pair of transverse measured foliations 
$(\mathcal{F}^+, \mu^+)$ and 
$(\mathcal{F}^-, \mu^-)$  such that 
$$\Phi((\mathcal{F}^+,  \mu^+)) = (\mathcal{F}^+, \lambda \mu^+)\ \hspace{2mm} and \hspace{2mm}
 \Phi((\mathcal{F}^-,  \mu^-)) = (\mathcal{F}^-, \tfrac{1}{\lambda} \mu^-).$$ 
 This means that 
 $\Phi$ preserves both foliations $\mathcal{F}^+$ and $\mathcal{F}^-$, 
 and it contracts the leaves of $\mathcal{F}^-$ by $\frac{1}{\lambda}$ and 
 it expands the leaves of $\mathcal{F}^+$ by $\lambda$. 
The invariant foliations  $\mathcal{F}^+$ and $\mathcal{F}^-$ 
are called the {\it unstable and stable foliations} for $\Phi$, and $\lambda>1$ is called the {\it stretch factor} for $\Phi$. 

\begin{remark}
\label{rem_singularities}
The invariant foliations $\mathcal{F}^+$ and $\mathcal{F}^-$ for the pseudo-Anosov homeomorphism $\Phi$  
are singular foliations 
which mean that 
they have common singularities in the interior of $\Sigma_ {g,m}$ or at punctures of $\Sigma_ {g,m}$. 
The number of singularities is finite. 
A $1$-pronged singularity may occur at a puncture of $\Sigma_{g,m}$, yet 
there are no $1$-pronged singularities in the interior of $\Sigma_ {g,m}$. 
\end{remark}

Each pseudo-Anosov mapping class $\phi \in  \mathrm{MCG}(\Sigma_{g,m})$ 
contains a pseudo-Anosov homeomorphism $\Phi$ as a representative of $\phi$.  
We set $\lambda(\phi)= \lambda(\Phi) $ and call it the {\it stretch factor} of the mapping class $\phi = [\Phi]$.  
The stretch factor $\lambda(\phi)$ is a conjugacy invariant of pseudo-Anosov mapping classes. 
Moreover $ \lambda(\phi)$ is the largest eigenvalue of a Perron-Frobenius integral matrix. 
Thus $\lambda(\phi)$ is an algebraic integer which is a real number grater than $1$ 
and $|\lambda'| < \lambda(\phi)$ holds for each conjugate element $\lambda' \ne \lambda(\phi)$. 
The logarithm  $\log (\lambda(\phi))$ of the stretch factor $\lambda(\phi)$ is called the {\it entropy} of $\phi$.

\begin{remark}
\label{remark_power}
If $\phi \in \mathrm{MCG}(\Sigma_{g,m})$ is pseudo-Anosov, then $\phi^k$ is pseudo-Anosov for all $k \ge 1$ and 
the equality $\lambda(\phi^k)= (\lambda(\phi))^k$ holds. 
\end{remark}

Recall the  two homomorphisms 
$\Gamma: B_m \rightarrow \mathrm{MCG}(D_m)$ and 
$\mathfrak{c}: \mathrm{MCG}(D_m) \rightarrow \mathrm{MCG}(\Sigma_{0, m}, \{\infty\}) < \mathrm{MCG}(\Sigma_{0, m+1})$.  
We say that a braid $b \in B_m$ is {\it periodic} (resp. {\it reducible}, {\it pseudo-Anosov}) 
if the mapping class $\mathfrak{c}(\Gamma(b))$ is of the corresponding type. 
When $b$ is a pseudo-Anosov braid, 
the {\it stretch factor} $\lambda(b)$ of $b$ is defined by the stretch factor $\lambda(\mathfrak{c}(\Gamma(b)))$ 
of the mapping class $\mathfrak{c}(\Gamma(b))$. 
In this case, it makes sense to say that the braid type $\langle b \rangle$ is {\it pseudo-Anosov}, 
and we can define the {\it stretch factor}  
$\lambda(\langle b \rangle)$ of the braid type $\langle b \rangle$  by 
\begin{equation}
\label{equation_braid-type}
\lambda(\langle b \rangle)= \lambda(b) = \lambda(\mathfrak{c}(\Gamma(b))),
\end{equation}
since both Nielsen-Thurston type and the stretch factor are conjugacy invariants.

\subsection{Pseudo-Anosov $3$-braids}
\label{subsection_3braids}

It is well-known that 
for positive integers $k_j$'s, $\ell_j$'s and $r$,  
the $3$-braid 
$$\sigma_1^{k_1} \sigma_2^{-\ell_1} \dots \sigma_1^{k_r} \sigma_2^{-\ell_r}$$ 
is pseudo-Anosov. 
Moreover any pseudo-Anosov $3$-braid $\alpha$ 
is conjugate to a braid 
$\sigma_1^{k_1} \sigma_2^{-\ell_1} \dots \sigma_1^{k_r} \sigma_2^{-\ell_r}$ in $B_3$ 
which is unique up to a cyclic permutation. 
See Murasugi \cite{Murasugi74} for example. 
Then the stretch factor $\lambda(\alpha)$ is 
the eigenvalue greater than $1$ of

\begin{equation}
\label{equation_eigenvalue}
M_{(k_1, \ell_1, \ldots, k_r, \ell_r)}=  
\left(
\begin{matrix}
  1&1\\
  0&1
\end{matrix}
\right)^{k_1} 
\left(
\begin{matrix}
  1&0\\
  1&1
\end{matrix}
\right)^{\ell_1}  
\dots 
\left(
\begin{matrix}
  1&1\\
  0&1
\end{matrix}
\right)^{k_r} 
\left(
\begin{matrix}
  1&0\\
  1&1
\end{matrix}
\right)^{\ell_r}. 
\end{equation}
See  Handel \cite{Handel97} for example.

\begin{ex}[Metallic $3$-braids (Appendix A in \cite{FinnThiffeault11})]
\label{ex_metallic-3braid}
For $p \ge 1$, 
the $3$-braid $ \sigma_1^{2p} \sigma_2^{-2p}$  is pseudo-Anosov, and 
the stretch factor $\lambda(\sigma_1^{2p} \sigma_2^{-2p})$ is the  eigenvalue greater than $1$ of 
$M_{(2p, 2p)} = 
\left(
\begin{matrix}
  1+ 4p^2&2p\\
  2p&1
\end{matrix}
\right).$
Thus 
$$\lambda( \sigma_1^{2p} \sigma_2^{-2p})  = (p+ \sqrt{p^2+1})^2 = \Bigl( \frac{1}{2} (2p+ \sqrt{4p^2+4}) \Bigr)^2
= (\mathfrak{s}_{2p})^2.$$
\end{ex}

\
\begin{center}
\begin{figure}[t]
\includegraphics[height=5cm]{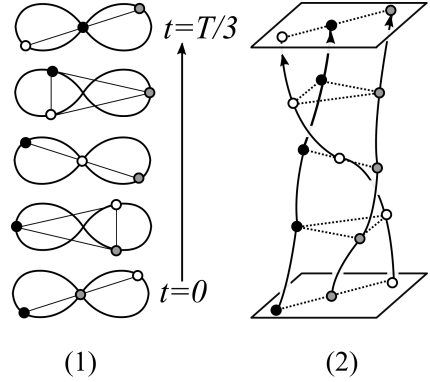}
\caption{(1) The figure-$8$ solution ${\bm x}(t)$ with period $T$. 
(2) A representative braid $\sigma_1^{-1} \sigma_2 \in \big\langle  b({\bm x}(t), [0, \frac{T}{3}])\big\rangle$.}
\label{fig_figure8}
\end{figure}
\end{center}

\begin{ex}
\label{ex_figure-eight}
Let us consider the figure-$8$ solution ${\bm x}(t)= (x_1(t), x_2(t), x_3(t))$  
 by Moore \cite{Moore93} and Chenciner-Montgomery \cite{ChencinerMontgomery00}, 
 see Figure \ref{fig_figure8}. 
The periodic solution ${\bm x}(t)$ has a property such that 
$$ x_1(t+ \tfrac{T}{3}) = x_2(t), \hspace{2mm} 
x_2(t+ \tfrac{T}{3}) = x_3(t), \hspace{2mm} 
x_3(t+ \tfrac{T}{3}) = x_1(t),$$
where $T>0$ is the period of ${\bm x}(t)$. 
This property tells us that 
${\bm x}(t)$ determines a braid $b({\bm x}(t), [0, \frac{T}{3}])$. 
One sees that $\sigma_1^{-1} \sigma_2 \in B_3$ is a representative of  $\big\langle  b({\bm x}(t), [0, \frac{T}{3}])\big\rangle$ and 
$(\sigma_1^{-1} \sigma_2)^3 $ represents the braid type $\big\langle  b({\bm x}(t), [0, T])\big\rangle$ of the solution ${\bm x}(t)$. 
It is easy to see that 
$\sigma_1^{-1} \sigma_2$ is conjugate with $\sigma_1 \sigma_2^{-1}$ in $B_3$. 
By \eqref{equation_eigenvalue}, $\sigma_1 \sigma_2^{-1}$ is a pseudo-Anosov braid 
with the stretch factor $(\mathfrak{s}_1)^2$. 
Thus the braid type of the figure-$8$ solution is pseudo-Anosov with the stretch factor $(\mathfrak{s}_1)^6$ 
(Remark \ref{remark_power}), 
and hence it is a non-trivial braid type in contrast with the Euler's solution and Lagrange's solution (Example \ref{ex_trivial-braid}). 
\end{ex}

\section{Proof of Theorem \ref{thm_braidtype}} 
\label{section_proof-of-theorem}

For $n \ge 2$ and $p \ge 1$, we define  braids $u_n, v_n, \beta_{n,p} \in B_ {2n}$ as follows. 
\begin{eqnarray*}
u_{n} &=&
(\sigma_1 \sigma_2 \dots \sigma_{2n-1}) (\sigma_1 \sigma_3 \dots \sigma_{2n-1})^{-1}, 
\\
v_{n}&=& (\sigma_1 \sigma_2 \dots \sigma_{2n-1})^{-1}
(\sigma_1 \sigma_3 \dots \sigma_{2n-1}) \ \mbox{and}\ 
\\
\beta_{n,p}&=& u_{n}^p v_{n}^p.
\end{eqnarray*}
See also Figure \ref{fig_uv} together with the braid relation (B1) in Section \ref{subsection_geometric-braids}. 
It is easy to check that 
$\hat{\sigma}(u_n)= (1,3,\dots, 2n-1)$ and 
$\hat{\sigma}(v_n)= (2,4, \dots, 2n)^{-1}$. 
Hence by (\ref{equation_permutation}), we have 
\begin{equation}
\label{equation_permutation-b}
\hat{\sigma}(\beta_{n,p}) = \hat{\sigma}(u_n^p v_n^p) = (1,3,\dots, 2n-1)^p (2,4, \dots, 2n)^{-p} = 
\sigma(h_{n,p}^{-1}). 
\end{equation}

\begin{center}
\begin{figure}[t]
\includegraphics[height=6cm]{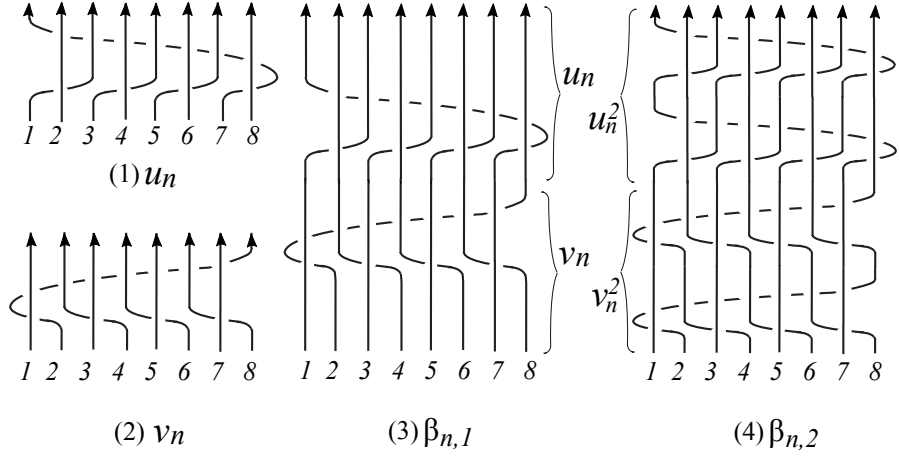}
\caption{
Case $n=4$. 
(1) $u_n$. 
(2) $v_n$. 
(3) $\beta_{n,1}= u_n v_n$. 
(4) $\beta_{n,2} = u_n^2 v_n^2 $.} 
\label{fig_uv}
\end{figure}
\end{center}

\begin{proof}[Proof of Theorem \ref{thm_braidtype}] 
The proof consists of the following two steps. 
In Step 1, we prove that 
for  $n \ge 2$ and {\it any} $p \ge 1$, 
the braid $\beta_{n,p}$ is pseudo-Anosov 
with $ \lambda(\beta_{n,p})=(\mathfrak{s}_{2p})^{2}$. 
(We have no restriction on $p$ in Step 1.) 
In Step 2, 
we prove that 
for any $n \ge 2$ and $p \in \{1, \dots, \lfloor  \frac{n}{2}\rfloor \}$, 
the braid types of $\beta_{n,p}$ and $y_{n,p} = b({\bm x}_{n,p}(t), [0, \frac{d}{n}])$ are the same. 
In other words, $\beta_{n,p} \in \langle y_{n,p} \rangle$. 
Since $X_{n,p}= \big\langle (y_{n,p})^{\frac{n}{d} } \big\rangle$, it follows that 
$X_{n,p}= \big\langle (\beta_{n,p})^{\frac{n}{d} } \big\rangle$. 
Hence by Step 1 together with Remark \ref{remark_power},  
$X_{n,p}$ is a pseudo-Anosov braid type with the stretch factor 
$$\lambda(X_{n,p})= \lambda \bigr((\beta_{n,p}) ^{\frac{n}{d}} \bigl) = 
(\lambda(\beta_{n,p}))^{\frac{n}{d} } = (\mathfrak{s}_{2p})^{\frac{2n}{d}}.$$

\noindent
{\bf Step 1.}
For $n \ge 2$ and $p \ge 1$, 
the braid $\beta_{n,p}$ is pseudo-Anosov 
with $ \lambda(\beta_{n,p})=(\mathfrak{s}_{2p})^{2}$. 
In particular 
$\lambda(\beta_{n,p}) < \lambda(\beta_{n, p'})$ if $p < p'$.

\begin{center}
\begin{figure}[t]
\includegraphics[height=6cm]{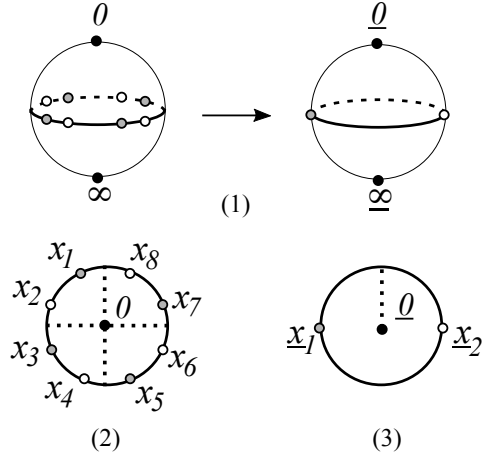}
\caption{Case $n=4$. 
(1) An $n$-fold branched cover 
$\mathfrak{p}: \Sigma_{0,2n} \rightarrow \Sigma_{0,2}$ with branched points $\underline{0}$ and $\underline{\infty}$, 
where punctures   lie on the equators. 
(2) The upper hemisphere for $\Sigma_{0,2n}$. 
(3) The upper hemisphere for $\Sigma_{0,2}$. } 
\label{fig_cover}
\end{figure}
\end{center}

\begin{center}
\begin{figure}[t]
\includegraphics[height=8cm]{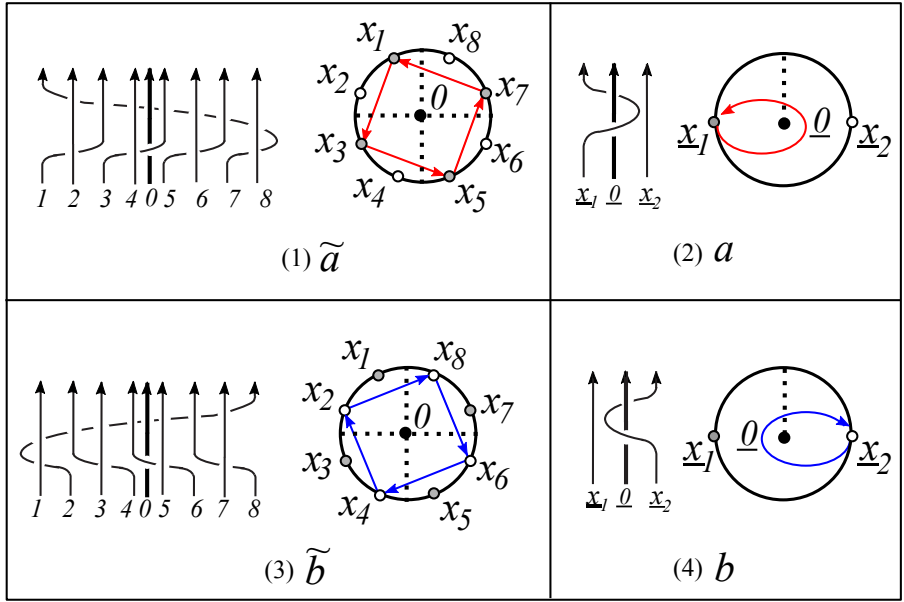}
\caption{
Case $n=4$. 
(1) A lift $\widetilde{a}= \widetilde{a_{n}} \in \mathrm{MCG}(\Sigma_{0,2n}, \{0\}, \{\infty\})$ of 
(2) $a = \sigma_1^2 \in \mathrm{MCG}(\Sigma_{0,2}, \{\underline{0}\}, \{\underline{\infty}\})$. 
(3) A lift $\widetilde{b}= \widetilde{b_{n}} \in  \mathrm{MCG}(\Sigma_{0,2n}, \{0\}, \{\infty\})$ of 
(4) $b= \sigma_2^{-2}  \in \mathrm{MCG}(\Sigma_{0,2}, \{\underline{0}\}, \{\underline{\infty}\})$. 
For (2) and (4), 
$a$ and $b$ (as elements of $B_3$) have base points 
$\underline{x}_1$, $\underline{0}$, and $\underline{x}_2$.} 
\label{fig_11-2-2}
\end{figure}
\end{center}

\begin{center}
\begin{figure}[t]
\includegraphics[height=3.5cm]{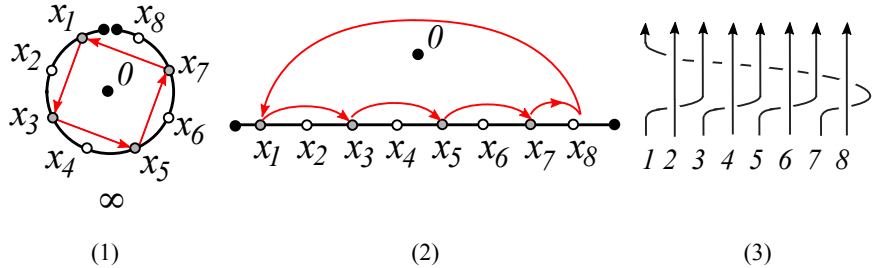}
\caption{
Case $n=4$. 
(1) An arc in the equator. 
(2) A segment in the plane. 
(3) The braid $u_n$  corresponding to $\widetilde{a}^{\bullet} \in \mathrm{MCG}(\Sigma_{0,2n}, \{\infty\})$. 
(The arc in (1) is identified with the segment in (2).) 
For (1) and (2), arrows indicate the image of the punctures under $\widetilde{a}$.} 
\label{fig_equator}
\end{figure}
\end{center}

\noindent
Proof of Step 1. 
We consider a $2$-punctured sphere $\Sigma_{0,2}$ and 
denote the two punctures of $\Sigma_{0,2}$ 
by $\underline{x}_1$ and $\underline{x}_2$. 
We pick two points in $\Sigma_{0,2}$ and call them 
$\underline{0}$ (the north pole) and $\underline{\infty}$ (the south pole). 
Given $n \ge 2$, we take an $n$-fold branched cover 
$$\mathfrak{p}: \Sigma_{0,2n} \rightarrow \Sigma_{0,2}$$ 
with branched points $\underline{0}$ and $\underline{\infty}$. 
(We cut a longitude of $\Sigma_ {0,2}$ between $\underline{0}$ and $\underline{\infty}$, 
take $n$ copies of the resulting surface, and past them to make an $2n$-punctured sphere.) 
We denote 
lifts of  $\underline{0}, \underline{\infty} \in \Sigma_{0,2}$ 
by $0, \infty \in \Sigma_{0,2n}$ respectively.
Let $x_1, x_2, \dots, x_{2n}$ be punctures of $\Sigma_{0,2n}$ 
such that 
$\mathfrak{p}$ sends $x_{2k}$ (resp. $x_{2k-1}$) to $\underline{x}_2$ (resp. $\underline{x}_1$).  
In the view from $0 \in \Sigma_{0,2n}$ in the upper hemisphere,  
we may assume that 
$x_1, \dots, x_{2n}$ lie on the equator counterclockwise and 
these $2n$ punctures form the regular $2n$-gon. 
See Figure \ref{fig_cover}.

Let $a= \sigma_1^2$, $b= \sigma_2^{-2} \in B_3$. 
Since $a$ and $b$ are pure $3$-braids,  
we can regard $a$ and $b$ as elements of $\mathrm{MCG}(\Sigma_{0,2}, \{\underline{0}\}, \{\underline{\infty}\})$, 
see Figure \ref{fig_11-2-2}(2)(4). 
We lift $a$ and $b$ to $\Sigma_{0, 2n}$, 
and call them 
$$\widetilde{a}, \widetilde{b} \in \mathrm{MCG}(\Sigma_{0,2n}, \{0\}, \{\infty\}) 
< \mathrm{MCG}(\Sigma_{0,2n+1}, \{\infty\}).$$
(Clearly  both $\widetilde{a}$ and $\widetilde{b}$ fix the two points $0$ and $\infty$.)  
We have 
\begin{eqnarray*}
\widetilde{a}(x_{2k-1}) &=& x_{2k+1}\hspace{2mm} \mbox{and}\hspace{2mm}\  
\widetilde{a}(x_{2k}) = x_{2k} \hspace{2mm} 
\mbox{for}\  k= 1, \dots, n, 
\\
\widetilde{b}(x_{2k-1}) &=& x_{2k-1}\hspace{2mm} \mbox{and}\hspace{2mm}\  
\widetilde{b}(x_{2k}) = x_{2k-2} \hspace{2mm} 
\mbox{for}\  k= 1, \dots, n, 
\end{eqnarray*}
where we interpret indices modulo $2n$. 
Notice that 
$\widetilde{a}$ rotates the regular $n$-gon $x_1 x_3 \dots x_ {2n-1}$ 
 by $\frac{\pi}{n}$ counterclockwise about the north pole $0$; 
$\widetilde{b}$ rotates the regular $n$-gon $x_2 x_4 \dots x_ {2n}$ 
 by $\frac{\pi}{n}$ clockwise about the same point $0$, 
see Figure \ref{fig_11-2-2}(1)(3). 
In other words, under the action of $\widetilde{a}$, 
each puncture $x_{2i-1}$ $(i= 1, \dots, 2n)$ with odd index 
is passing through in front of the puncture $x_{2i}$ with even index 
from the view of the north pole $0 \in \Sigma_ {0,2n}$. 
Similarly, under the action of $\widetilde{b}$, 
each puncture $x_{2i}$ $(i= 1, \dots, 2n)$ with even index 
is passing through in front of the puncture $x_{2i-1}$ with odd index.

Forgetting the point $0 \in \Sigma_{0, 2n}$, 
we think of $\widetilde{a}$ and $\widetilde{b}$ as elements, say 
$\widetilde{a}^{\bullet}$ and $\widetilde{b}^{\bullet} $
of $\mathrm{MCG}(\Sigma_{0,2n}, \{\infty\})$ 
respectively. 
To find the planar $2n$-braids for  $\widetilde{a}^{\bullet}$ and $\widetilde{b}^{\bullet}$, 
we cut the equator of $\Sigma_ {0, 2n}$ at a point between the consecutive punctures 
$x_{2n}$ and $x_1$ (in the cyclic order) into an arc, and 
we regard the arc as  a segment  in the plane  
containing the punctures $x_1, \dots, x_{2n}$ in this order, 
see Figure \ref{fig_equator}(1)(2). 
Then from the actions of $\widetilde{a}^{\bullet}$ and $\widetilde{b}^{\bullet}$ on $2n$ punctures in the plane, 
one sees that $2n$-braids 
corresponding to $\widetilde{a}^{\bullet}$, $\widetilde{b}^{\bullet}$ 
are given by $u_n$, $v_n \in B_{2n}$ respectively. 
See Figure \ref{fig_equator}(3). 
(Although we do not need braid representatives corresponding to $\widetilde{a}$ and $\widetilde{b}$ in the proof of Step 1, 
Figure \ref{fig_11-2-2}(1) and (3) illustrate these representatives for $\widetilde{a}$ and $\widetilde{b}$ 
respectively in case $n=4$.)

We define 
$$\phi_p=(\widetilde{a})^p (\widetilde{b})^p \in \mathrm{MCG}(\Sigma_{0,2n}, \{0\}, \{\infty\}) 
< \mathrm{MCG}(\Sigma_{0,2n+1}, \{\infty\}) .$$
It follows that 
$\phi_p$ is a lift of $a^p b^p = \sigma_1^{2p} \sigma_2^{-2p}  \in \mathrm{MCG(\Sigma_{0,2}, \{\underline{0}\}, \{\underline{\infty}\})}$. 
Recall that $a^p b^p $ is  a pseudo-Anosov mapping class  with the stretch factor $(\mathfrak{s}_{2p})^2$, 
see Example \ref{ex_metallic-3braid}. 
Since $\phi_p$ is a lift of $a^p b^p$, 
$\phi_p$ is also pseudo-Anosov with the same stretch factor as $a^p b^p$. 
Hence $\lambda(\phi_p) = (\mathfrak{s}_p)^2$. 

Forgetting the point $0 \in \Sigma_{0, 2n}$, we obtain 
$\phi_p^{\bullet} \in \mathrm{MCG}(\Sigma_{0, 2n}, \{\infty\})$ 
from $\phi_p = (\widetilde{a})^p (\widetilde{b})^p$. 
Note that 
$\phi_p^{\bullet} = (\widetilde{a}^{\bullet})^p (\widetilde{b}^{\bullet})^p $ 
is a mapping class corresponding to the braid $ \beta_{n,p}= u_{n}^p v_{n}^p$. 
\medskip

\noindent
{\bf Claim.} 
The stable/unstable foliation $\mathcal{F}^{+/-}$ of $\phi_p$ is not $1$-pronged at $0 \in \Sigma_{0, 2n}$. 
\medskip

For the proof of Step 1,
it is enough  to prove  Claim. 
The reason is that 
if $\mathcal{F}^{+/-}$ is not $1$-pronged at the point $0 \in \Sigma_ {0,2n}$, 
then the same singular foliations $\mathcal{F}^{+}$ and $\mathcal{F}^{-}$ are still invariant  foliations for $\phi_p^{\bullet}$, 
see Remark \ref{rem_singularities}. 
This implies that 
$\phi_p^{\bullet}$ (and hence the braid $\beta_{n,p}$) is pseudo-Anosov 
with the same stretch factor $(\mathfrak{s}_p)^2$ as   $\phi_p$, 
i.e., 
$$\lambda(\beta_{n,p})\Bigl(= \lambda(\phi_p^{\bullet})\Bigr)= \lambda(\phi_p) =  (\mathfrak{s}_{2p})^2.$$

\noindent
{\it Proof of Claim.} 
Let us consider the stable/unstable foliation $\underline{\mathcal{F}}^+$ and $\underline{\mathcal{F}}^-$ 
for the pseudo-Anosov element $a^p b^p$. 
Then 
$\underline{\mathcal{F}}^{+/-}$ has $1$-pronged singularities at each of the two punctures of $\Sigma_{0,2}$ 
and at each of the two points 
$\underline{0}$ and $\underline{\infty}$. 
Let $\mathcal{F}^{+}$ and $\mathcal{F}^-$ denote lifts of $\underline{\mathcal{F}}^+$ and $\underline{\mathcal{F}}^-$ respectively.  
It follows that $\mathcal{F}^{+/-}$ is the stable/unstable foliation for $\phi_p$, and 
$\mathcal{F}^{+/-}$ has a $1$-pronged singularity at each of the $2n$ punctures and 
$\mathcal{F}^{+/-}$ has $n$-pronged singularities $(n \ge 2)$ at  
the points $0$ and $\infty$ in $\Sigma_{0, 2n}$. 
In particular $\mathcal{F}^{+/-}$ is not  $1$-pronged at $0 \in \Sigma_{0, 2n}$. 
This completes the proof of Claim. 
\medskip

\noindent
By Claim,  we finished the proof of Step 1. 
\medskip

Recall that $y_{n,p} = b({\bm x}_{n,p}(t), [0, 2 \bar{T} ])$ and $\bar{T}=\frac{dT}{2n}$. 
\medskip

\noindent
{\bf Step 2.} 
$\beta_{n,p} \in \langle y_{n,p} \rangle$. 
In particular $(\beta_{n,p})^{\frac{n}{d} } \in X_{n,p}= \big\langle (y_{n,p})^{\frac{n}{d} } \big\rangle$. 
\medskip

\noindent
Proof of Step 2. 
Let us consider the shape curve ${\bm u}(t)$ ($t \in [0, 2 \bar{T}]$)
for the solution ${\bm x}_{n,p}(t)$. 
By the arguments in Section \ref{subsection_shapesphare}, 
the shape curve ${\bm u}(t)$ ($t \in [0, 2 \bar{T}]$)  
satisfies the following properties.

\begin{enumerate}
\item[(s1)] 
${\bm u}(0) \in \mathrm{B}_{-1}$, ${\bm u}(\bar{T}) \in \mathrm{B}_{2p-1}$ and ${\bm u}(2 \bar{T}) \in \mathrm{B}_{4p-1}$. 

\item[(s2)] 
$u_1(t)>0$ for $0 < t < \bar{T}$.

\item[(s3)] 
 $u_1(t)<0$ 
for $ \bar{T} < t < 2 \bar{T} $.

\item[(s4)] 
$x_i(2 \bar{T} ) = x_{\sigma(h_{n,p}^{-1})(i)} (0)$ for $i= 1, \dots, 2n$. 
\end{enumerate}

Recall that 
  $n$ bodies with odd indices  and $n$ bodies with even  indices rotate in mutually opposite directions. 
The above (s1) (${\bm u}(0) \in \mathrm{B}_{-1}$, ${\bm u}(\bar{T}) \in \mathrm{B}_{2p-1}$) 
and  (s2) tell us that 
each of bodies $x_{2i}(t)$'s ($i= 1, \dots, n$) 
with even indices is passing through in front of bodies  
 with odd indices  (in the time interval $(0, \bar{T})$) 
from the view of the origin $0 \in {\Bbb R}^2$. 
 Similarly (s1) (${\bm u}(\bar{T}) \in \mathrm{B}_{2p-1}$, ${\bm u}(2 \bar{T}) \in \mathrm{B}_{4p-1}$) and  (s3) imply that 
 each of bodies $x_{2i-1}(t)$'s ($i = 1, \dots, n$) 
 with odd indices is passing through in front of the bodies  
with even indices  (in the time interval $(\bar{T}, 2 \bar{T}$)). 
 These properties connect up   $(\widetilde{b})^p \in \mathrm{MCG}(\Sigma_{0, 2n}, \{0\}, \{\infty\})$ 
 with ${\bm u}(t)$  for $t \in [0, \bar{T} ]$ 
 (resp. $(\widetilde{a})^p \in \mathrm{MCG}(\Sigma_{0, 2n}, \{0\}, \{\infty\})$ 
 with ${\bm u}(t)$  for $t \in  [\bar{T}, 2 \bar{T}]$), 
  see Figure \ref{fig_11-2-2}(3)(4). 
Recall that 
 the permutation $\hat{\sigma}(\beta_{n,p})$ of the braid $\beta_{n,p}= u_n^p v_n^p$ 
 coincides with  $ \sigma(h_{n,p}^{-1})$, 
 see (\ref{equation_permutation-b}). 
 Putting these properties together with  (s4), we conclude that  the $2n$-braid $u_n^p v_n^p (= \beta_{n,p}) $ 
is a representative of the braid type $\langle  y_{n,p} \rangle$ of $y_{n,p}= b({\bm x}_{n,p}(t), [0, 2 \bar{T}])$. 
This completes the proof of Step 2. 
 \medskip
 
 By Steps 1 and 2, we have finished the proof of Theorem \ref{thm_braidtype}. 
\end{proof}

 We end this section with an example. 

\begin{ex}
\label{example_occasionally}
Corollary \ref{cor_prime} and Table \ref{table_entropy} may suggest that 
$\lambda(X_{n,p})$ does not coincide with $\lambda(X_{n,p'})$ for different pairs 
$(n,p) \ne (n,p')$. 
 However,  
 the stretch factors happen to coincide for different pairs  occasionally: 
 The $k$th metallic ratio $\mathfrak{s}_k$ has a   formula 
$(\mathfrak{s}_k)^3 = \mathfrak{s}_{k^3+ 3k}$ for each $k \in {\Bbb N}$. 
In particular $(\mathfrak{s}_6)^3= \mathfrak{s}_{234}$ when $k=6$. 
We now claim that 
$ \lambda(X_{n,3})= \lambda(X_{n, 117})$ for  all $n= 3^2  2^k$ with $k \ge 5$. 
Then $117 \le \lfloor \frac{n}{2}  \rfloor$. 
By Theorem \ref{thm_stretch-factor}, 
we have 
$ \lambda(X_{n,3})= (\mathfrak{s}_{6})^{\frac{2n}{3}}= (\mathfrak{s}_{6})^{3 \cdot 2^{k+1}}$ and 
$\lambda(X_{n, 117}) = (\mathfrak{s}_{234})^{\frac{2n}{9}} =  (\mathfrak{s}_{234})^{2^{k+1}}$. 
By the equality $(\mathfrak{s}_6)^3= \mathfrak{s}_{234}$, we have 
$$ \lambda(X_{n,3})=  ((\mathfrak{s}_6)^3)^{2^{k+1}}= (\mathfrak{s}_{234})^{2^{k+1}} = \lambda(X_{n, 117}).$$
\end{ex}


\section{New numerical periodic solutions of the $2n$-body problem}
\label{section_numerical}

We numerically found the periodic solutions $\bm{x}_{n,p}(t)$ for $p= 1, \dots, \lfloor \frac{n}{2}  \rfloor$ in Figure  \ref{fig_orbit}.
In order to obtain those, 
we consider the Fourier series of the solutions and compute the Fourier coefficient by using the steepest descent method.
Though the existence of the periodic orbits theoretically guarantees for $p = 1, \dots,  \lfloor \frac{n}{2}  \rfloor$, 
new numerical solutions are  obtained for 
several pairs with $(n,p)$ with $  \lfloor \frac{n}{2}  \rfloor < p< n$. 
See Figure \ref{fig_num_orb}.
Then it is natural to ask the following question. 

\begin{ques}
\label{ques_last}
For $n \ge 2$ and $ \lfloor \frac{n}{2}  \rfloor < p< n$, 
does there exist a periodic solution ${\bm x}_{n,p}(t)$ of the planar $2n$-body problem 
whose braid type $X_{n,p}$ is given by the braid $(\beta_ {n,p})^{\frac{n}{d}}$ with 
 $d= \gcd(n,p)$?  
\end{ques}
If the answer of Question \ref{ques_last} is positive, then 
Theorem \ref{thm_stretch-factor} is extended  to  some pairs $(n,p)$ with  $\lfloor \frac{n}{2}  \rfloor < p< n$, i.e. 
if the answer of Question \ref{ques_last} is positive, then 
the braid type $X_ {n,p}$ of the periodic solution ${\bm x}_{n,p}(t)$ of the planar $2n$-body problem 
is pseudo-Anosov with the stretch factor 
$(\mathfrak{s}_{2p})^{\frac{2n}{d}}$. 
See also Step 1 of the proof of Theorem \ref{thm_braidtype}.

\begin{figure}[h]
\begin{center}

\includegraphics[width=1.8cm]{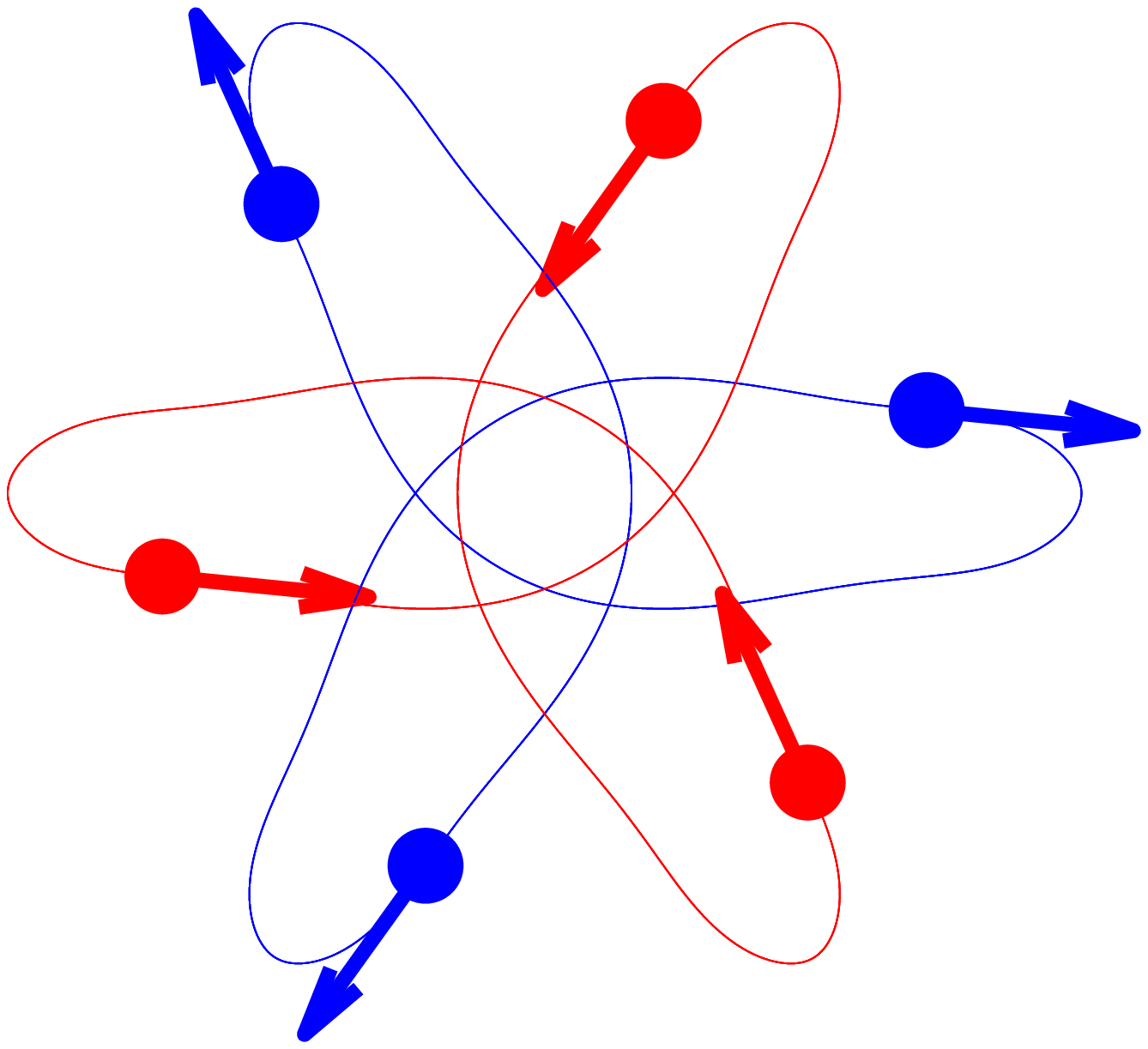}\hspace{0.5cm}
\includegraphics[width=1.8cm]{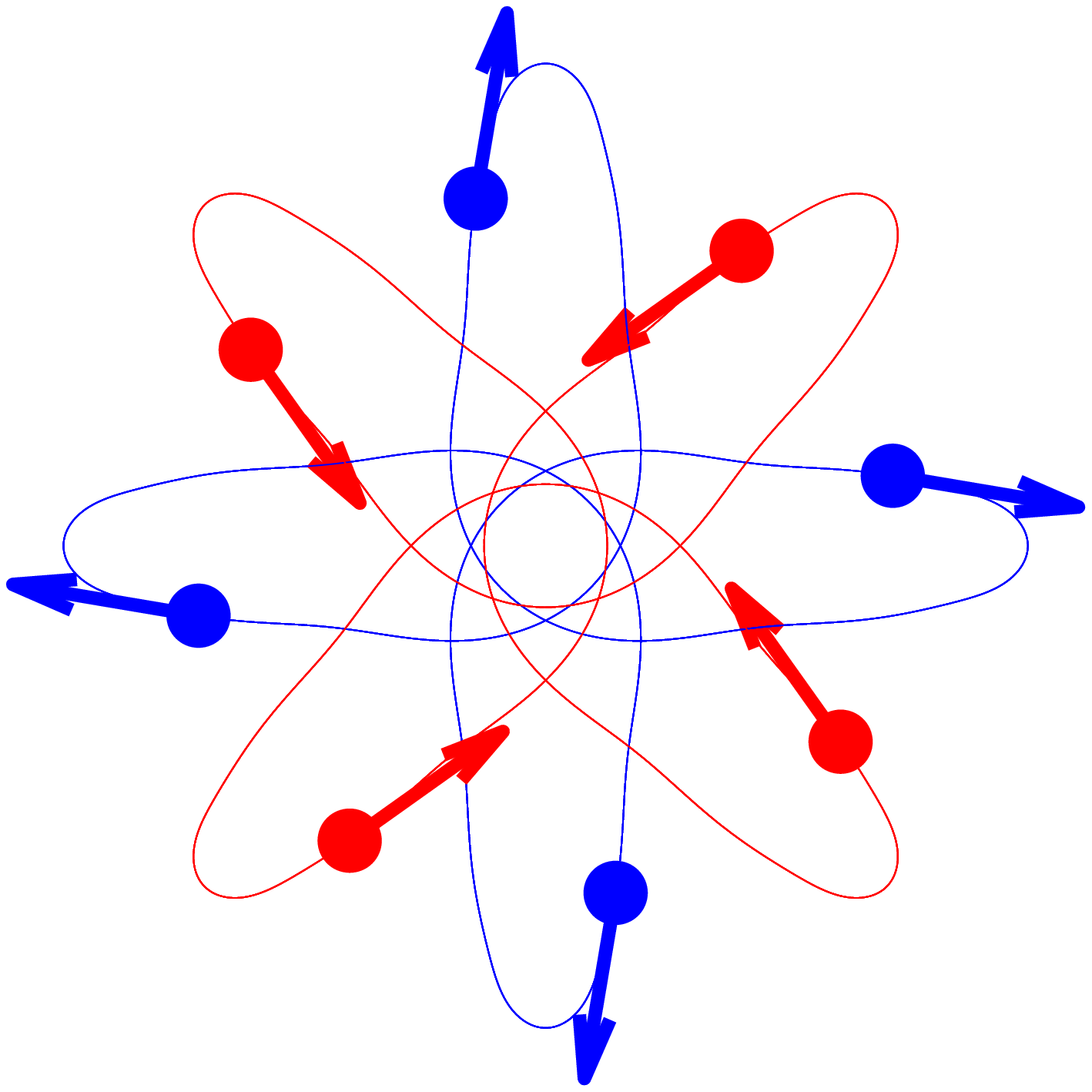}\hspace{0.5cm}
\includegraphics[width=1.8cm]{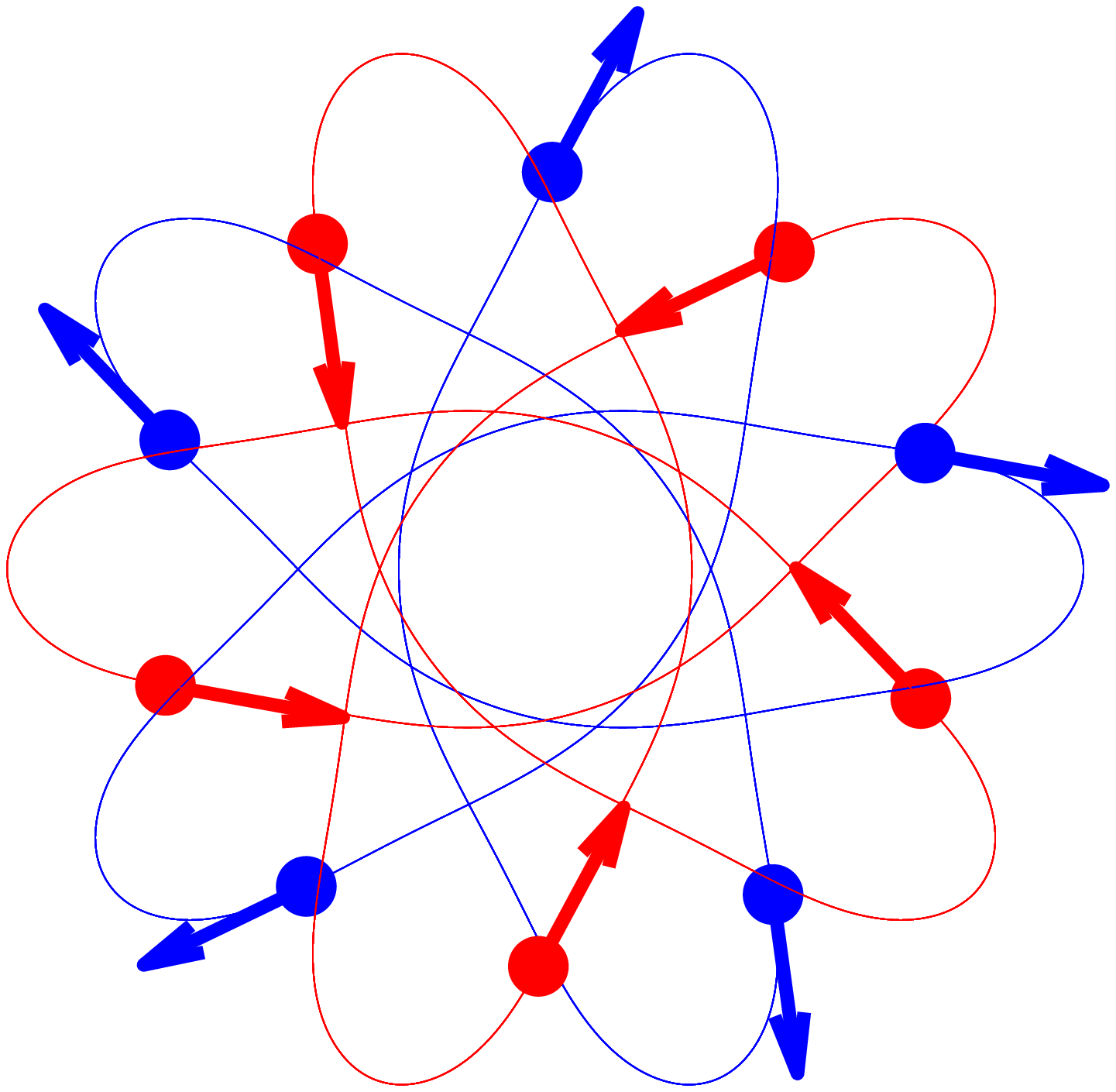}\hspace{0.5cm}
\includegraphics[width=1.8cm]{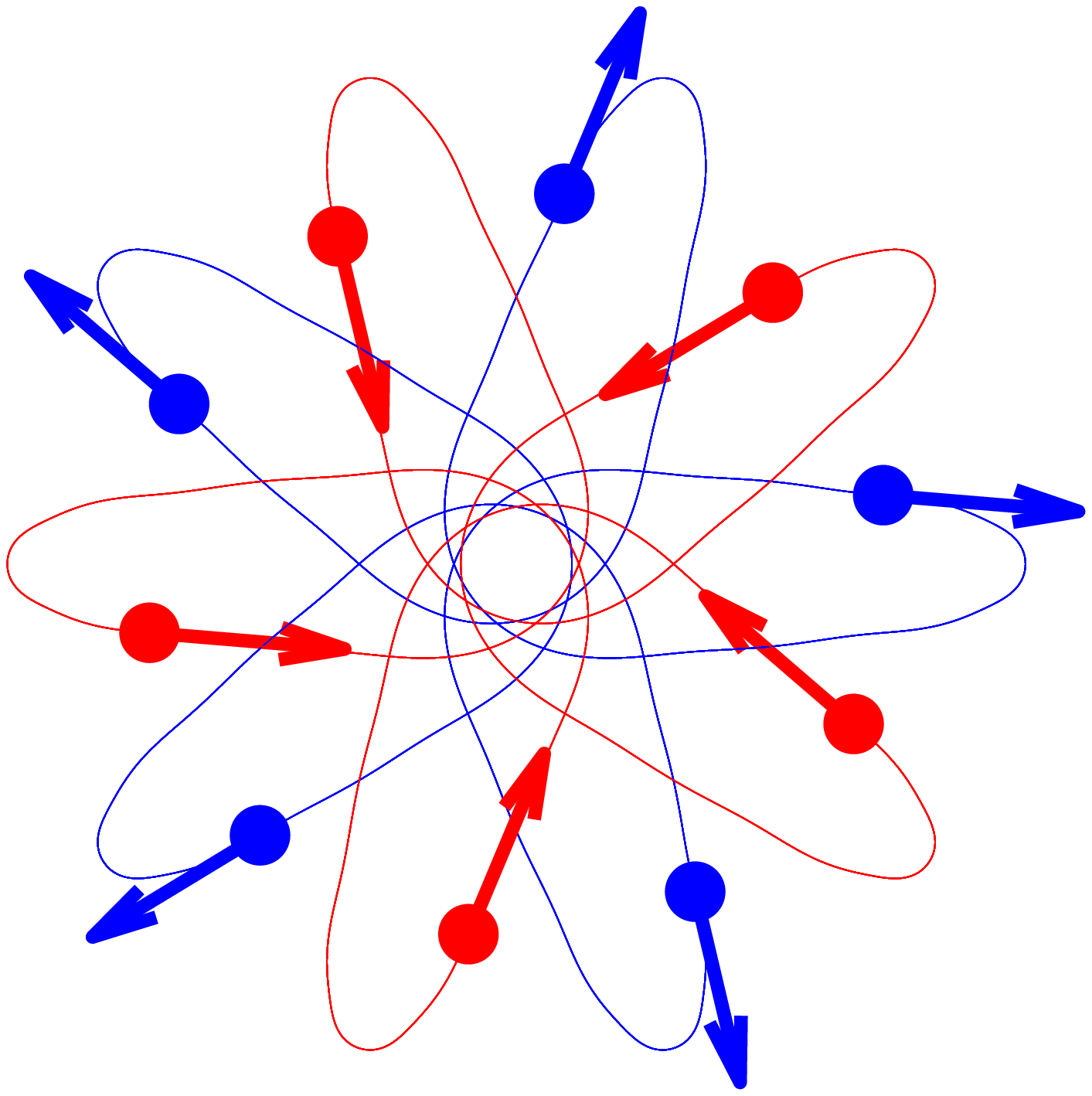}\hspace{0.5cm}
$t=\frac{dT}{n}$

\includegraphics[width=1.8cm]{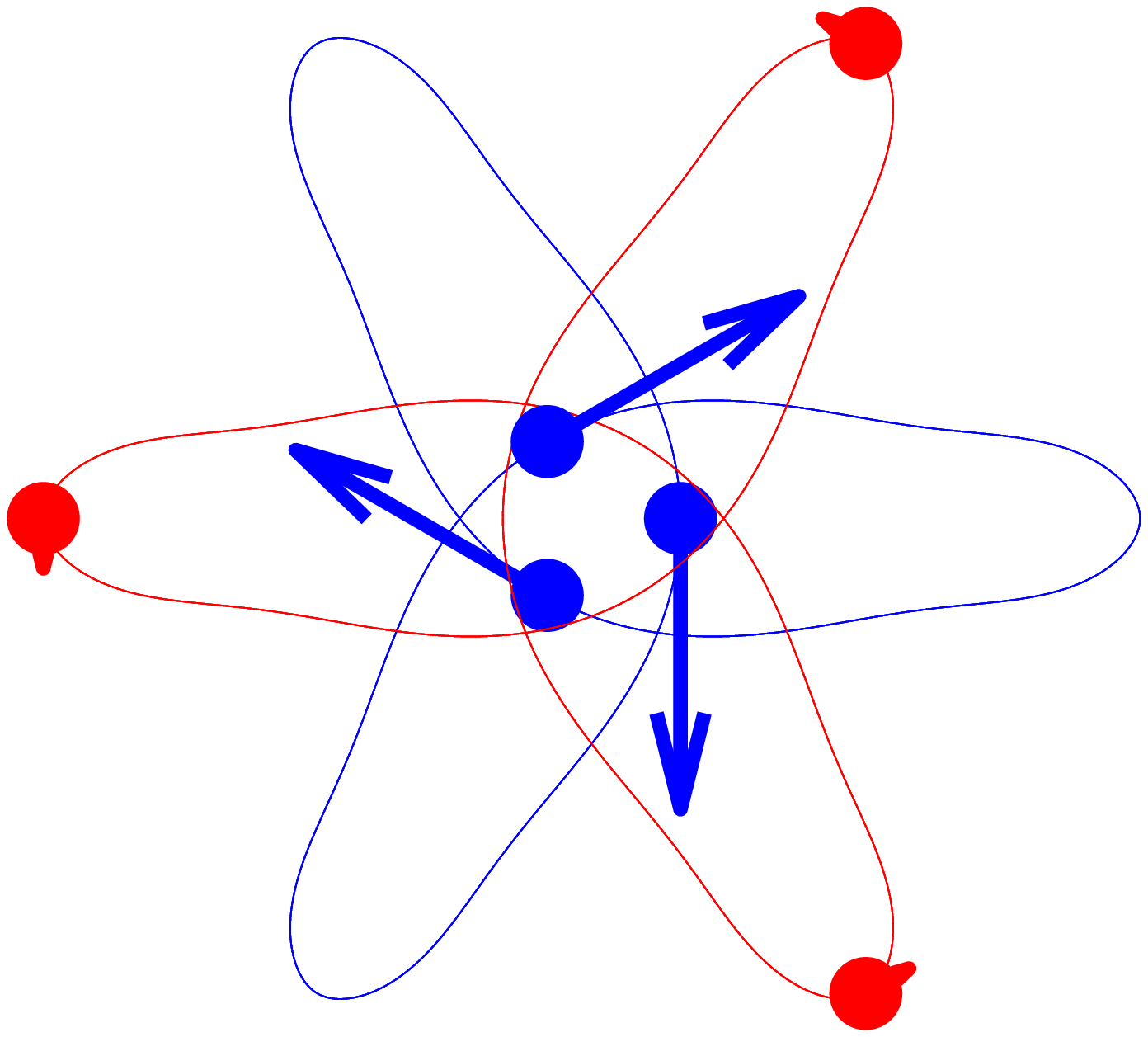}\hspace{0.5cm}
\includegraphics[width=1.8cm]{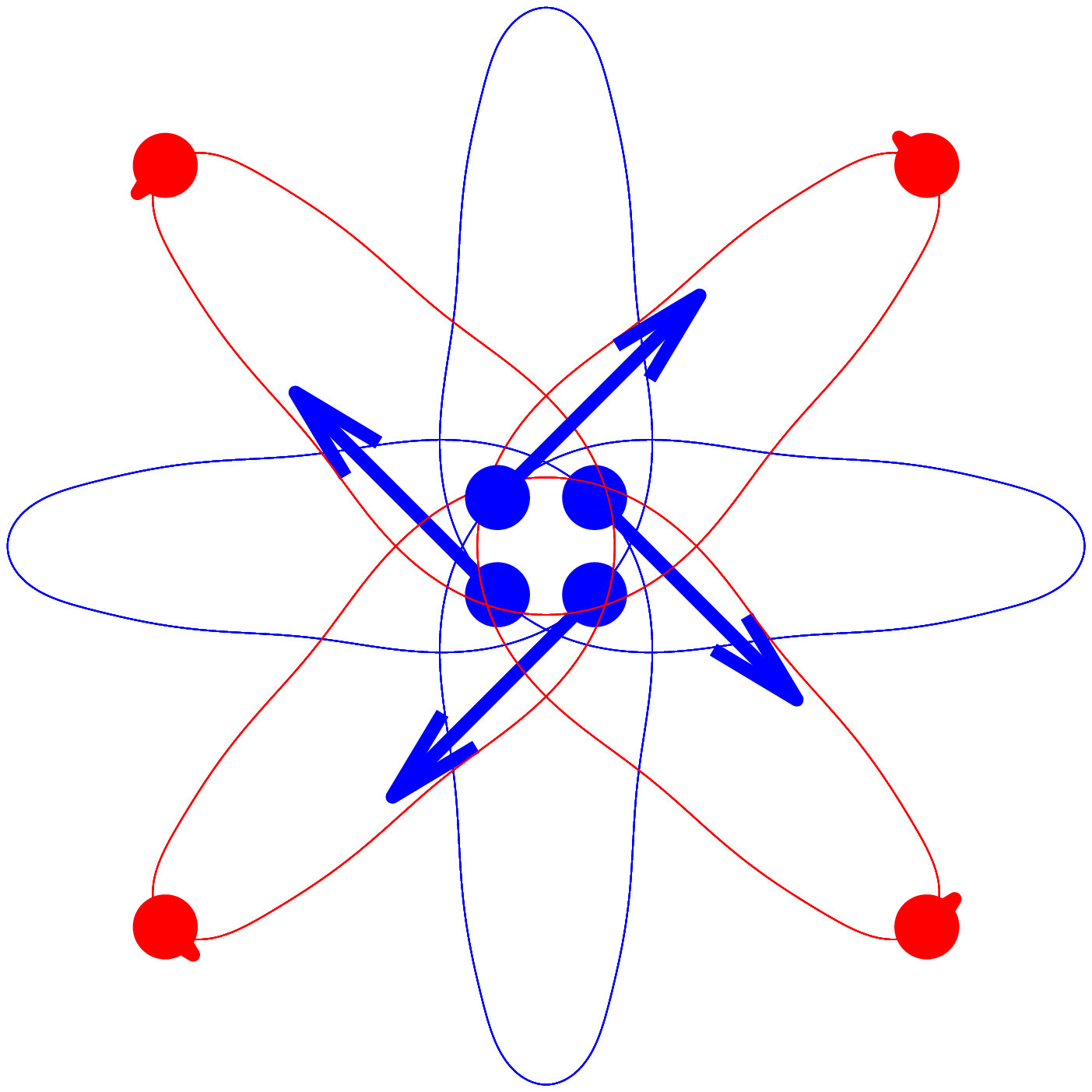}\hspace{0.5cm}
\includegraphics[width=1.8cm]{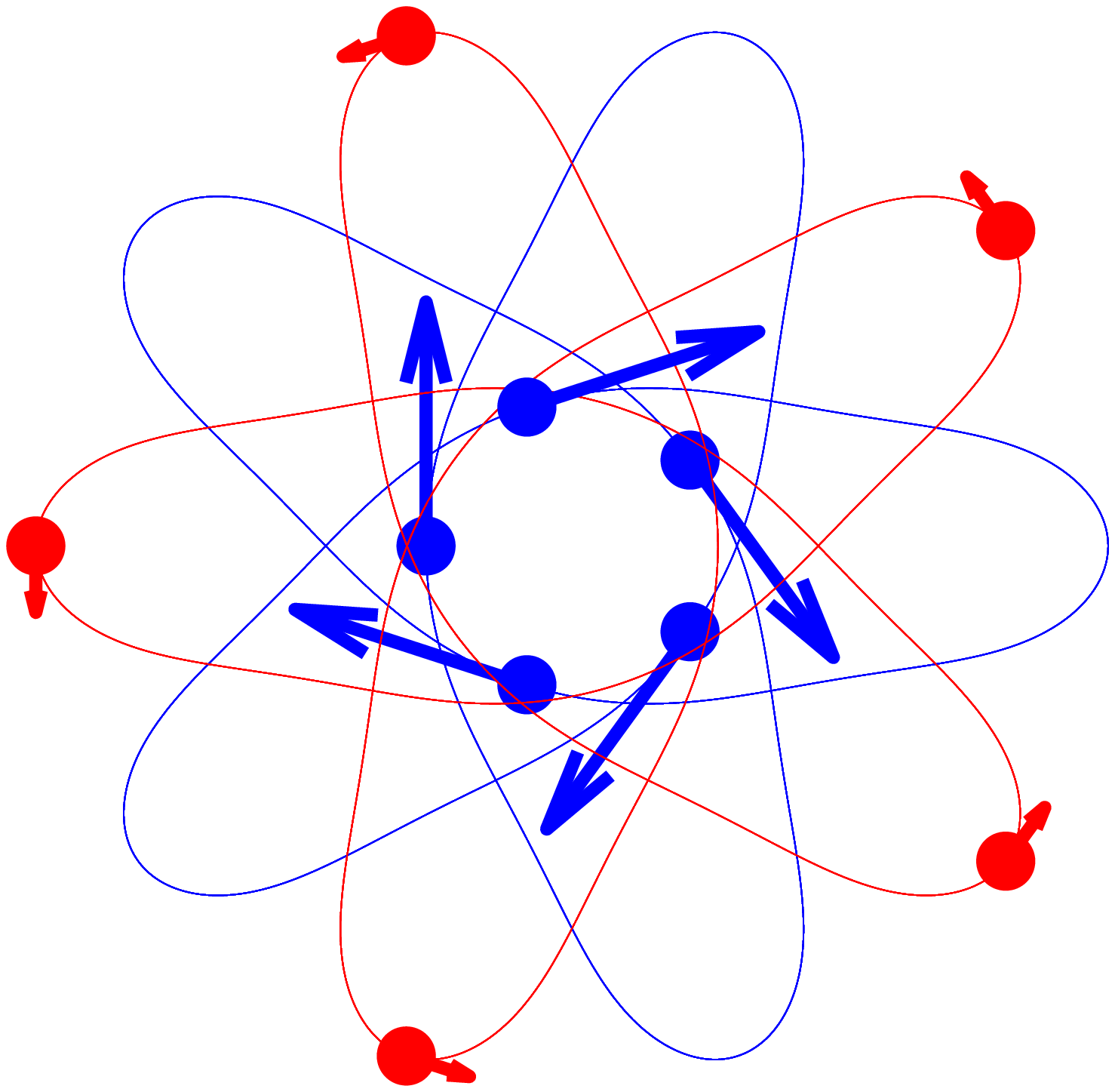}\hspace{0.5cm}
\includegraphics[width=1.8cm]{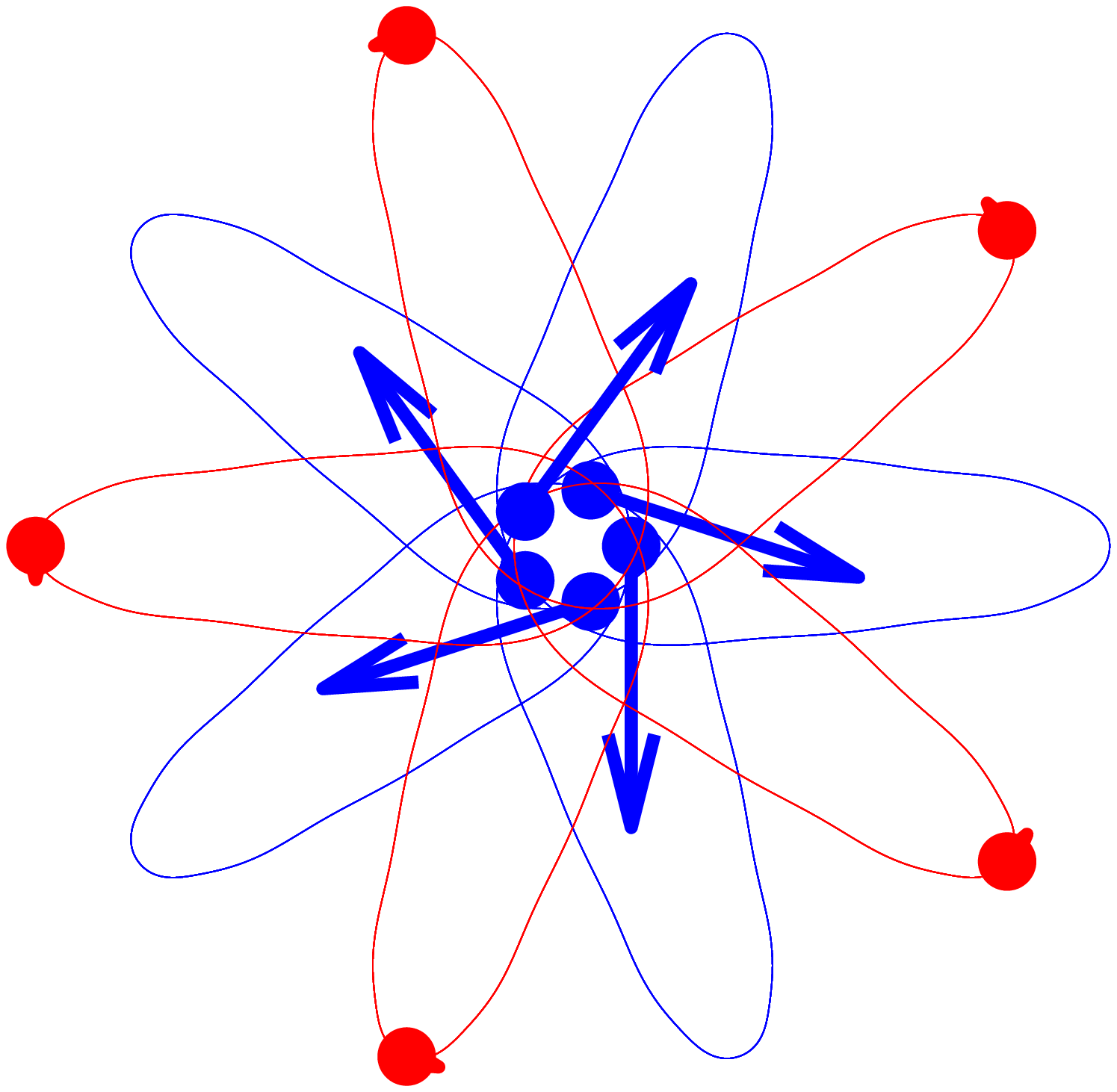}\hspace{0.5cm}
$t=\frac{3dT}{4n}$

\includegraphics[width=1.8cm]{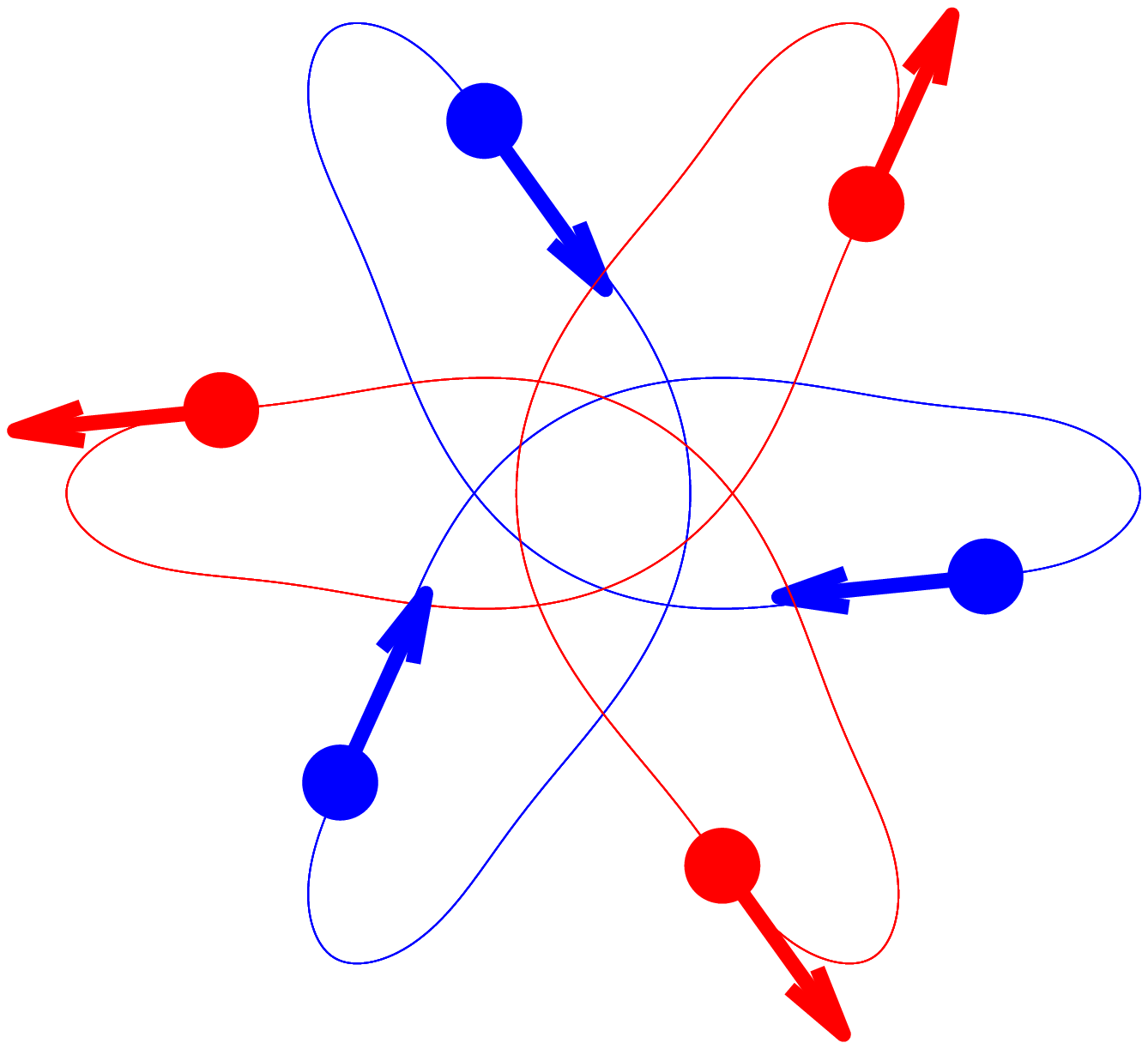}\hspace{0.5cm}
\includegraphics[width=1.8cm]{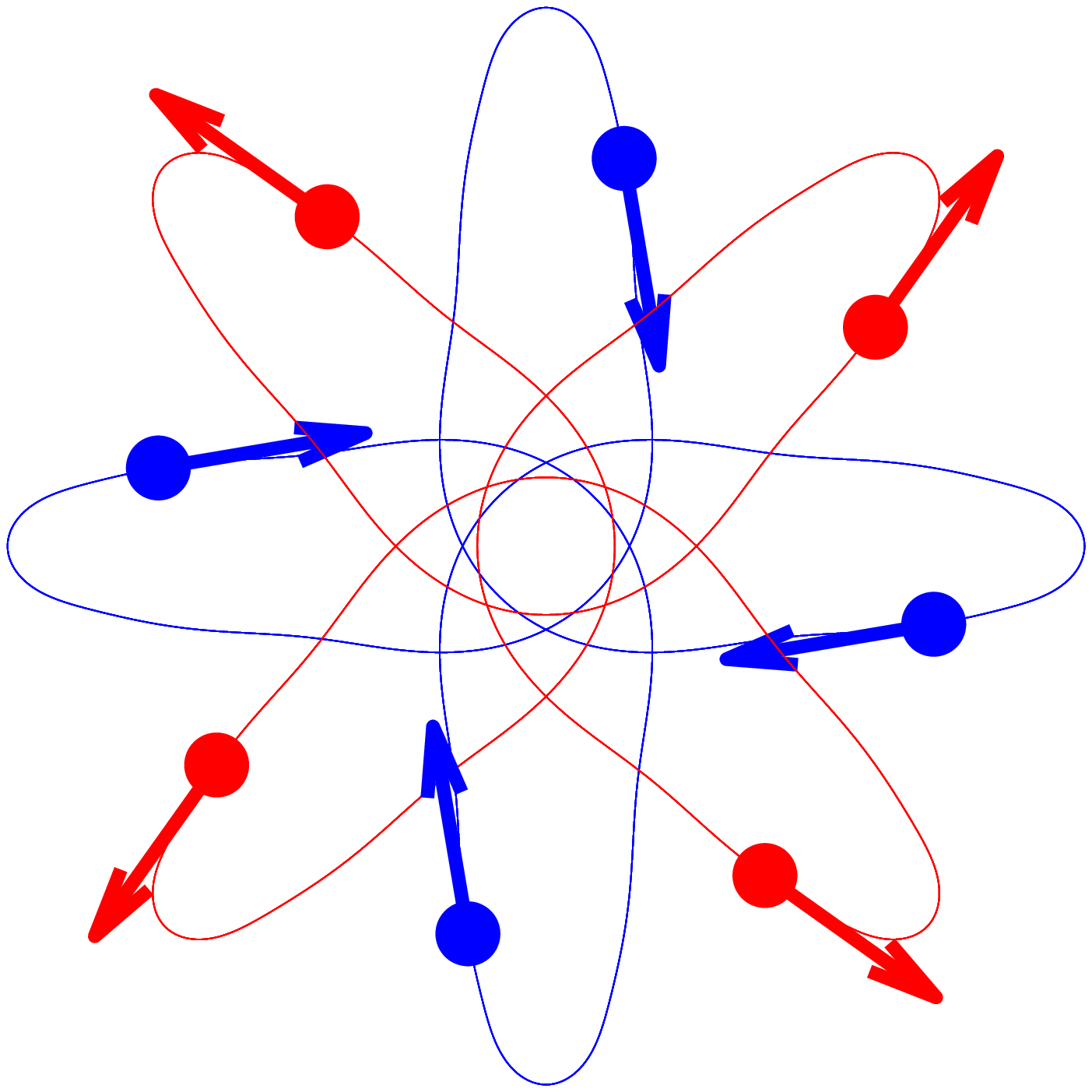}\hspace{0.5cm}
\includegraphics[width=1.8cm]{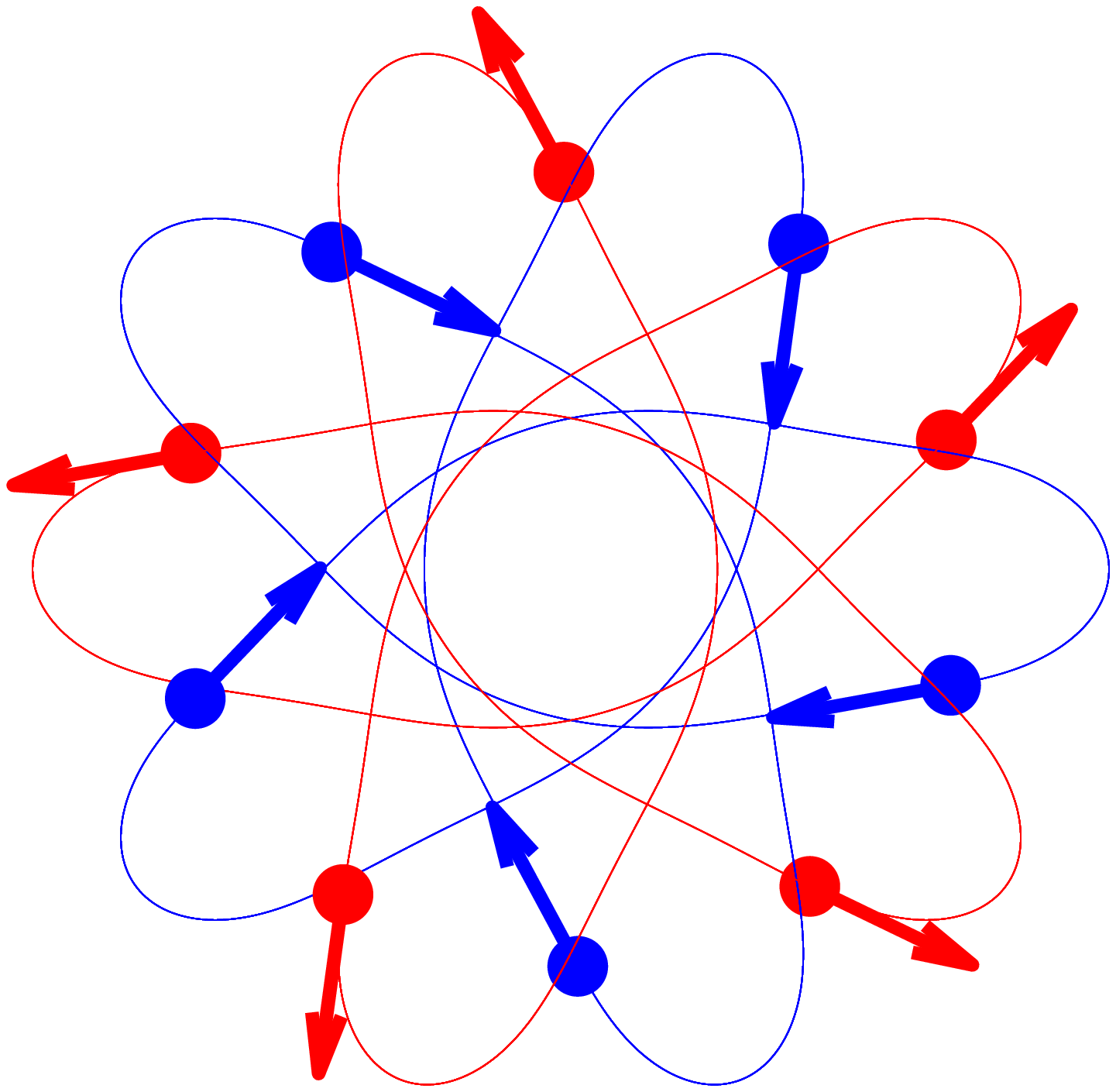}\hspace{0.5cm}
\includegraphics[width=1.8cm]{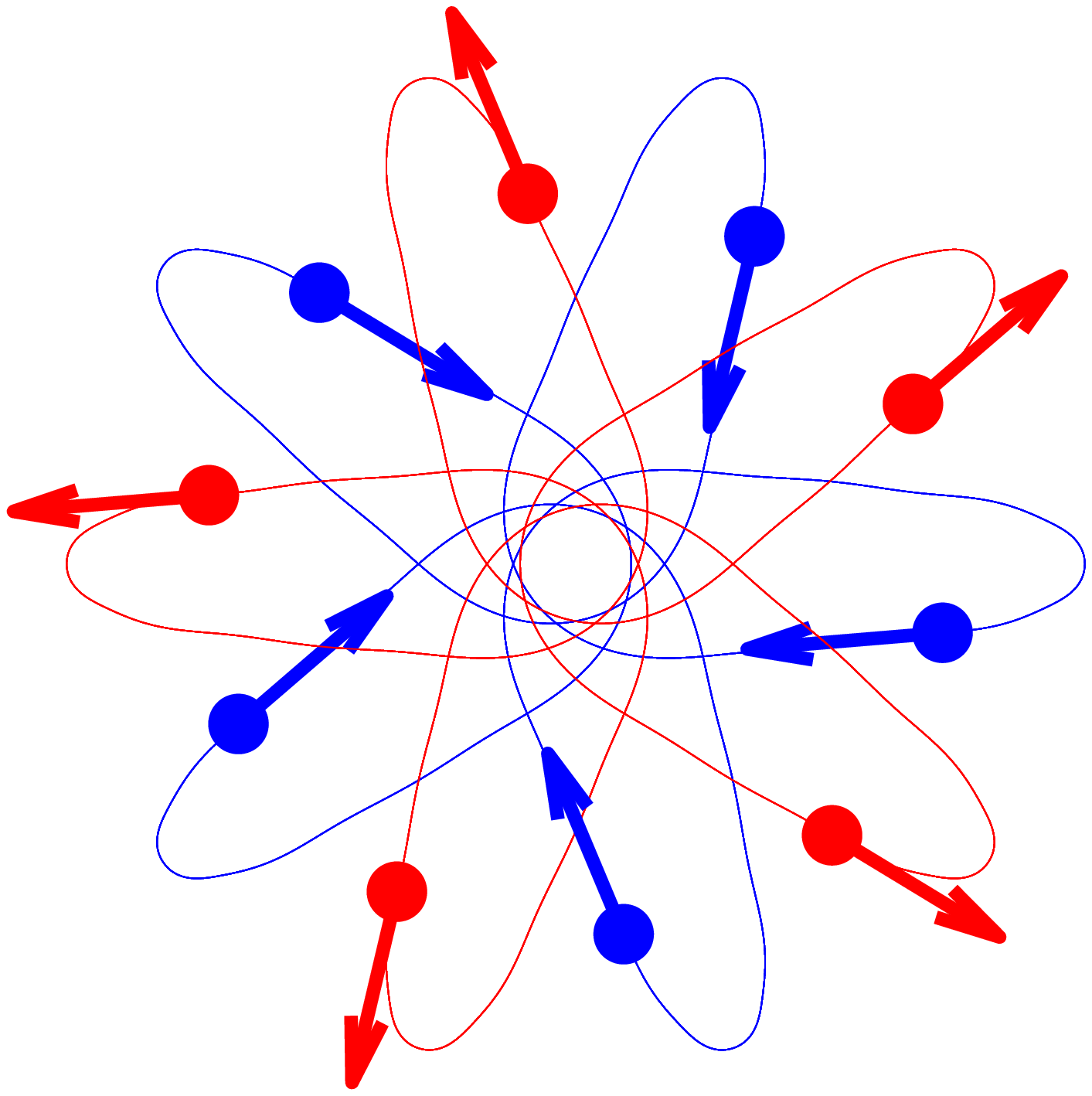}\hspace{0.5cm}
$t=\frac{dT}{2n}$

\includegraphics[width=1.8cm]{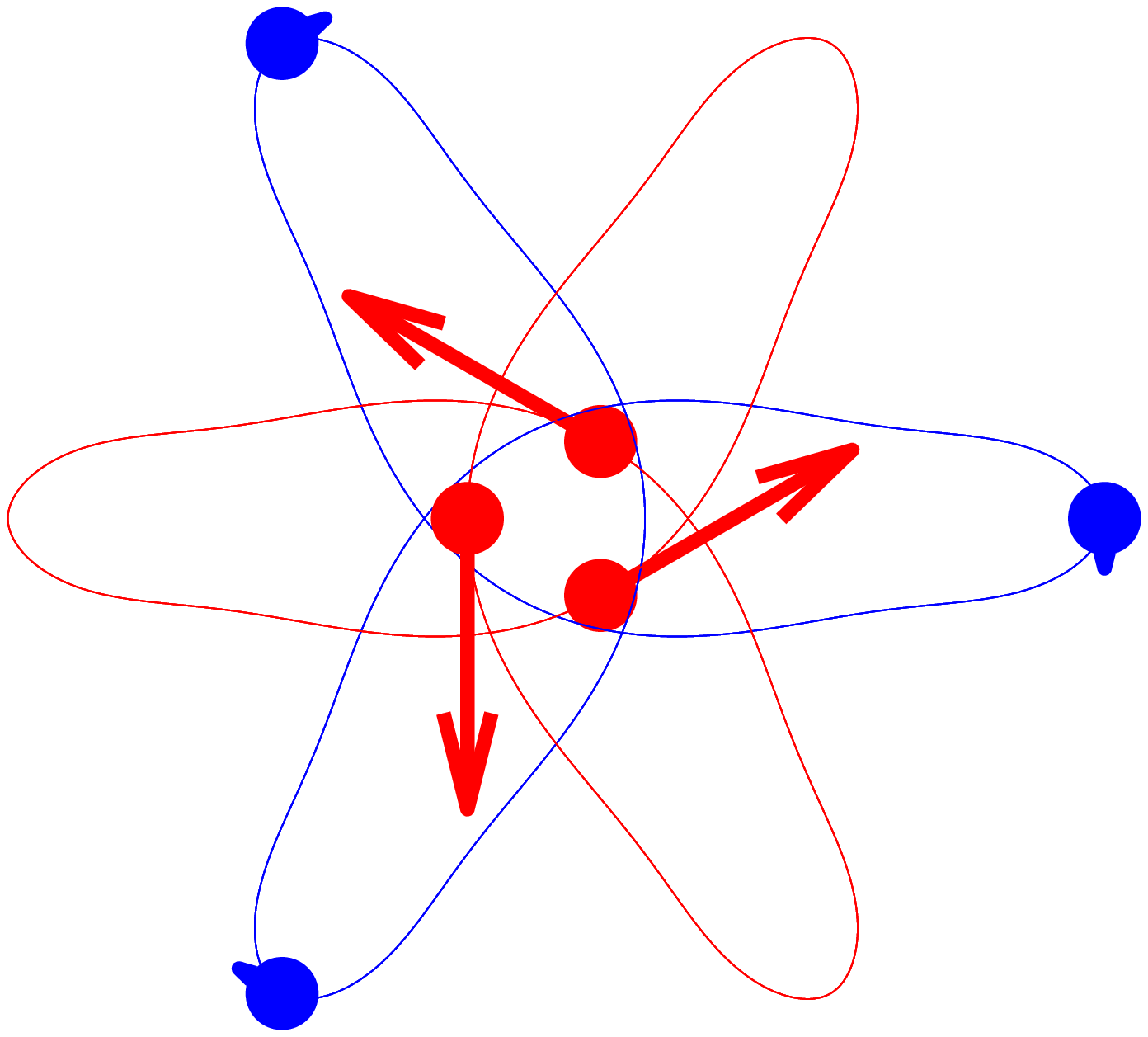}\hspace{0.5cm}
\includegraphics[width=1.8cm]{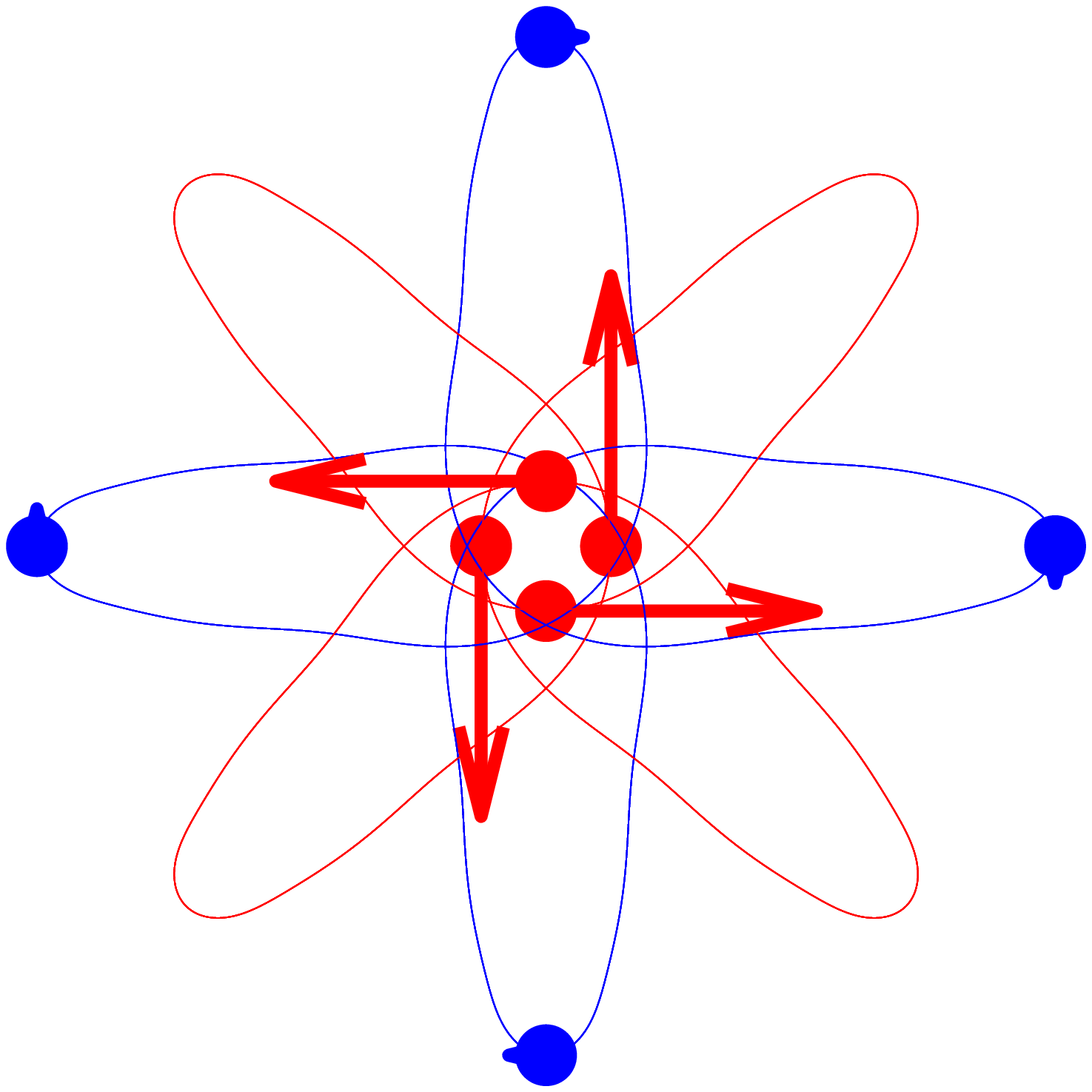}\hspace{0.5cm}
\includegraphics[width=1.8cm]{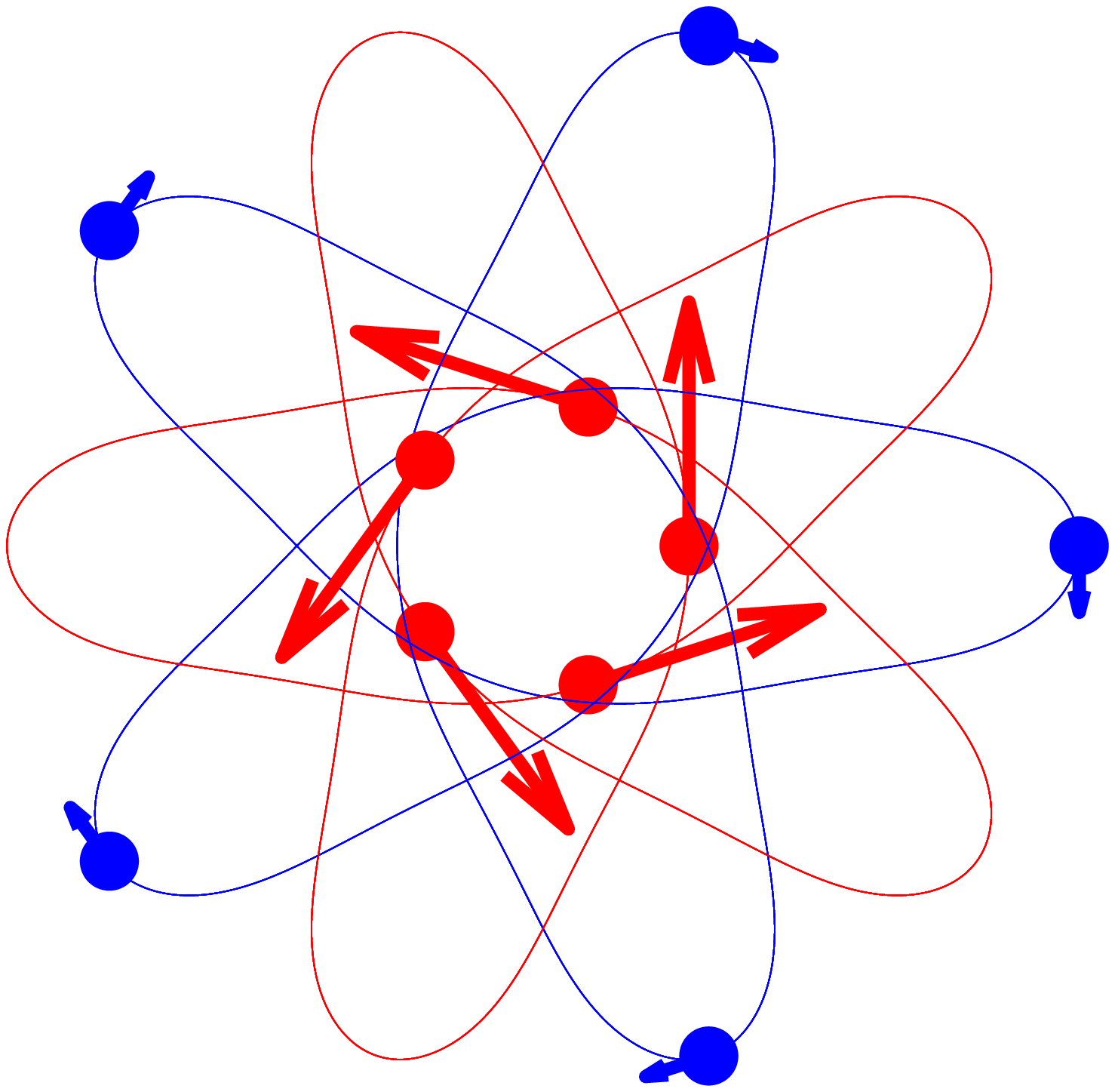}\hspace{0.5cm}
\includegraphics[width=1.8cm]{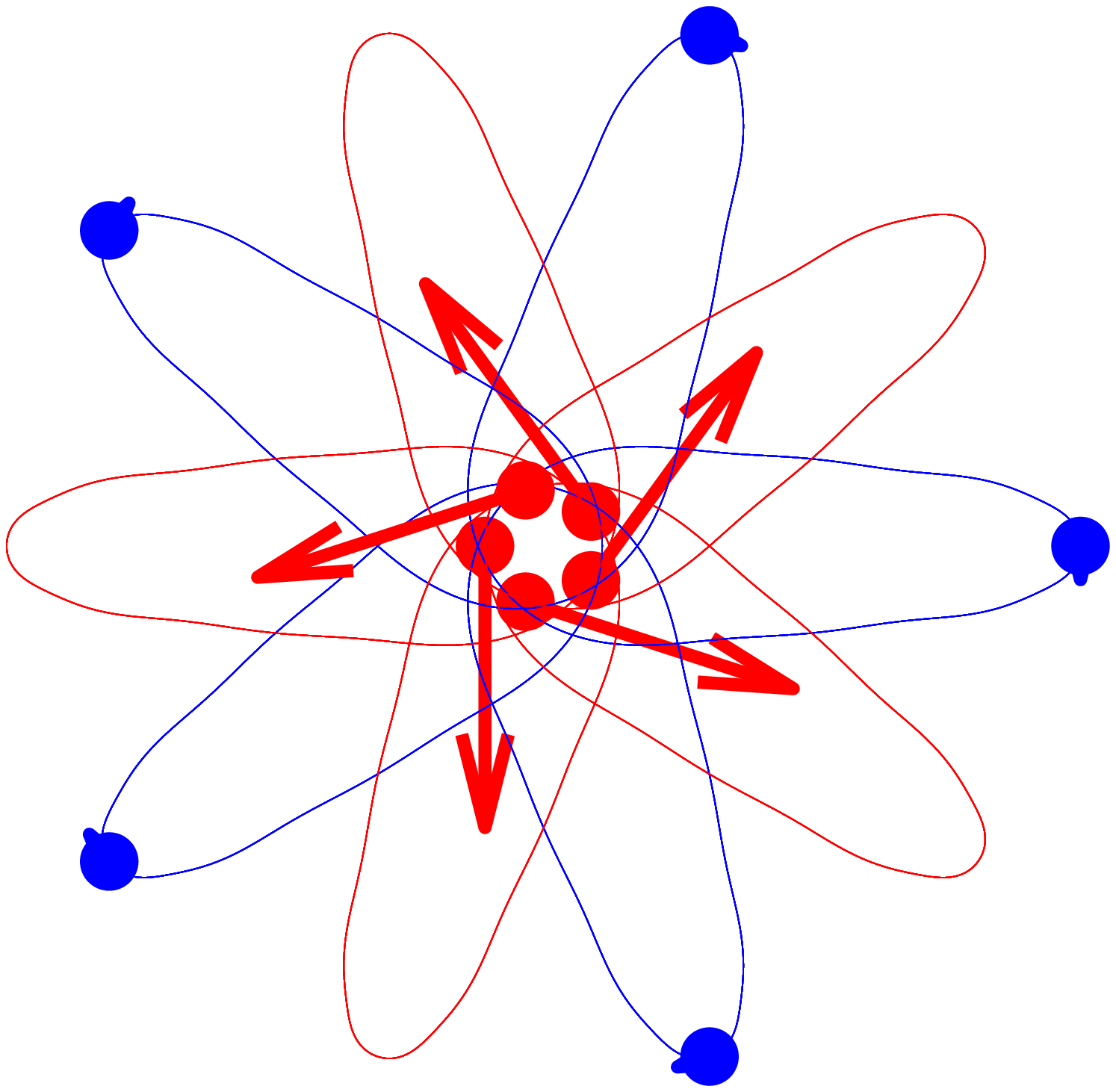}\hspace{0.5cm}
$t=\frac{dT}{4n}$

\includegraphics[width=1.8cm]{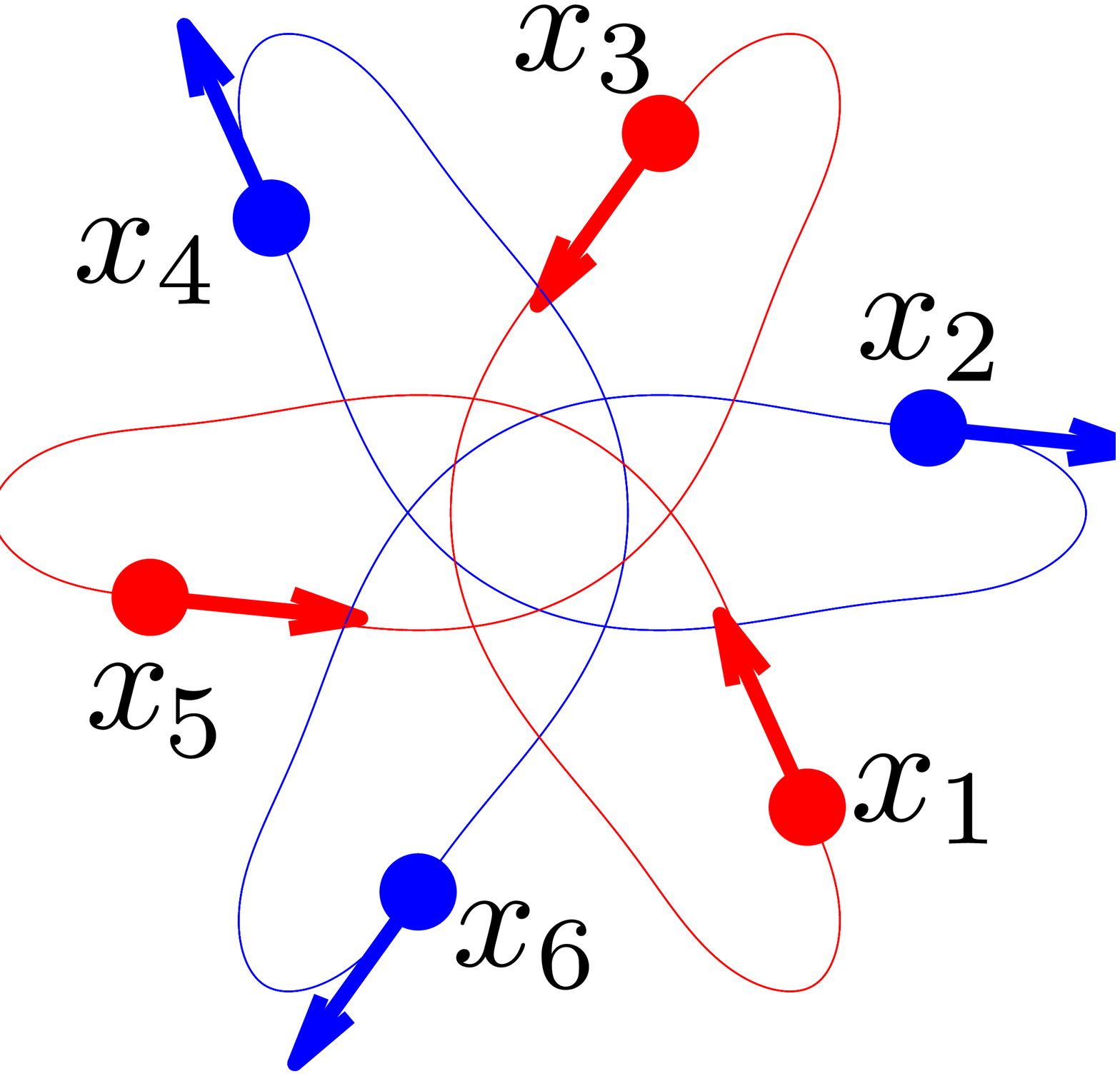}\hspace{0.5cm}
\includegraphics[width=1.8cm]{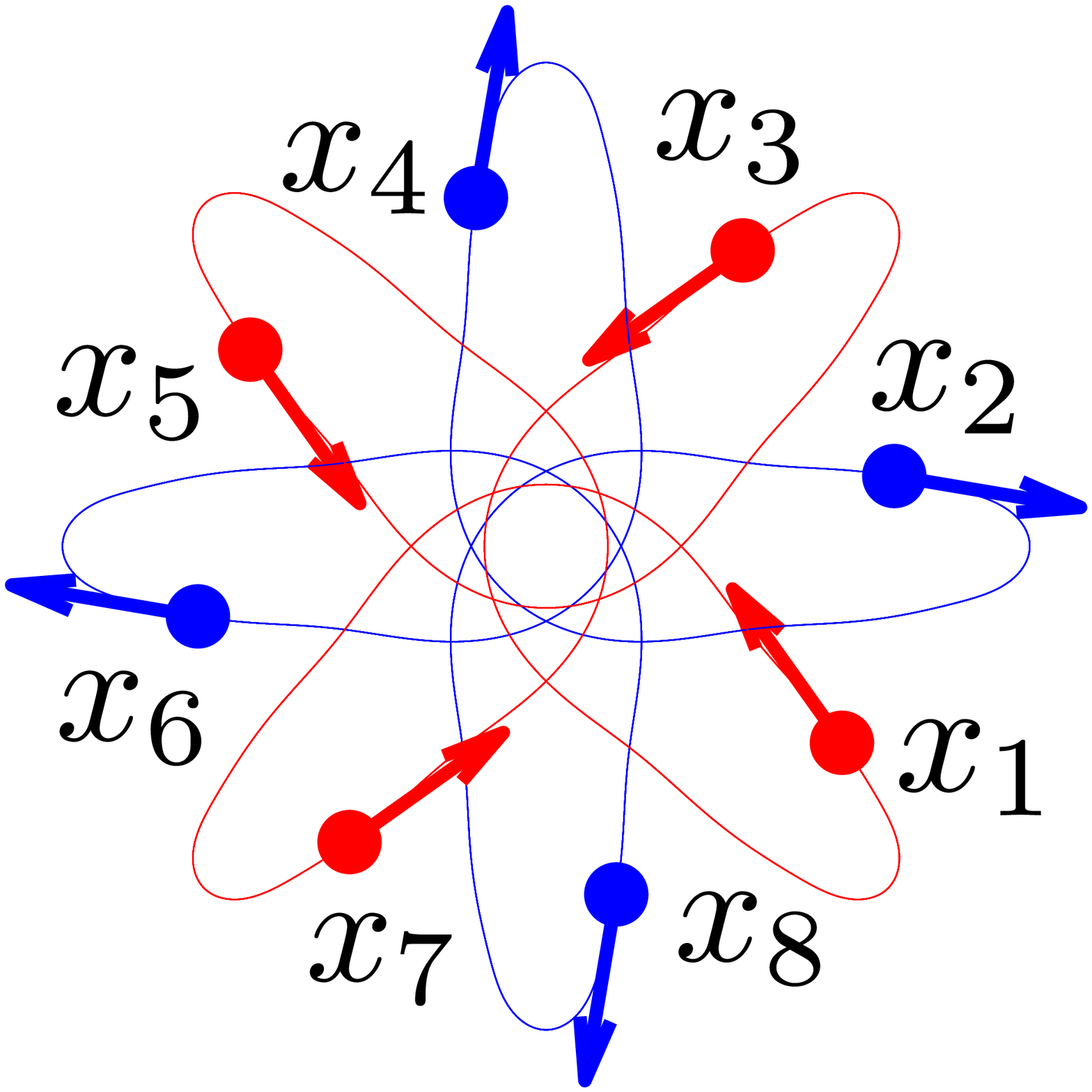}\hspace{0.5cm}
\includegraphics[width=1.8cm]{BRn=5_p=3_1.eps}\hspace{0.5cm}
\includegraphics[width=1.8cm]{BRn=5_p=4_1.eps}\hspace{0.5cm}
$t=0$

(1) \hspace{1.7cm}  (2) \hspace{1.7cm}  (3) \hspace{1.7cm}  (4) \hspace{2.0cm}

\caption{
(1) ${\bm x}_{3,2}(t)$. 
(2) ${\bm x}_{4,3}(t)$. 
(3) ${\bm x}_{5,3}(t)$. 
(4) ${\bm x}_{5,4}(t)$.
}
\label{fig_num_orb}
\end{center}

\end{figure}

\bibliographystyle{alpha} 
\bibliography{2n-braids}

\end{document}